\documentclass{article}
\usepackage[utf8]{inputenc}
\usepackage[T1]{fontenc}
\usepackage{amsmath}
\usepackage{amsfonts}
\usepackage{amssymb}
\usepackage{tikz}
\usepackage{xcolor}
\usepackage[english]{babel}
\usepackage{wasysym}
\usepackage{MnSymbol}
\usepackage{enumitem}
\usepackage{esint}
\usepackage{float}
\usepackage{graphicx}
\usepackage{epstopdf}

\usepackage{xcolor}
\colorlet{darkgreen}{black!30!green}
\definecolor{medblue}{rgb}{0.0,0.535,1.0}
\definecolor{lightblue}{rgb}{0.225,1.0,0.743}
\definecolor{lightgreen}{rgb}{0.743,1.0,0.225}

\newcommand*{\colorline}
[1]{$\vcenter{\hbox{\protect\tikz \protect\draw[thick,#1] (0,0)--(0.4,0);}}$}
\newcommand*{\sqsolid}
[1]{$\vcenter{\hbox{\tikz{\draw[thick,#1] (0,0)--(0.4,0) node[midway]{$\medsquare$};}}}$}
\newcommand*{\circdashed}
[1]{$\vcenter{\hbox{\tikz{\draw[thick,dashed,#1] (0,0)--(0.55,0) node[midway]{$\medcircle$};}}}$}

\usetikzlibrary{calc,arrows,decorations.pathmorphing,backgrounds,positioning,fit,petri,shapes.geometric,shapes.symbols,shapes.misc}
\colorlet{darkgreen}{black!30!green!}
\colorlet{darkred}{black!30!red!}
\colorlet{lightyellow}{yellow!20!white!}
\newtheorem{theorem}{Theorem}

\newtheorem{corollary}[theorem]{Corollary}

\newtheorem{definition}[theorem]{Definition}

\newtheorem{lemma}[theorem]{Lemma}

\newtheorem{remark}[theorem]{Remark}

\newenvironment{proof}[1][Proof]{\noindent\textbf{#1.} }{\ \rule{0.5em}{0.5em}}

\begin{document}

\title{High order Lagrange-Galerkin methods for the conservative formulation
of the advection-diffusion equation}
\author{Rodolfo Bermejo$^{a}$, Manuel Colera$^{b}$ \\
{\small (a) Dpto. Matem\'{a}tica Aplicada a la Ingenier\'{\i}a Industrial
ETSII. Universidad Polit\'ecnica de Madrid.} \\
[0pt] {\small e-mail: rodolfo.bermejo@upm.es } \\
[0pt] {\small (b) Dpto. Ingenier\'ia Energ\'etica ETSII. Universidad
Polit\'ecnica de Madrid.} \\
[0pt] {\small e-mail: m.colera@upm.es }}
\date{}
\maketitle

\begin{abstract}
We introduce in this paper the numerical analysis of high order both in time
and space Lagrange-Galerkin methods for the conservative formulation of the
advection-diffusion equation. As time discretization scheme we consider the
Backward Differentiation Formulas up to order $q=5$. The development and
analysis of the methods are performed in the framework of time evolving
finite elements presented in C. M. Elliot and T. Ranner, IMA Journal of
Numerical Analysis \textbf{41}, 1696-1845 (2021). The error estimates show
through their dependence on the parameters of the equation the existence of
different regimes in the behavior of the numerical solution; namely, in the
diffusive regime, that is, when the diffusion parameter $\mu$ is large, the
error is $O(h^{k+1}+\Delta t^{q})$, whereas in the advective regime, $\mu
\ll 1$, the convergence is $O(\min (h^{k},\frac{h^{k+1}
}{\Delta t})+\Delta t^{q})$. It is worth remarking that the error constant
does not have exponential $\mu ^{-1}$ dependence.
\end{abstract}



\section{Introduction}

We present the numerical analysis of Lagrange-Galerkin methods, introduced
in \cite{colera}, developed to integrate the conservative formulation of the
advection-diffusion equation. Specifically, let $[0,T]$, $T>0$, be a time
interval, and let $\Omega \subset \mathbb{R}^{d}$ ($d=2$ or $3$) be a
bounded open domain with a piecewise smooth Lipschitz continuous boundary $%
\partial \Omega $, $\nu =\nu (x)$ being the outward unit normal vector at
each point $x$ of $\partial \Omega $. Using the notations $\overline{\Omega }%
:=\Omega \cup \partial \Omega $, $Q_{T}:=\Omega \times (0,T)$ and $\overline{%
Q}_{T}:=\overline{\Omega }\times \lbrack 0,T)$, we consider the
scalar-valued function $c:\overline{Q}_{T}\rightarrow \mathbb{R}$ that
satisfies the equation%
\begin{equation}
\left\{
\begin{array}{l}
\displaystyle\frac{\partial c}{\partial t}+\nabla \cdot (uc)-\mu \Delta
c+a_{0}c=f(x,t)\text{\ \textrm{in} }Q_{T}, \\
\\
(\mu \nabla c)\cdot \nu =0\ \ \text{\textrm{on}\ \ }\partial \Omega \times
(0,T), \\
\\
c(x,0)=c^{0}(x)\text{ \ \textrm{in} \ }\Omega ,%
\end{array}%
\right.  \label{int1}
\end{equation}%
where $u:=\overline{Q}_{T}\rightarrow \mathbb{R}^{d}$ is a vector-valued
function representing a flow velocity, for simplicity we assume that for all
$t$, $u\cdot \nu =0$ on $\partial \Omega $; $\mu $ and $a_{0}$ are positive constant
coefficients. If $c(x,t)$ describes the space-time evolution of a physical
property, let us say, the concentration of a chemical substance, the
partial differential equation of problem (\ref{int1}) is the point-wise
representation of the conservation law of such a substance. The global mass
balance of $c(x,t)$ is given by the expression%
\begin{equation}
\int_{\Omega }c(x,t)d\Omega +a_{0}\int_{0}^{t}\int_{\Omega }c(x,\tau
)d\Omega d\tau =\int_{\Omega }c^{0}(x)d\Omega +\int_{0}^{t}\int_{\Omega
}f(x,\tau )d\Omega d\tau .  \label{int1.1}
\end{equation}%
Thus, if $a_{0}=0$ and $f=0$ then from (\ref{int1.1}) it follows that the
integral of $c(x,t)$ over the domain $\Omega $ is constant for all $t$.
Since the global mass balance is a property derived from the conservation
law for $c(x,t)$, then the numerical methods employed to calculate a
numerical solution of (\ref{int1}) should preserve such a property; in this
case we say that the numerical methods are mass conservative or simply
conservative.

Many authors, such as (\cite{BK}, \cite{BS1}-\cite{BS4}, \cite{boukir}, \cite{DR}, \cite
{ewing}, \cite{Pi}, \cite{RT}, \cite{Su}), just to cite a few,
have applied Lagrange-Galerkin methods to calculate a numerical solution of (%
\ref{int1}) when the problem is advection dominated and the velocity field $%
u $ is divergence free, i.e., $\nabla \cdot u=0$ a.e. in $Q_{T}$, for in
this case%
\begin{equation*}
\frac{\partial c}{\partial t}+\nabla \cdot (uc)=\frac{Dc}{Dt}:=\frac{%
\partial c}{\partial t}+u\cdot \nabla c\text{,}
\end{equation*}%
$\frac{D}{Dt}$ being the material derivative. Advection
dominated-diffusion problems are characterized by the formation of
narrow regions along the solid boundaries in which the solution
develops strong gradients. The dimensional analysis of equation
(\ref{int1}) yields the so called P\'{e}clet number $Pe=UL/\mu $
($U$ and $L$ are characteristic velocity and length scales
respectively), which is very high in this type of problems, and via
perturbation analysis one can show that the width of the boundary
layers is $O(Pe^{-\alpha }),\ 0<\alpha <1$. So, in order to obtain a
good numerical solution when $Pe$ is very high, one needs to
allocate a sufficient number of mesh points in the boundary layers,
otherwise spurious oscillations develop there polluting the
numerical solution in the whole domain; this means that the mesh
parameter $h$ in the boundary layer region must be very small so that
numerical stability reasons recommend the usage of implicit time
marching schemes, because for
explicit time schemes the length of the time step $\Delta t$ should be $O(h^{2})$ in order to obtain a
stable numerical solution, but this may be an unreasonably small
time step. Different approaches have been devised to overcome this
drawback. In the Eulerian formulation of the numerical method, which
uses a fixed mesh, we should refer to the SUPG
(Stream-Upwind-Petrov-Galerkin) and the Galerkin/least squares
methods developed by T. Hughes and coworkers, see for instance
\cite{BH} and \cite{HFH}. In the Lagrangian approach \cite{MG} one
attempts to devise a stable numerical method by allowing the mesh to
follow the trajectories of the flow; the problem now is that after a number of time steps the mesh
undergoes large deformations due to stretching and shearing, and consequently some sort of remeshing has
to be made in order to proceed with the calculations. The latter may
become a source of large errors. In the Eulerian-Lagrangian approach
the purpose is to get a method that combines the good properties of
both the Eulerian and Lagrangian approaches. There have been various
methods trying to do so, among them we shall cite the
characteristics streamline diffusion (hereafter CSD), the backward
Lagrange-Galerkin or simply Lagrange-Galerkin (LG) methods (also
termed Characteristics Galerkin). The CSD method has been developed
by, for example, \cite{hansbo-1} and \cite{johnson-caracteristicas}
and intends to combine the good properties of both the Lagrangian
methods and the streamline diffusion method by orienting the
space-time mesh along the characteristics in space-time, thus
yielding to a particular version of the streamline diffusion method.
The Lagrange-Galerkin methods approximate the material derivative
\begin{equation*}
\frac{Dc}{Dt}=\frac{\partial c}{\partial t}+u\cdot \nabla c
\end{equation*}
at each time step by a backward in time discretization along the
characteristics $X\left( x,t_{n+1};t\right) $ of the
operator ${\displaystyle\frac{\partial }{\partial t}}+u\cdot \nabla
$, $t_{n-l}\leq t<t_{n+1}$, $l$ being an integer. At $t=t_{n+1}$,
$X\left( x,t_{n+1};t_{n+1}\right) =x \in \Omega $. The diffusion
terms are implicitly discretized on the fixed mesh generated on
$\overline{\Omega }$. The point here is how to evaluate $c\left(
X\left( x,t_{n+1};t\right) ,t\right) $. One way is by $L^{2}$
projection in the finite dimensional space associated with the fixed
mesh, as the Lagrange-Galerkin methods do. Another way is by
polynomial interpolation projection, in this case the method is called
semi-Lagrangian. The advantages of the LG methods are various. From a
practical point of view we have the following: (i) they
partially circumvent the troubles caused by the advective terms,
because discretizing backward in time along the characteristic curves is a
natural way of introducing upwinding in the space discretization of
the differential equation; (ii) unlike the pure Lagrangian methods,
they do not suffer from mesh-deformation, so that no remeshing is
needed; (iii) they yield algebraic symmetric systems of equations to
be solved. From a numerical analysis point of view, we shall show in
this paper that the constant $C$ that appears in the error
estimates of the LG methods is much smaller than the corresponding
constant of the standard Eulerian Galerkin methods and, what it is
more important, is uniformly bounded with respect to the parameter
$\mu $.\ To appreciate the relevance of this behavior of $C$, we
recall \cite{Quar-valli} that the error constant of standard Eulerian
Galerkin methods in convection-diffusion problems takes the form
$C_{G}\sim \mu ^{-1}\exp (tC\mu ^{-1})$, and is sharp because the
Gibbs phenomenon, which appears in the boundary layers when the mesh
is coarse, grows exponentially.\ The $\mu ^{-1}$ dependence of the
constant $C_{G}$ makes the standard Eulerian Galerkin methods be unreliable
in advection dominated diffusion problems because in such problems
$\mu \ll 1$ and consequently $C_{G}$ becomes very large. This does not
happen in the LG methods because the dependence on $\mu $ is uniformly
bounded; however, this does not mean that these methods are free from the Gibbs
phenomenon if the grid is coarse, but such a phenomenon is well
under control and so is its pollutant effect. Nevertheless, the LG
methods have several drawbacks such as: (i) the calculation of the
feet of the characteristic curves every time step, this requires
solving, backward in time, many systems of ordinary differential
equations; and (ii) the calculation of some integrals, which come
from the Galerkin projection, whose integrands are the product of
functions defined in two different meshes. The first shortcoming is
in some way related to the second one because in order to keep the numerical solution stable, such integrals have to be calculated exactly; since, in general, this cannot be done this way, then one has to make use of high order quadrature rules, see \cite{Bermudez} where a
study on the behavior of the method with different quadrature rules
is performed. The use of high order quadrature rules means that many
quadrature points per element should be employed to evaluate the
integrals and, therefore, since each quadrature point has associated
a departure point, then many systems of ordinary differential
equations have to be solved numerically every time step; hence, the
whole procedure may become less efficient than it looks at first,
particularly in unstructured meshes, because the numerical
calculation of the feet of the characteristics curves requires the
location and identification of the elements containing such points,
and this task is not easy to do in such meshes; however, Bermejo and
Saavedra \cite{BS2} \ introduced the so called Modified
Lagrange-Galerkin (MLG) methods to alleviate the drawback (ii) of
the conventional LG methods while maintaining the rate of
convergence. The MLG methods were also applied to integrate the
Navier Stokes equations, see \cite{BS1} and \cite{BS3}. A brief
description of the conventional Lagrange-Galerkin approach to
integrate advection diffusion reaction problems when $\nabla
\cdot u=0$ is as follows. Let%
\begin{equation*}
\mathcal{P}_{\Delta t}:0=t_{0}<t_{1}<\cdots <t_{N}=T,\ \Delta
t=t_{n}-t_{n-1},
\end{equation*}%
be a uniform partition of $[0,T]$, at each time $t_{n}$ $(n=1,2,\ldots ,N)$,
the material derivative is discretized backward in time along its
characteristic curves $X(x,t_{n};t)$, $t_{n-1}\leq t<t_{n}$, as%
\begin{equation*}
\left. \frac{Dc(x,t)}{Dt}\right\vert _{t=t_{n}}=\frac{%
c(x,t_{n})-c(X(x,t_{n};t_{n-1}),t_{n-1})}{\Delta t}.
\end{equation*}%
The characteristic curves $X(x,t_{n};t)$ are, under the hypothesis
$u\in C([0,T];W^{1,\infty }(\Omega ))$, the unique solution
of the differential system%
\begin{equation}
\left\{
\begin{array}{l}
\displaystyle\frac{dX(x,t_{n};t)}{dt}=u(X(x,t_{n};t),t)\text{,} \\
\\
X(x,t_{n};t_{n})=x,\ \ x\in \overline{\Omega }.%
\end{array}%
\right.  \label{int2}
\end{equation}%
Here, $t\rightarrow X(x,t_{n};t)$ represents the parametrized trajectory of
a particle that moving with speed $u$ occupies the point $x$\ at time $t_{n}$%
, i.e.,
\begin{equation}
X(x,t_{n};t)=x-\int_{t}^{t_{n}}u(X(x,t_{n};s),s)ds.  \label{int3}
\end{equation}
Let $\Omega (t):=\{X(x,t_{n};t),\ x\in \Omega \}$, then assuming that for
all $t$, $\nu \cdot u=0$ on $\partial \Omega $, it follows that for $x\in
\partial \Omega $, $X(x,t_{n};t)\in \partial \Omega $, consequently $\Omega (t)=\Omega
$ is invariant \cite{CW}. So, for $\left\vert t_{n}-t\right\vert $
sufficiently small, $x\rightarrow X(x,t_{n};t)$ is a quasi-isometric
homeomorphism from $\Omega $ onto $\Omega $ with Jacobian determinant $%
J(x,t_{n},t):=\det \left( \frac{\partial X_{i}(x,t_{n};t)}{\partial x_{j}}%
\right) =1$ a.e. in $\Omega $, and such that%
\begin{equation}
K^{-1}\left\vert x-y\right\vert \leq \left\vert
X(x,t_{n},t)-X(y,t_{n},t)\right\vert \leq K\left\vert x-y\right\vert ,
\label{int4}
\end{equation}%
where $K=\exp \left( (t_{n}-t)\left\vert u\right\vert _{L^{\infty
}(0,T;\ W^{1},^{\infty }(\Omega ))}\right) $ and $\left\vert \cdot
\right\vert $ is the Euclidean distance between two points in
$\overline{\Omega }$. Hence, the domain $\Omega $ is invariant under
the flow mapping $X(x,t_{n};t)$. The same result is
proven in \cite{Su} when $u=0$ on $\partial \Omega $. Thus, assuming that $%
\nabla \cdot u=0$, for $t$ $\in \lbrack t_{n-1},t_{n}]$ we can write the
backward Euler discretization of the differential equation of (\ref{int1})
along the characteristic curves as
\begin{equation*}
\displaystyle\frac{c^{n}-c^{n-1}(X(x,t_{n};t_{n-1}))}{\Delta t}-\mu \Delta
c^{n}+a_{0}c^{n}=f^{n}\text{ }\mathrm{in}\text{ }\Omega ,
\end{equation*}%
hereafter we shall use the notation $a^{n}$ to denote a function $a(x,t_{n})$
unless otherwise stated. Now, if we apply the finite element method for the
space approximation of this equation, denoting by $V_{h}$ the $H^{1}$%
-conforming finite element space, it follows that for all $n$ we calculate $%
c_{h}^{n}\in V_{h}$ as the solution of the equation%
\begin{equation*}
\displaystyle\frac{(c_{h}^{n}-c_{h}^{n-1}(X(x,t_{n};t_{n-1})),v_{h})_{\Omega
}}{\Delta t}+\mu \left( \nabla c_{h}^{n},\nabla v_{h}\right) _{\Omega
}+a_{0}(c_{h}^{n},v_{h})_{\Omega }=(f^{n},v_{h})_{\Omega }\ \forall v_{h}\in
V_{h},
\end{equation*}%
where $\left( c_{h}^{0},v_{h}\right) _{\Omega }=\left( c^{0},v_{h}\right)
_{\Omega }$, with $(a,b)_{\Omega }:=\int_{\Omega }abd\Omega $, $a$ and $b$
being measurable functions defined on $\Omega $. In this spirit, we can
construct the Lagrange-Galerkin methods when the discretization at time $%
t_{n} $ of the material derivative is performed via Backward Differentiation
Formulas (BDF) of order $q=1,2,3,4,5$ and $6$ as follows. Assuming we are
given starting approximations $c_{h}^{0},\ldots ,c_{h}^{q-1}\in V_{h}$, for
each $n>q-1$ we calculate $c_{h}^{n}$ as the solution of the equation%
\begin{equation*}
\frac{1}{\Delta t}\left( \sum_{i=0}^{q}\alpha _{i}\widehat{c}%
_{h}^{n-i},v_{h}\right) _{\Omega }+\mu \left( \nabla c_{h}^{n},\nabla
v_{h}\right) _{\Omega }+a_{0}(c_{h}^{n},v_{h})_{\Omega
}=(f^{n},v_{h})_{\Omega }\ \forall v_{h}\in V_{h},
\end{equation*}%
where $\widehat{c}_{h}^{n-i}:=c_{h}^{n-i}(X(x,t_{n};t_{n-i}))$. Notice that
for $i=0$, $\widehat{c}_{h}^{n}=c_{h}^{n}(x)$ because $X(x,t_{n};t_{n})=x$,
see (\ref{int2}). The real coefficients $\alpha _{i}$ are tabulated in, for
example, \cite{HNW}. For $q=1$, the above formula reduces to the first order
backward Euler Lagrange-Galerkin method with $\alpha _{0}=1$ and $\alpha
_{1}=-1$, whereas for $q=2$ the formula yields the second order in time
Lagrange-Galerkin-BDF-2 method with $\alpha _{0}=3/2$,\ $\alpha _{1}=-2$ and
$\alpha _{2}=1/2$.

Recently, the Lagrange-Galerkin approach in combination with the BDF-2 time
marching scheme has been applied to integrate problem (\ref{int1}) when $%
\nabla \cdot u\neq 0$ in \cite{FK}, being this work an extension of the
Lagrange-Galerkin-BDF-1 method introduced in \cite{RT2}, see Remark \ref%
{tabata} below. Our Lagrange-Galerkin methods to integrate problem (\ref{int1}%
), see (\ref{s2:11}) below, differs from those used in \cite{RT2} and \cite%
{FK} in that it sets up the weak formulation of (\ref{int1}) on the
transported domain $\Omega _{h}(t):=X(\Omega ,t_{n},t)$, see (\ref{s2:8})
below, and then applies the finite element technique combined with the BDF-q
scheme ($1\leq q\leq 5$) to calculate the numerical solution $c_{h}^{n}$;
so, our methods have to calculate terms such as $\int_{\Omega
_{h}(t_{n-i})}c_{h}^{n-i}(y)v_{h}^{n-i}(y)d\Omega _{h}(t_{n-i})$ that must
be approximated by quadrature rules of high order; nevertheless, making use
of the algorithm developed in \cite{BS1}-\cite{BS4} for the conventional LG
methods when $\nabla \cdot u=0$, these integrals can be evaluated very
efficiently \cite{colera}.

We introduce some notation about the functional spaces we use in the paper.
For $s\geq 0$ real and real $1\leq p\leq \infty $, $W^{s,p}(\Omega )$
denotes the real Sobolev spaces defined on $\Omega $ for scalar real-valued
functions. $\left\Vert \cdot \right\Vert _{W^{s,p}(\Omega )}$ and $%
\left\vert \cdot \right\vert _{W^{s,p}(\Omega )}$ denote the norm and
semi-norm, respectively, of $W^{s,p}(\Omega )$. When $s=0$, $W^{0,p}(\Omega
):=L^{p}(\Omega )$. For $p=2$, the spaces $W^{s,2}(\Omega )$ are denoted by $H^{s}(\Omega )$, which are real Hilbert spaces with inner product $(\cdot,\cdot )_{s, \Omega}$. For $s=0$, $H^{0}(\Omega ):=L^{2}(\Omega )$, the inner
product in $L^{2}(\Omega )$ is denoted by $(\cdot ,\cdot )_{\Omega}$. The
corresponding spaces of real vector-valued functions, $v:\Omega \rightarrow
\mathbb{R}^{d}$, $d>1$ integer, are denoted by $W^{s,p}(\Omega )^{d}$ $%
:=\{v:\Omega \rightarrow \mathbb{R}^{d}:\ v_{i}\in W^{s,p}(\Omega ),\ 1\leq
i\leq d\}$. Let $X$ be a real Banach space $(X,\left\Vert \cdot \right\Vert
_{X})$, if $v:(0,T)\rightarrow X$ is a strongly measurable function with
values in $X$, we set $\left\Vert v\right\Vert _{L^{p}(0,t;X)}:=\left(
\int_{0}^{t}\left\Vert v(\tau )\right\Vert _{X}^{p}d\tau \right) ^{1/p}$ for
$1\leq p<\infty $, and $\left\Vert v\right\Vert _{L^{\infty
}(0,t;X)}:=ess\sup_{0<\tau <t}\left\Vert v(\tau )\right\Vert _{X}$; when $%
t=T $, we shall write, unless otherwise stated, $\left\Vert v\right\Vert
_{L^{p}(X)}$. We shall also use the following discrete norms:%
\begin{equation*}
\left\Vert v\right\Vert _{l^{p}(X)}:=\left( \Delta t\sum_{i=1}^{N}\left\Vert
v(\tau _{i})\right\Vert _{X}^{p}\right) ^{1/p}\text{ \ \textrm{and \ }}%
\left\Vert v\right\Vert _{l^{\infty }(X)}:=\max_{1\leq i\leq N}\left\Vert
v(\tau _{i})\right\Vert _{X}
\end{equation*}%
corresponding to the time discrete spaces

\begin{equation*}
l^{p}(0,T;X):=\left\{ v:(0,t_{1},t_{2},\ldots ,t_{N}=T)\rightarrow X:\
\left( \Delta t\sum_{i=0}^{N}\left\Vert v(\tau _{i})\right\Vert
_{X}^{p}\right) ^{1/p}<\infty \right\} \ \ \mathrm{and}
\end{equation*}%
\begin{equation*}
l^{\infty }(0,T;X):=\left\{ v:(0,t_{1},t_{2},\ldots ,t_{N}=T)\rightarrow X:\
\max_{1\leq i\leq N}\left\Vert v(\tau _{i})\right\Vert _{X}<\infty \right\} ,
\end{equation*}%
respectively; and write $l^{p}(0,T;X)$ as $l^{p}(X)$. Finally, we shall also
use the spaces of $r$-times continuously differentiable functions on $\Omega $ $C^{r}(\Omega )$,
when $r=0$ we write $C(\Omega )$ instead of $C^{0}(\Omega )$, the spaces $%
C^{r,1}(\overline{\Omega })$, $r\geq 0$, of functions defined on the closure
of $\Omega $, $r$ -times continuously differentiable and with the $r$th
derivative being Lipschitz continuous, and the space of $r$-times continuously differentiable and
bounded functions in time with values in $X$ denoted by $C^{r}([0,T];X)$ with
norm $\left\Vert v\right\Vert _{C^{r}([0,T];X)}=\sup_{0\leq t\leq T}\sum_{0\leq \alpha \leq r}\left\Vert
D^{\alpha}v(t)\right\Vert _{X}$.

Throughout this paper $C$ will denote a generic positive constant which is
independent of the diffusion ($\mu $) and reaction ($a_{0}$) coefficients
and of both the space and time discretization parameters $h$ and $\Delta t$
respectively; $C$ will take different values in different places of
appearance. Many times we use the term elementary inequality to denote the
Cauchy inequality $ab\leq \frac{\epsilon a^{2}}{2}+\frac{b^{2}}{2\epsilon }$
($a,b,\epsilon >0$). We shall need the version of the discrete Gronwall
inequality \cite{HRan}%
\begin{equation}
a_{N}+\Delta t\sum_{n=0}^{N}b_{n}\leq \exp \left( \Delta
t\sum_{n=0}^{N}\sigma _{n}\gamma _{n}\right) \left( \Delta
t\sum_{n=0}^{N}c_{n}+B\right) ,  \label{int5}
\end{equation}%
where for integers $n\geq 0$, $\sigma _{n}=(1-\Delta t\gamma _{n})^{-1}$, $%
\Delta t\gamma _{n}<1$, and $a_{n},b_{n},c_{n}$ and $B$ are positive numbers
such that%
\begin{equation}
a_{N}+\Delta t\sum_{n=0}^{N}b_{n}\leq \Delta t\sum_{n=0}^{N}\gamma
_{n}a_{n}+\Delta t\sum_{n=0}^{N}c_{n}+B.  \label{int6}
\end{equation}%
If the first term on the right hand side of (\ref{int6}) only extends up to $%
N-1$, then the inequality (\ref{int5}) holds for all $\Delta t>0$ with $%
\sigma _{n}\equiv 1$.

The paper is organized as follows. In Section 2, we introduce the method in
the framework of parabolic problems in evolving domains. The study of the
properties of the family of moving meshes and their associated finite
element spaces is presented in Section 3. Section 4 is devoted to the
stability and error analysis in the $L^{2}$-norm respectively. Some
numerical tests are presented in Section 5. Since the elements of the moving
meshes have, in general, $d$-curved faces, then, for the sake of
completeness of the paper, we collect some results of the theory of curved
elements developed by Bernardi \cite{Bern}, which are relevant for us, in an
Appendix.

\section{The nearly-conservative Lagrange-Galekin-(BDF-q) methods}

Returning to problem (\ref{int1}), we shall describe in this section the
procedure to apply the LG methods in order to calculate a numerical solution
and term such a procedure the nearly-conservative Lagrange Galerkin (NCLG)
methods. (Below, we point out the reason of the name nearly-conservative). To
this end, let $h$ be the mesh parameter such that $0<h<1$, at time $t=0,$ a
family of conforming, quasi-uniformly regular partitions $\ \left( \mathcal{T%
}_{h}\right) _{h}$ is generated on $\overline{\Omega }$; the partitions $%
\mathcal{T}_{h}$ are composed of curved $d-$simplices $T_{l}$ of order $k$
(see Appendix for the definition and properties of curved $d-$simplices of
order $k$). Let $NE>1$ and $M>1$ denote the number of elements and nodes
respectively in the partition, then
\begin{equation*}
\mathcal{T}_{h}=\left\{ T_{l}\right\} _{l=1}^{NE}\text{ \ \textrm{such that}
\ }\overline{\Omega }=\bigcup\limits_{l=1}^{NE}T_{l}.
\end{equation*}%
The partition $\mathcal{T}_{h}$ is said to be exact because $\overline{%
\Omega }=\bigcup\limits_{l=1}^{NE}T_{l}$. We allow the existence of curved
elements in the sense defined in \cite{Bern}, in particular the elements
that have a $(d-1)-$face which is a piece of the boundary $\partial \Omega $,
whereas the elements with only one vertex on the boundary or void
intersection with the boundary have plain $(d-1)-$faces. Our formulation of the
NCLG methods is in the framework of finite elements-backward differences formulas of
order $q$ (BDF-$q$) with $1\leq q\leq 6$. For $n$ integer and $n\geq q$, at
each time $t_{n}$ of the partition $\mathcal{P}_{\Delta t}$ we identify
every point $x$ of $\overline{\Omega }$ with a material particle moving with
velocity $u$ and trace backward its trajectory $X(x,t_{n};t)$ for $t\in
\lbrack t_{n-q},t_{n})$. Then, based on the uniqueness of the trajectories $%
X(x,t_{n};t)$, we can define an element $T_{l}(t)$ as the image of the
element $T_{l}\in \mathcal{T}_{h}$, i.e.,$T_{l}(t):=X(T_{l},t_{n};t)$.\ $%
T_{l}(t)\subset \overline{\Omega }$ because $X(x,t_{n};t)\in \overline{%
\Omega }$. Notice that for all $n\geq q$, $%
T_{l}(t_{n})=X(T_{l},t_{n};t_{n})=T_{l}\in \mathcal{T}_{h}$. Thus, for all $%
t\in \lbrack t_{n-q},t_{n})$ we can construct via the flow mapping $%
X(x,t_{n};t)$ a family of partitions $\left( \mathcal{T}_{h}(t)\right) _{h}$
where%
\begin{equation*}
\mathcal{T}_{h}(t):=\left\{ T_{l}(t)=X(T_{l},t_{n};t):\ T_{l}\in \mathcal{T}%
_{h},\ 1\leq l\leq NE\right\} .
\end{equation*}%
So that we set: for all $t_{n}\in \mathcal{P}_{\Delta t},\ t_{n}\geq t_{q}$,
and $\forall t\in \lbrack t_{n-q},t_{n})$
\begin{equation*}
\mathcal{T}_{h}(t)=\left\{ T_{l}(t)\right\} _{l=1}^{NE}\text{, \ \ \ \ \ }%
\overline{\Omega }_{h}(t)=\bigcup\limits_{l=1}^{NE}T_{l}(t).
\end{equation*}%
Due to the invariance of $\overline{\Omega }$ and given that for all $l$, $%
T_{l}(t)\subset \overline{\Omega }$, then $\overline{\Omega }_{h}(t)=%
\overline{\Omega }$. In Section 3 we study the properties of the
families of partitions $\left( \mathcal{T}_{h}\right) _{h}$ and $\left(
\mathcal{T}_{h}(t)\right) _{h}$. Now, we state the following regularity
result that is needed in the sequel, a proof of it can be found in (\cite[%
p.100]{Hart}).

\begin{lemma}
\label{solX} \textit{Assume that }$u\in C$\textit{\ }$([0,T];W^{k+1,\infty
}(\Omega )^{d}),\ k\geq 0$.\textit{\ Then for all }$n$ such that $q\leq $%
\textit{\ }$n\leq N$\textit{, there exists a unique solution }$t\rightarrow
X(x,t_{n};t)\ (t\in \lbrack t_{n-q},t_{n})\subset \lbrack 0,T])$\textit{\ of
(\ref{int2}) such that }$X(x,t_{n};t)\in C^{1}(0,T;W^{k+1,\infty }(\Omega
)^{d})$\textit{. Furthermore, let the multi-index }$\alpha \in N^{d}$\textit{%
, for all }$\alpha $,\textit{\ }$1\leq \mid \alpha \mid \leq k+1,$\textit{\ }%
$\displaystyle\frac{\partial ^{\left\vert \mathbf{\alpha }\right\vert
}X_{i}(x,t_{n},t)}{\partial x_{j}^{\alpha _{j}}}\in
C^{1}(0,T;W^{k+1-\left\vert \alpha \right\vert ,\infty }(\Omega ))$, $1\leq
i,j\leq d.$
\end{lemma}

The Jacobian determinant $J(x,t_{n};t)$ $:=\det \left( \frac{\partial
X_{i}(x,t_{n};t)}{\partial x_{j}}\right) $ of the transformation $%
X(x,t_{n};t)$ is, according to Lemma \ref{solX}, in $C
^{1}([0,T];W^{k,\infty }(\Omega ))$, this means that $t\rightarrow
J(x,t_{n};t) $ is a continuous function which does not vanish and $%
J(x,t_{n};t_{n})=1$, then $J(x,t_{n};t)>0$ and satisfies the equation \cite%
{CHM}%
\begin{equation}
\frac{\partial \ln (J)}{\partial t}=\nabla _{X}\cdot u(X(x,t_{n};t),t),\
J(X(x,t_{n};t_{n}))=1.\   \label{s2:0}
\end{equation}%
Hence%
\begin{equation}
J(x,t_{n};t)=\exp \left( -\int_{t}^{t_{n}}\nabla _{X}\cdot
u(X(x,t_{n};s),s)ds\right) .  \label{s2:1}
\end{equation}%
Given that $\left\Vert \nabla _{X}\cdot u\right\Vert _{L^{\infty }(Q_{T})}$
is bounded, we set $C_{\mathrm{div}}:=\left\Vert \nabla _{X}\cdot
u\right\Vert _{L^{\infty }(Q_{T})}$, and consequently it follows that
\begin{equation}
\exp (-C_{\mathrm{div}}\left\vert t-t_{n}\right\vert )\leq J(x,t_{n};t)\leq
\exp (C_{\mathrm{div}}\left\vert t-t_{n}\right\vert ).  \label{s2:2}
\end{equation}%
We remark that for all $n\geq q$ and $t$ taking values in the intervals $%
I_{n}:=[t_{n-q},t_{n}]$, one can also view $\Omega _{h}(t)$ as an evolving
domain defined by%
\begin{equation}
\forall n\geq q,\ \ \overline{\Omega }_{h}(t)=\left\{
\begin{array}{l}
\bigcup\limits_{l=1}^{NE}T_{l}(t)\ \ \mathrm{for\ \ }t_{n-q}\leq t<t_{n}, \\
\\
\overline{\Omega }=\bigcup\limits_{l=1}^{NE}T_{l}\ \ \mathrm{for\ \ }t=t_{n}.%
\end{array}%
\right.  \label{s2:20}
\end{equation}%
So, for the developments that follow, we adopt the methodology presented in
\cite{Alph} and \cite{ER} for parabolic equations in evolving domains. Thus,
for $t\in I_{n}$, let $\{Y(t)\}$ be a family of normed functional spaces of
real vector or scalar-valued functions defined on $\Omega _{h}(t)$, then we
can define a family of linear homeomorphisms $\phi
_{t_{n}}^{t}:Y(t_{n})\rightarrow Y(t)$, with inverse $\phi
_{t}^{t_{n}}:Y(t)\rightarrow Y(t_{n})$, via the mapping $X(x,t_{n};t)$ as%
\begin{equation}
v\in Y(t_{n})\rightarrow \phi _{t_{n}}^{t}v\in Y(t)\ \mathrm{such\ that}\
\left( \phi _{t_{n}}^{t}v\right) (\cdot )=v(X(\cdot ,t;t_{n})).
\label{s2:21}
\end{equation}%
Notice that $X(\cdot ,t;t_{n})$ is the inverse of $X(\cdot ,t_{n};t)$
because, for example, if $y=X(x,t_{n};t)\in \Omega _{h}(t)$, then $%
X(y,t;t_{n})=X(X(x,t_{n};t),t;t_{n})=x\in \Omega $. Based on Theorem 1.1.7/1
of \cite{Maz} and Lemma \ref{solX}, we can say that if the mapping $%
X(x,t_{n};t)$ is of class $C^{k,1}(\Omega )$ and $Y(t)=W^{k+1,p}(\Omega
_{h}(t))$, then $v\in Y(t_{n})=W^{k+1,p}(\Omega )$, $1\leq p\leq \infty $.
Moreover, there is a constant $c_{1}$ such that%
\begin{equation}
c_{1}^{-1}\left\Vert v\right\Vert _{W^{k+1,p}(\Omega )}\leq \left\Vert \phi
_{t_{n}}^{t}v\right\Vert _{W^{k+1,p}(\Omega _{h}(t))}\leq c_{1}\left\Vert
v\right\Vert _{W^{k+1,p}(\Omega )}  \label{s2:22}
\end{equation}%
and the map $t\rightarrow \left\Vert \phi _{t_{n}}^{t}v\right\Vert
_{W^{k+1,p}(\Omega (t))}$ is continuous for all $v\in W^{k+1,p}(\Omega )$. \
Now, setting%
\begin{equation*}
G_{I_{n}}=\bigcup_{t\in I_{n}}\Omega _{h}(t)\times \{t\},
\end{equation*}%
and using the notation $y=X(x,t_{n};t)\in \Omega _{h}(t)$, we can recast the
conservation law (\ref{int1}) restricted to the intervals $I_{n}$ as
\begin{equation}
\left\{
\begin{array}{l}
\displaystyle\frac{Dc}{Dt}-\mu \nabla _{y}\cdot \left( \nabla _{y}c\right)
+a_{0}c+c\nabla _{y}\cdot u=f(y,t)\text{\ \textrm{in}\ }G_{I_{n}}, \\
\\
\mu \nabla _{y}c\cdot \nu (y)=0\ \ \text{\textrm{on}\ \ }\partial \Omega
_{h}\times \lbrack t_{n-q},t_{n}),%
\end{array}%
\right.  \label{s2:3}
\end{equation}%
where%
\begin{equation}
\frac{Dc}{Dt}=\frac{\partial c(y,t)}{\partial t}+u(y,t)\cdot \nabla
_{y}c(y,t).  \label{s2:4}
\end{equation}%
Let $Y(t)$ be a Hilbert space, we introduce the spaces
\begin{equation*}
L_{(n,q)Y}^{2}:=\left\{ v:I_{n}\rightarrow \bigcup_{t\in I_{n}}Y(t)\times
\{t\},t\rightarrow (\overline{v}(t),t):\phi _{t}^{t_{n}}\overline{v}(t)\in
L^{2}(t_{n-q},t_{n};Y(t_{n}))\right\} .
\end{equation*}%
$L_{(n,q)Y}^{2}$ are separable Hilbert spaces with norm \cite{Alph}%
\begin{equation*}
\left\Vert v\right\Vert _{L_{(n,q)Y}^{2}}=\left(
\int_{t_{n-q}}^{t_{n}}\left\Vert v(t)\right\Vert _{Y(t)}^{2}dt\right) ^{1/2},
\end{equation*}%
and isomorphic to the spaces $L^{2}(t_{n-q},t_{n};Y(t_{n}))$. For any
integer $r\geq 0$, we also consider the space of smoothly evolving in time
functions

\begin{equation*}
C_{(n,q)Y}^{r}:=\left\{ v\in L_{(n,q)Y}^{2}:t\rightarrow \phi
_{t}^{t_{n}}v(t)\in C^{r}(\left[ t_{n-q},t_{n}\right] ;Y(t_{n}))\right\} .
\end{equation*}%
Thus, for $c\in C_{(n,q)Y}^{1}$ and $u\in C([0,T];W^{k+1,\infty
}(\Omega)^{d})$ one can write the strong material derivative as%
\begin{equation}
\frac{Dc}{Dt}=\phi _{t_{n}}^{t}\frac{d}{dt}\left( \phi _{t}^{t_{n}}c\right) .
\label{s2:5}
\end{equation}%
Notice that for $v\in Y(t)$, $\displaystyle\frac{Dv}{Dt}=0$ $%
\Longleftrightarrow \exists \widehat{v}=\phi _{t}^{t_{n}}v\in Y(t_{n})$ such
that $\displaystyle\frac{d\widehat{v}}{dt}=0$ \cite{Alph, ER}.

From (\ref{s2:3}), we obtain that\ for $c(t)\in C_{(n,q)Y}^{1}$ and $\forall
v\in Y(t)$
\begin{equation}
\int_{\Omega _{h}(t)}\left\{ \left( \frac{Dc}{Dt}+c\nabla _{y}\cdot u\right)
v+\mu \nabla _{y}c\cdot \nabla _{y}v+a_{0}cv-fv\right\} d\Omega _{h}(t)=0.
\label{s2:6}
\end{equation}%
Making use of the inverse homeomorphism $\phi _{t}^{t_{n}}:Y(t)\rightarrow
Y(t_{n})$, so that for all $z(X(x,t_{n},t),t)\in Y(t)$ there is one and only
one $\widehat{z}(x,t)\in Y(t_{n})$ such that $\widehat{z}%
(x,t)=z(X(x,t_{n};t),t)$, and the equation (\ref{s2:0}) of the Jacobian
determinant, we can recast (\ref{s2:6}) as follows:
\begin{equation}
\int_{\Omega }\frac{DJ\widehat{c}}{Dt}\widehat{v}d\Omega +\int_{\Omega
}\left( \mu \left( F^{-T}\nabla _{x}\widehat{c}\right) \cdot \left(
F^{-T}\nabla _{x}\widehat{v}\right) +a_{0}\widehat{c}\widehat{v}-\widehat{f}%
\widehat{v}\right) Jd\Omega =0,  \label{s2:7}
\end{equation}%
where $\widehat{c}:=\widehat{c}(x,t)\in C^{1}(\left[ t_{n-q},t_{n}\right]
;Y(t_{n}))$, and $F^{-T}$ denotes the transpose of the inverse Jacobian
matrix $F_{ij}=\left( \frac{\partial X_{i}(x,t_{n};t)}{\partial x_{j}}%
\right) $. Moreover, noting that%
\begin{equation*}
\left\{
\begin{array}{l}
\displaystyle\frac{d}{dt}\int_{\Omega _{h}(t)}c(y,t)v(y,t)d\Omega_{h} (t)=%
\displaystyle\frac{d}{dt}\int_{\Omega }\widehat{c}(x,t)\widehat{v}%
(x,t)Jd\Omega \\
\\
=\displaystyle\int_{\Omega }\frac{DJ\widehat{c}}{Dt}\widehat{v}d\Omega +%
\displaystyle\int_{\Omega }\frac{D\widehat{v}}{Dt}\widehat{c}Jd\Omega \ \ \
\ \ \ \ \ \ \ \ \ \ \ \ \ \ \ \ \ \mathrm{(by\ virtue\ of\ (\ref{s2:7}))} \\
\\
=-\displaystyle\int_{\Omega }\left( \mu \left( F^{-T}\nabla _{x}\widehat{c}%
\right) \cdot \left( F^{-T}\nabla _{x}\widehat{v}\right) +a_{0}\widehat{c}%
\widehat{v}-\widehat{f}\widehat{v}\right) Jd\Omega +\displaystyle%
\int_{\Omega }\frac{D\widehat{v}}{Dt}\widehat{c}Jd\Omega \\
\\
=-\displaystyle\int_{\Omega _{h}(t)}\left\{ \mu \nabla _{y}c\cdot \nabla
_{y}v+a_{0}cv-fv\right\} d\Omega _{h}(t)+\displaystyle\int_{\Omega _{h}(t)}%
\frac{Dv}{Dt}cd\Omega _{h}(t),%
\end{array}%
\right.
\end{equation*}%
it follows that (\ref{s2:6}) can also be written as%
\begin{equation}
\frac{d}{dt}\int_{\Omega _{h}(t)}cvd\Omega _{h}(t)+\int_{\Omega
_{h}(t)}\left( \mu \nabla _{y}c\cdot \nabla _{y}v+a_{0}cv-fv\right) d\Omega
_{h}(t)-\int_{\Omega _{h}(t)}c\frac{Dv}{Dt}d\Omega _{h}(t)=0.  \label{s2:8}
\end{equation}%
Hence, it is clear that (\ref{s2:7}) and (\ref{s2:8}) are equivalent. If we
identify the Hilbert spaces $Y(t)$ and $Y(t_{n})$ with the spaces $%
H^{1}(\Omega _{h}(t))$ and $H^{1}(\Omega )$ respectively, then
(\ref{s2:8}) is a weak formulation of (\ref{int1}) in the intervals
$[t_{n-q},t_{n}]$ provided that the initial condition $c^{0}\in
L^{2}(\Omega )$ and $f\in L^{2}(L^{2}(\Omega ))$. This weak
formulation satisfies the global mass property because if for $t\in
\lbrack t_{n-q},t_{n}]$ we let $v=1$, then, noting that
$\int_{\Omega
_{h}(t)}c(y,t)d\Omega _{h}(t)=\int_{\Omega }c(x,t)d\Omega $, from (\ref{s2:8}%
) it readily follows (\ref{int1.1}). Next, we describe the methods to
integrate (\ref{s2:8}) termed the NCLG methods. To this end, we
assume for the moment that the following properties hold:

(1) Given the family of exact, conforming, quasi-uniformly regular
partitions $\ \left( \mathcal{T}_{h}\right) _{h}$ of $\Omega $ composed of
curved $d-$simplices $T_{l}$ of order $k$, let $\widehat{T}$ be the reference
unit $d-$simplex and let the mapping $F_{l}:\widehat{T}\rightarrow T_{l}$ be
a diffeomorphism of class $C^{k+1}$; then, we can define the family of $%
H^{1}-$conforming finite element spaces
\begin{equation*}
\left\{
\begin{array}{l}
V_{h}:=\left\{ v_{h}\in C(\overline{\Omega }):\forall T_{l}\in \mathcal{T}%
_{h}\ \left. v_{h}\right\vert _{T_{l}}\in P(T_{l}),\ 1\leq l\leq NE\right\} ,
\\
\\
P(T_{l}):=\left\{ p(x):p(x)=\widehat{p}\circ F_{l}^{-1}(x),\ \widehat{p}\in
\widehat{P}_{k}(\widehat{T})\right\} ,%
\end{array}%
\right.
\end{equation*}
where $\left. v_{h}\right\vert _{T_{l}}$ denotes the restriction of $v_{h}$
to the element $T_{l}$, $\widehat{P}_{k}(\widehat{T})$ stands for the set of
polynomials of degree $\leq k,\ k$ being an integer $\geq 1$, defined on the
element of reference $\widehat{T}$. The $M-$dimensional set of global basis
functions of $V_{h}$ is denoted by $\left\{ \chi _{j}(x)\right\} _{j=1}^{M}$%
. Notice that $V_{h}\subset H^{1}(\Omega )$ because the family of partitions
$\left( \mathcal{T}_{h}\right) _{h}$ is exact.

(2) For all $n\geq q$, and $\forall t\in \lbrack t_{n-q},t_{n})$, the family
of partitions $\left( \mathcal{T}_{h}(t)\right) _{h}$, composed of curved $%
d- $simplices $T_{l}(t)$ of order $k$, is conforming, quasi-uniformly
regular and complete with respect to the domain $\Omega _{h}(t)$, and the
mapping $F_{l}(t):\widehat{T}\rightarrow T_{l}(t)$ is a diffeomorphism of
class $C^{k,1}$. Then, we can define the family of $H^{1}-$conforming finite
element spaces%
\begin{equation*}
\left\{
\begin{array}{l}
V_{h}(t):=\left\{ v_{h}\in C(\overline{\Omega }_{h}(t)):\forall T_{l}(t)\in
\mathcal{T}_{h}(t)\ \left. v_{h}\right\vert _{T_{l}(t)}\in P(T_{l}(t)),\
1\leq l\leq NE\right\} , \\
\\
P(T_{l}(t)):=\left\{ p(y):p(y)=\widehat{p}\circ F_{l}^{-1}(t)(y),\ \widehat{p%
}\in \widehat{P}_{k}(\widehat{T})\right\} .%
\end{array}%
\right.
\end{equation*}
Setting $y=X(x,t_{n};t)$, the $M-$dimensional set of global basis functions
of $V_{h}(t)$ is the set $\left\{ \chi _{j}(X(x,t_{n};t)\right\} _{j=1}^{M}$%
. Also, notice that $V_{h}(t)\subset H^{1}(\Omega _{h}(t))$. Since $\chi
_{j}(X(x,t_{n};t))=\phi _{t_{n}}^{t}\chi _{j}(x)$, then by virtue of (\ref%
{s2:5})
\begin{equation}
\frac{D\chi _{j}(X(x,t_{n};t))}{Dt}=\phi _{t_{n}}^{t}\frac{d}{dt}\left( \chi
_{j}(x)\right) =0\Longrightarrow \forall t\in \left[ t_{n-q},t_{n}\right] ,\
\chi _{j}(X(x,t_{n};t))=\chi _{j}(x).  \label{s2:9}
\end{equation}%
Now, we can then apply the Galerkin method to (\ref{s2:8}) in order to
calculate $c_{h}(t)\in V_{h}(t)$ by choosing test functions $v_{h}(t)\in
V_{h}(t)$ that satisfy $\frac{Dv_{h}}{Dt}=0$. Thus, for $t\in I_{n}$ and for
each $x\in \overline{\Omega }$, let $X(x,t_{n};t)\in \overline{\Omega}%
_{h}(t) $, we set
\begin{equation}
v_{h}(t)=v_{h}(X(x,t_{n};t)):=\sum_{j=1}^{M}V_{j}\chi _{j}(X(x,t_{n};t)).\
\label{s2:92}
\end{equation}%
Note that in view of (\ref{s2:9}) it follows that for $t\in I_{n},\
v_{h}(t)=v_{h}(t_{n})=\sum_{j=1}^{M}V_{j}\chi _{j}(x)\in V_{h}$. The
Galerkin formulation for (\ref{s2:8}) yields
\begin{equation}
\frac{d}{dt}\int_{\Omega _{h}(t)}c_{h}v_{h}d\Omega _{h}(t)+\int_{\Omega
_{h}(t)}\left( \mu \nabla _{y}c_{h}\cdot \nabla
_{y}v_{h}+a_{0}c_{h}v_{h}-fv_{h}\right) d\Omega _{h}(t)=0.  \label{s2:10}
\end{equation}
To simplify the writing of the formulas that appear in the sequel, we
introduce the bilinear form $a:H^{1}(\Omega )\times H^{1}(\Omega
)\rightarrow \mathbb{R}$
\begin{equation}
a(u,v):=\int_{\Omega }\left( \mu \nabla u\cdot \nabla v+a_{0}uv\right)
d\Omega .  \label{se2:101}
\end{equation}
$a(\cdot ,\cdot )$ is symmetric, continuous and coercive. The BDF-q
discretization of (\ref{s2:10}) reads as follows. Assuming we are given the
starting approximations $c_{h}^{0},\ldots ,c_{h}^{q-1}\in V_{h}$, for each $%
n,\ q\leq n\leq N$, we calculate $c_{h}^{n}$ $\in V_{h}$ as the solution of
the equation
\begin{equation}
\frac{1}{\Delta t}\sum_{i=0}^{q}\alpha _{i}\left(
c_{h}^{n-i},v_{h}^{n-i}\right) _{\Omega
_{h}(t_{n-i})}+a(c_{h}^{n},v_{h})=(f^{n},v_{h})_{\Omega }\ \ \forall
v_{h}\in V_{h}.  \label{s2:11}
\end{equation}
Here, we employ the following notations:
\begin{equation}
\left\{
\begin{array}{l}
v_{h}^{n-i}:=v_{h}(X(x,t_{n};t_{n-i}))=\sum_{j=1}^{M}V_{j}\chi
_{j}(X(x,t_{n};t_{n-i})), \\
\\
\left( c_{h}^{n-i},v_{h}^{n-i}\right) _{\Omega _{h}(t_{n-i})}:=\displaystyle%
\int_{\Omega _{h}(t_{n-i})}c_{h}^{n-i}(X(x,t_{n};t_{n-i}))v_{h}^{n-i}d\Omega
_{h}(t_{n-i}),%
\end{array}%
\right.  \label{s2:12}
\end{equation}%
$c_{h}^{n-i}(X(x,t_{n};t_{n-i}))$ denotes the value that $c_{h}^{n-i}\in
V_{h}$ takes at the point $X(x,t_{n};t_{n-i})$.

\begin{remark}
\label{initval} Setting $c_{h}^{0}=c^{0}$, the choice of the $q-1$
values $c_{h}^{1},\ldots ,c_{h}^{q-1}$ corresponds to applying
explicit Runge-Kutta schemes of order $O(\Delta t^{q-1})$ to the
finite element formulation of problem (\ref{int1}) with time step
$\Delta t/2$ since it is only a small number of time steps.
\end{remark}

\begin{remark}
\label{tabata} If in each interval $I_{n}$ one applies the Galerkin method
combined with the BDF-q scheme to (\ref{s2:7}), and takes into account that:
$J^{n,n}:=J(x,t_{n},t_{n})=1$, $\widehat{c}_{h}^{n}:=c_{h}^{n}(x)$,\ $\widehat{c}
_{h}^{n-i}:=c_{h}^{n-i}(X(x,t_{n};t_{n-i}))$ and the Jacobian matrix $F$
is the identity matrix when $t=t_{n}$, it results the equation%
\begin{equation}
\frac{1}{\Delta t}\sum_{i=0}^{q}\alpha _{i}\left( J^{n,n-i}\widehat{c}%
_{h}^{n-i},v_{h}\right) _{\Omega }+a(c_{h}^{n},v_{h})=(f^{n},v_{h})_{\Omega
}\ \ \forall v_{h}\in V_{h}.  \label{s2:13}
\end{equation}%
This is the formulation used by \cite{RT2}, when $q=1$, and \cite{FK} when $%
q=2$.
\end{remark}

\section{On the partitions $\mathcal{T}_{h}$ and$\ \mathcal{T}_{h}(t)$ and
their associated finite element spaces $V_{h}$ and $V_{h}(t)$ respectively}

In this section we study some properties of the partitions $\mathcal{T}_{h}$
and $\mathcal{T}_{h}(t)$ and characterize their respective associated finite
element spaces $V_{h}$ and $V_{h}(t)$. We also describe the procedure to
approximate the integrals $\left( c_{h}^{n-i},v_{h}^{n-i}\right) _{\Omega
(t_{n-i})}$ of (\ref{s2:11}). As we said in Section 2, at time $t=0$, we
construct a family of partitions $(\mathcal{T}_{h})_{h}$ on the bounded
region $\Omega \subset \mathbb{R}^{d}$ with curved piecewise smooth boundary
$\partial \Omega $. The partitions consist of $d-$ simplices such that
\begin{equation*}
\mathcal{T}_{h}=\left\{ T_{l}\right\} _{l=1}^{NE}\text{ \ \textrm{and}\ \ }%
\overline{\Omega }=\bigcup\limits_{l=1}^{NE}T_{l}.
\end{equation*}%
The partitions are said to be exact because $\overline{\Omega }
=\bigcup\limits_{l=1}^{NE}T_{l}$. The elements that have at most one vertex
on $\partial \Omega $ are known as interior elements and are straight, the
other, known as boundary elements, share either a $(d-1)-$face with $%
\partial \Omega $ or and edge. See, for instance, \cite{Bern} and \cite{ER1} on the construction of
exact \ partitions on domains with curved piecewise smooth
boundaries. In this paper, we allow the existence of curved elements
in the sense defined in \cite{Bern} (see Appendix: Definitions
\ref{definition1} and \ref{definition1.5} and Lemma
\ref{lemma1_apen}
); thus, if $T_{l}$ is a curved $d-$simplex of class $C^{k+1}$ and $%
\widehat{T}$ is the unit reference $d-$simplex, then there exists a mapping $%
F_{l}:\widehat{T}\rightarrow T_{l}$ which is a $C^{k+1}$- diffeomorphism and
such that
\begin{equation}
F_{l}=\overline{F}_{l}+\Theta _{l},  \label{se2:81}
\end{equation}%
where $\overline{F}_{l}$ is an invertible affine mapping from $\widehat{T}$
onto the straight $d-$simplex $\overline{T}_{l}$ the vertices of which are
the vertices $\{x_{j}^{(l)}\}_{j=1}^{d+1}$ of the element $T_{l}$, see
Figure 1, thus,

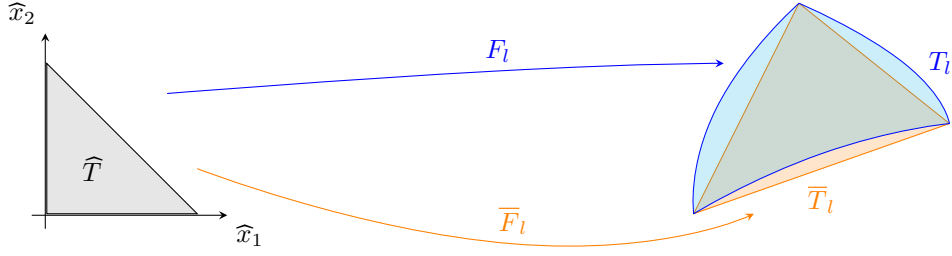
\begin{figure}[ht]
    \centering
    \begin{tikzpicture}[scale=2.0, >=stealth]

        \draw[black,fill=black!10!white!]   (0,0) -- (1,0) -- (0,1) -- (0,0);
        \draw[->] (-0.1,-0.01) -- (1.2,-0.01) node[below right]{$\widehat{x}_1$};
        \draw[->] (-0.01,-0.1) -- (-0.01,1.2) node[above left]{$\widehat{x}_2$};

        \coordinate (v1) at (6,0.6);
        \coordinate (v2) at (5,1.4);
        \coordinate (v3) at (4.3,0.0);
        \coordinate (v4) at ($0.5*(v1)+0.5*(v2)+(0.2,0.0)$);
        \coordinate (v5) at ($0.5*(v2)+0.5*(v3)+(-0.2,0.0)$);
        \coordinate (v6) at ($0.5*(v3)+0.5*(v1)+(0.0,0.1)$);
        \coordinate (vbar) at ($0.333*(v1)+0.333*(v2)+0.333*(v3)$);

        \draw[orange, fill=orange, fill opacity=0.2] (v1) -- (v2) -- (v3) -- (v1);

        \draw[blue] plot[smooth, tension=1] coordinates{(v1) (v4) (v2)}
        plot[smooth, tension=1] coordinates{(v2) (v5) (v3)}
        plot[smooth, tension=1] coordinates{(v3) (v6) (v1)};
        \fill[cyan, fill opacity=0.2] (vbar) -- plot[smooth, tension=1] coordinates{(v1) (v4) (v2)}
        (vbar) -- plot[smooth, tension=1] coordinates{(v2) (v5) (v3)}
        (vbar) -- plot[smooth, tension=1] coordinates{(v3) (v6) (v1)};

        \node[below=0.1cm]() at ($0.5*(v3)+0.5*(v1)$) {\textcolor{orange}{$\overline{T}_l$}};
        \node[right=0.2cm]() at (v4) {\textcolor{blue}{$T_l$}};
        \node[]() at (0.3,0.3) {\textcolor{black}{$\widehat{T}$}};

        \draw[blue,->] plot[smooth, tension=1] coordinates{(0.8,0.8) (3, 0.95) (4.5,1.0)};
        \node[above]() at (3,0.95) {\textcolor{blue}{$F_l$}};
        \draw[orange,->] plot[smooth, tension=1] coordinates{(1.0,0.3) (3.1, -0.2) (4.7,0.0)};
        \node[above]() at (3.1,-0.2) {\textcolor{orange}{$\overline{F}_l$}};

    \end{tikzpicture}
    \caption {\ The straight simplex $\overline{T}_{l}$ is the affine image of the
simplex $\widehat{T}$, whereas the blue curved simplex
$T_{l}=F_{l}(\widehat{T})$.}
\end{figure}
 \label{fig1}

\begin{equation}
\overline{F}_{l}(\widehat{x})=\overline{B}_{l}\widehat{x}+\overline{a}_{l},\
\ \overline{B}_{l}\in \mathcal{L}(\mathbb{R}^{d})\ \ \mathrm{and\ \ }%
\overline{a}_{l}\in \mathbb{R}^{d}\text{;}  \label{se2:82}
\end{equation}%
and\ $\Theta _{l}:\widehat{T}\rightarrow \mathbb{R}^{d}$ is a $C^{k+1}$
mapping that satisfies the bound%
\begin{equation}
c_{_{T_{l}}}=\sup_{\widehat{x}}\left\Vert D\Theta _{l}(\widehat{x})\overline{%
B}_{l}^{-1}\right\Vert <1.  \label{se2:83}
\end{equation}%
Moreover, for $2\leq s\leq k+1$, there are constants $c_{s}$ and $c_{-s}$,
the latter depending continuously on $c_{_{T_{l}}},\ c_{2}(T_{l}),\ldots
,c_{k+1}(T_{l})$, such that
\begin{equation}
\left\{
\begin{array}{l}
c_{s}(T_{l})=\sup_{\widehat{x}\in \widehat{T}}\left\Vert D^{s}F_{l}(\widehat{%
x})\right\Vert \cdot \left\Vert \overline{B}_{l}\right\Vert ^{-s}, \\
\\
\sup_{x\in T}\left\Vert D^{s}F_{l}^{-1}(x)\right\Vert \leq c_{-s}\left\Vert
\overline{B}_{l}\right\Vert ^{2(s-1)}\left\Vert \overline{B}%
_{l}^{-1}\right\Vert ^{s}.%
\end{array}%
\right.  \label{se2:84}
\end{equation}
We further assume that the curved finite elements $(T_{l},P(T_{l}),\Sigma
_{T_{l}})$ associated with the family of partitions $\left( \mathcal{T}%
_{h}\right) _{h}$ are $H^{1}-$conforming (see Appendix: Definitions \ref
{definition4} and \ref{definition5}); so, we consider the $H^{1}$-conforming
finite element space $V_{h}\subset H^{1}(\Omega )$ associated with $\mathcal{%
T}_{h}$, i.e.,
\begin{equation}
\left\{
\begin{array}{l}
V_{h}:=\left\{ v_{h}\in C(\overline{\Omega }):\forall T_{l}\in \mathcal{T}%
_{h}\ \left. v_{h}\right\vert _{T_{l}}\in P(T_{l}),\ 1\leq l\leq NE\right\} ,
\\
\\
P(T_{l}):=\left\{ p(x):p(x)=\widehat{p}\circ F_{l}^{-1}(x),\ \widehat{p}\in
\widehat{P}_{k}(\widehat{T})\right\} ,%
\end{array}%
\right.  \label{se2:85}
\end{equation}
where $\left. v_{h}\right\vert _{T_{l}}$ denotes the restriction of $v_{h}$
to the element $T_{l}$, $\widehat{P}_{k}(\widehat{T})$ stands for the set of
polynomials of degree $\leq k,\ k$ being an integer $\geq 1$, defined on the
element of reference $\widehat{T}$. The $M-$dimensional set of global basis
functions of $V_{h}$ is denoted by $\left\{ \chi _{j}(x)\right\} _{j=1}^{M}$%
. Note that if $T_{l}$ is a straight $d-$simplex, then $\Theta _{l}(\widehat{%
x})\in \widehat{P}_{k}(\widehat{T})$ and $F_{l}$ is thus the standard
regular isoparametric transformation of type$-(k)$ as defined in \cite{Ci}.
The approximation properties of the finite element space $V_{h}$ are stated
in Theorems \ref{theorem3} and \ref{corollary1} of the Appendix.

We now turn to study the partitions $\mathcal{T}_{h}(t)$ when $t\in I_{n}$
and $q\leq n\leq N$. As stated above, the partitions $\mathcal{T}_{h}(t)$
are composed of elements $T_{l}(t)$ with curved $d-$ faces, such that $%
T_{l}(t):=X(T_{l},t_{n};t)$, $T_{l}\in \mathcal{T}_{h}$, and by virtue of
Lemma \ref{solX} for each $T_{l}(t)$ there is one and only one $T_{l}$. In
general, $T_{l}(t)$ may intersect several elements $T_{l}$ of the fixed
partition $\mathcal{T}_{h}$. See Figure 2.

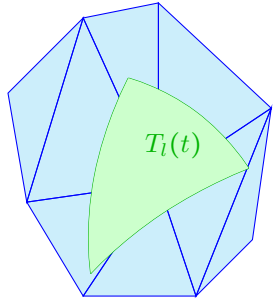
\begin{figure}[h]
    \centering
    \begin{tikzpicture}[scale=1.0, >=stealth]

        \begin{scope}[xshift=-1cm, yshift=-1cm]
        \coordinate (N1) at (0,1.2);
        \coordinate (N2) at (1,2.2);
        \coordinate (N3) at (2,2.4);
        \coordinate (N4) at (0.25,-0.25);
        \coordinate (N5) at (2.0,0.0);
        \coordinate (N6) at (3.5,1.0);
        \coordinate (N7) at (1.0,-1.5);
        \coordinate (N8) at (2.5,-1.5);
        \coordinate (N9) at (3.25,-0.75);
        \end{scope}
        \draw[blue,fill=cyan,fill opacity=0.2] (N1) -- (N4) -- (N2) -- cycle;
        \draw[blue,fill=cyan,fill opacity=0.2] (N2) -- (N4) -- (N5) -- cycle;
        \draw[blue,fill=cyan,fill opacity=0.2] (N5) -- (N3) -- (N2) -- cycle;
        \draw[blue,fill=cyan,fill opacity=0.2] (N5) -- (N6) -- (N3) -- cycle;
        \draw[blue,fill=cyan,fill opacity=0.2] (N4) -- (N7) -- (N5) -- cycle;
        \draw[blue,fill=cyan,fill opacity=0.2] (N5) -- (N7) -- (N8) -- cycle;
        \draw[blue,fill=cyan,fill opacity=0.2] (N5) -- (N8) -- (N6) -- cycle;
        \draw[blue,fill=cyan,fill opacity=0.2] (N6) -- (N8) -- (N9) -- cycle;

        \coordinate (w1) at (2.2,-0.8);
        \coordinate (w2) at (0.6,0.4);
        \coordinate (w3) at (0.1,-2.2);
        \coordinate (w4) at ($0.5*(w1)+0.5*(w2)+(0.1,0.15)$);
        \coordinate (w5) at ($0.5*(w2)+0.5*(w3)+(-0.2,0.0)$);
        \coordinate (w6) at ($0.5*(w3)+0.5*(w1)+(-0.1,0.1)$);
        \coordinate (wbar) at ($0.333*(w1)+0.333*(w2)+0.333*(w3)$);
        \draw[darkgreen] plot[smooth, tension=1] coordinates{(w1) (w4) (w2)}
        plot[smooth, tension=1] coordinates{(w2) (w5) (w3)}
        plot[smooth, tension=1] coordinates{(w3) (w6) (w1)};
        \fill[green!20!white!, fill opacity=1.0] (wbar) -- plot[smooth, tension=1] coordinates{(w1) (w4) (w2)}
        (wbar) -- plot[smooth, tension=1] coordinates{(w2) (w5) (w3)}
        (wbar) -- plot[smooth, tension=1] coordinates{(w3) (w6) (w1)};
        \node[]() at (1.2,-0.5) {\textcolor{darkgreen}{$T_{l}(t)$}};
    \end{tikzpicture}
    \caption{{\ The green curved element $T_{l}(t) \in \mathcal{T}_{h}(t)$ may intersect several elements (blue) of the
fixed partition $\mathcal{T}_{h}$.}}
\end{figure}
\label{Fig2}

Next, by virtue of Lemma \ref{solX}, we define a family of quasi-isometric
mappings $F_{l}(t):\widehat{T}\rightarrow T_{l}(t)$ such that for each $\widehat{x}\in \widehat{T}$ there is one and only one $y$ in $T_{l}(t)$ defined by
\begin{equation}
\widehat{x}\rightarrow y=F_{l}(t)(\widehat{x})=X(\cdot ,t_{n};t)\circ F_{l}(%
\widehat{x})\in T_{l}(t),  \label{se3:1}
\end{equation}
where $F_{l}$ is the mapping defined in (\ref{se2:81}). We will show below
that $F_{l}(t)$ is a mapping of class $C^{k,1}$, see Figure \ref{Fig3}
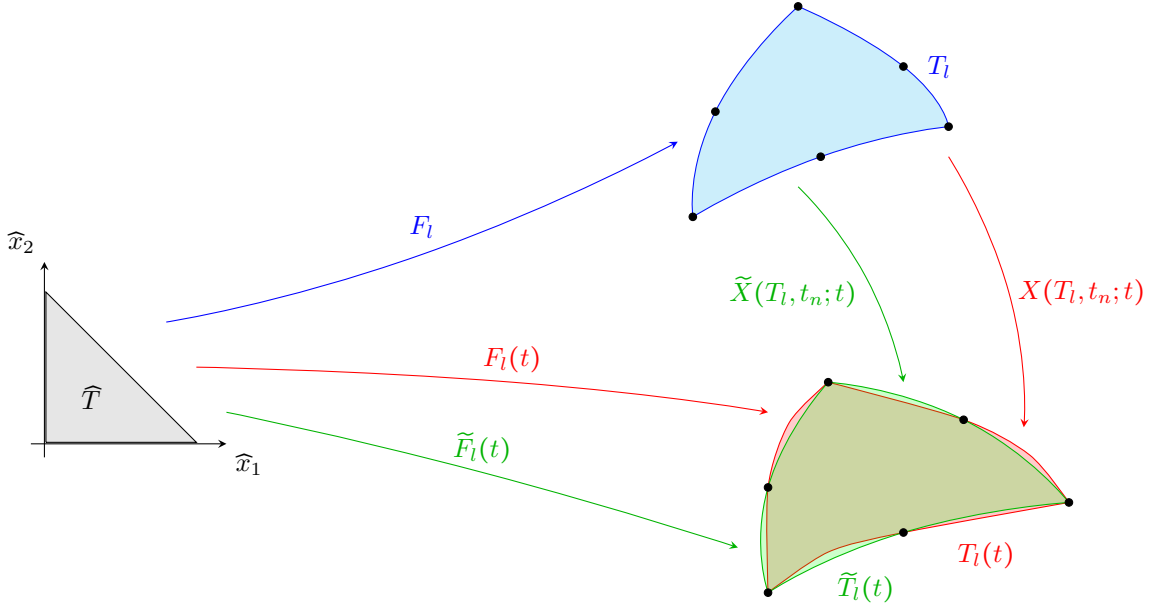
\begin{figure}[ht]
    \centering
    \begin{tikzpicture}[scale=2.0, >=stealth]

        \draw[black,fill=black!10!white!]   (0,0) -- (1,0) -- (0,1) -- (0,0);
        \draw[->] (-0.1,-0.01) -- (1.2,-0.01) node[below right]{$\widehat{x}_1$};
        \draw[->] (-0.01,-0.1) -- (-0.01,1.2) node[above left]{$\widehat{x}_2$};

        \coordinate (v1) at (6,2.1);
        \coordinate (v2) at (5,2.9);
        \coordinate (v3) at (4.3,1.5);
        \coordinate (v4) at ($0.5*(v1)+0.5*(v2)+(0.2,0.0)$);
        \coordinate (v5) at ($0.5*(v2)+0.5*(v3)+(-0.2,0.0)$);
        \coordinate (v6) at ($0.5*(v3)+0.5*(v1)+(0.0,0.1)$);
        \coordinate (vbar) at ($0.333*(v1)+0.333*(v2)+0.333*(v3)$);
        \draw[blue] plot[smooth, tension=1] coordinates{(v1) (v4) (v2)}
        plot[smooth, tension=1] coordinates{(v2) (v5) (v3)}
        plot[smooth, tension=1] coordinates{(v3) (v6) (v1)};
        \fill[cyan, fill opacity=0.2] (vbar) -- plot[smooth, tension=1] coordinates{(v1) (v4) (v2)}
        (vbar) -- plot[smooth, tension=1] coordinates{(v2) (v5) (v3)}
        (vbar) -- plot[smooth, tension=1] coordinates{(v3) (v6) (v1)};

        \coordinate (w1) at (6.8,-0.4);
        \coordinate (w2) at (5.2,0.4);
        \coordinate (w3) at (4.8,-1.0);
        \coordinate (w4) at ($0.5*(w1)+0.5*(w2)+(0.1,0.15)$);
        \coordinate (w5) at ($0.5*(w2)+0.5*(w3)+(-0.2,0.0)$);
        \coordinate (w6) at ($0.5*(w3)+0.5*(w1)+(-0.1,0.1)$);
        \coordinate (w14) at ($0.5*(w1)+0.5*(w4)+(0.09,0.04)$);
        \coordinate (w25) at ($0.5*(w2)+0.5*(w5)+(-0.06,0.08)$);
        \coordinate (w36) at ($0.5*(w3)+0.5*(w6)+(-0.05,0.07)$);
        \coordinate (wbar) at ($0.333*(w1)+0.333*(w2)+0.333*(w3)$);
        \draw[red] plot[smooth, tension=0.5] coordinates{(w1) (w14) (w4) (w2)}
        plot[smooth, tension=0.5] coordinates{(w2) (w25) (w5) (w3)}
        plot[smooth, tension=0.5] coordinates{(w3) (w36) (w6) (w1)};
        \fill[red, fill opacity=0.2] (wbar) -- plot[smooth, tension=0.5] coordinates{(w1) (w14) (w4) (w2)}
        (wbar) -- plot[smooth, tension=0.5] coordinates{(w2) (w25) (w5) (w3)}
        (wbar) -- plot[smooth, tension=0.5] coordinates{(w3) (w36) (w6) (w1)};

        \draw[darkgreen] plot[smooth, tension=1] coordinates{(w1) (w4) (w2)}
        plot[smooth, tension=1] coordinates{(w2) (w5) (w3)}
        plot[smooth, tension=1] coordinates{(w3) (w6) (w1)};
        \fill[green, fill opacity=0.2] (wbar) -- plot[smooth, tension=1] coordinates{(w1) (w4) (w2)}
        (wbar) -- plot[smooth, tension=1] coordinates{(w2) (w5) (w3)}
        (wbar) -- plot[smooth, tension=1] coordinates{(w3) (w6) (w1)};

        \node[draw,circle,fill=black,inner sep=0pt,minimum size=0.10cm,label=right:]() at (v1) {};
        \node[draw,circle,fill=black,inner sep=0pt,minimum size=0.10cm,label=right:]() at (v2) {};
        \node[draw,circle,fill=black,inner sep=0pt,minimum size=0.10cm,label=right:]() at (v3) {};
        \node[draw,circle,fill=black,inner sep=0pt,minimum size=0.10cm,label=right:]() at (v4) {};
        \node[draw,circle,fill=black,inner sep=0pt,minimum size=0.10cm,label=right:]() at (v5) {};
        \node[draw,circle,fill=black,inner sep=0pt,minimum size=0.10cm,label=right:]() at (v6) {};
        \node[draw,circle,fill=black,inner sep=0pt,minimum size=0.10cm,label=right:]() at (w1) {};
        \node[draw,circle,fill=black,inner sep=0pt,minimum size=0.10cm,label=right:]() at (w2) {};
        \node[draw,circle,fill=black,inner sep=0pt,minimum size=0.10cm,label=right:]() at (w3) {};
        \node[draw,circle,fill=black,inner sep=0pt,minimum size=0.10cm,label=right:]() at (w4) {};
        \node[draw,circle,fill=black,inner sep=0pt,minimum size=0.10cm,label=right:]() at (w5) {};
        \node[draw,circle,fill=black,inner sep=0pt,minimum size=0.10cm,label=right:]() at (w6) {};
        %

        \node[right=0.2cm]() at (v4) {\textcolor{blue}{$T_l$}};
        \node[]() at (0.3,0.3) {\textcolor{black}{$\widehat{T}$}};
        \node[below=0.2cm]() at ($0.5*(w1)+0.5*(w6)$) {\textcolor{red}{$T_l(t)$}};
        \node[below=0.0cm]() at ($0.5*(w3)+0.5*(w6)+(0.2,0)$) {\textcolor{darkgreen}{$\widetilde{T}_l(t)$}};

        \draw[blue,->] plot[smooth, tension=1] coordinates{(0.8,0.8) (2.5, 1.25) (4.2,2.0)};
        \node[above=0.1cm]() at (2.5,1.25) {\textcolor{blue}{$F_l$}};
        \draw[red,->] plot[smooth, tension=1] coordinates{(1.0,0.5) (3.1, 0.4) (4.8,0.2)};
        \node[above]() at (3.1,0.4) {\textcolor{red}{$F_l(t)$}};
        \draw[darkgreen,->] plot[smooth, tension=1] coordinates{(1.2,0.2) (2.9, -0.2) (4.6,-0.7)};
        \node[above]() at (2.9,-0.2) {\textcolor{darkgreen}{$\widetilde{F}_l(t)$}};

        \draw[darkgreen,->] plot[smooth, tension=1] coordinates{(5,1.7) (5.45,1.1) (5.7,0.4)};
        \node[left]() at (5.45,1.0) {\textcolor{darkgreen}{$\widetilde{X}(T_l, t_n; t)$}};
        \draw[red,->] plot[smooth, tension=1] coordinates{(6,1.9) (6.4,1.0) (6.5,0.1)};
        \node[right]() at (6.4,1.0) {\textcolor{red}{$X(T_l, t_n; t)$}};

    \end{tikzpicture}
    \caption{{\ The blue simplex $T_{l}$ belongs to fixed mesh
    $\mathcal{T}_{h}$; the red simplex $T_{l}(t)$ belong to the mesh
    $\mathcal{T}_{h}(t)$ and is the image of the blue simplex
    $T_{l}$ by the mapping $X(x,t_{n};t)$; the green simplex
    $\widetilde{T}_{l}(t)$ is the isoparametric simplex that
    approximates the red element $T_{l}(t)$.}}
\end{figure}
 \label{Fig3}. In addition to the family
of partitions $(\mathcal{T}_{h}(t))_{h}$, we shall also work with another
family of partitions denoted by $(\widetilde{\mathcal{T}}_{h}(t))_{h}$, the
elements of which are isoparametric simplices of type $\mathrm{(k)}$, $\widetilde{T}_{l}(t)$, which are constructed as follows, see Figure \ref{Fig3}. For $1\leq l\leq NE$ and $t\in \lbrack t_{n-q},t_{n})$, let the
mapping $\widetilde{F}_{l}(t):\widehat{T}\rightarrow \mathbb{R}^{d}$ be
defined as
\begin{equation}
\begin{array}{c}
\widehat{x}\rightarrow \widetilde{y}(t)=\widetilde{F}_{l}(t)(\widehat{x}%
)=\sum_{j=1}^{NN}F_{l}(t)(\widehat{x}_{j})\widehat{\chi }_{j}(\widehat{x})\
\ \left( \mathrm{by\ virtue\ of\ (\ref{se3:1})}\right) \\
\\
\ \ \ \ \ \ \ \ \ \ \ \ \ \ \ =\sum_{j=1}^{NN}X(\cdot ,t_{n};t)\circ F_{l}(\widehat{x}_{j})\widehat{\chi }
_{j}(\widehat{x}),%
\end{array}
\label{se3:5}
\end{equation}
where\ $\left\{ \widehat{x}_{j}\right\} _{j=1}^{NN}$ and $\left\{ \widehat{%
\chi }_{j}(\widehat{x})\right\} _{j=1}^{NN}$ are the nodes of the element of
reference $\widehat{T}$ and the set of basis functions of $\widehat{P}_{k}(%
\widehat{T})$, respectively. Noticing that $X(x_{j}^{(l)},t_{n};t)=X(\cdot
,t_{n};t)\circ F_{l}(\widehat{x}_{j})$, where $x_{j}^{(l)}$ is the jth node
of the element $T_{l}$, then the points, $X(x_{j}^{(l)},t_{n};t)$, the
images of the nodes of $T_{l}$ by the mapping $X(x,t_{n};t)$, are the nodes
of $\widetilde{T}_{l}(t)$. Hence, if we denote by $\widehat{I}_{k}\ $the
Lagrange interpolant of degree $k$ in the space $C(\widehat{T})$ of
continuous functions defined on $\widehat{T}$, we can set that for all $t\in
\lbrack t_{n-q},t_{n})$ and $1\leq l\leq NE$,
\begin{equation}
\widetilde{F}_{l}(t)(\widehat{x})=\widehat{I}_{k}F_{l}(t)(\widehat{x})\ \
\mathrm{and\ \ }\widetilde{T}_{l}(t)=\widetilde{F}_{l}(t)(\widehat{T}).
\label{se3:6}
\end{equation}%
Notice that
\begin{equation*}
\widetilde{\mathcal{T}}_{h}(t)=\left\{ \widetilde{T}_{l}(t)\right\}
_{l=1}^{NE}\text{, \ \ \ }\overline{\widetilde{\Omega }}_{h}(t)=\bigcup%
\limits_{l=1}^{NE}\widetilde{T}_{l}(t).
\end{equation*}
Making use of both (\ref{int3}) and (\ref{se3:1}) we can also express $F_{l}(t)(\widehat{x})$ as
\begin{equation}
F_{l}(t)(\widehat{x})=F_{l}(\widehat{x})-\int_{t}^{t_{n}}u(X(F_{l}(\widehat{x%
}),t_{n};\tau ),\tau )d\tau .  \label{se3:7}
\end{equation}
Following the approach of \cite{Bern} and (\cite{Ci}, Chapter 4.3) for the
theory of isoparametric elements, we shall decompose $\widetilde{F}_{l}(t)(\widehat{x})$ as
\begin{equation}
\widetilde{F}_{l}(t)(\widehat{x})=\overline{F}_{l}(t)(\widehat{x})+%
\widetilde{\Theta }_{l}(t)(\widehat{x}),  \label{se3:9}
\end{equation}%
where $\overline{F}_{l}(t)$ is an invertible affine mapping from $\widehat{T}
$ onto the straight $d-$simplex $\overline{T}_{l}(t)$ the vertices of which
are the points $X(x_{j}^{(l)},t_{n};t),\ 1\leq j\leq d+1,$ $x_{j}^{(l)}$ being the
vertices of the element $T_{l}$; thus,

\begin{equation}
\overline{F}_{l}(t)(\widehat{x})=\overline{B}_{l}(t)\widehat{x}+\overline{b}%
_{l}(t),\ \ \overline{B}_{l}(t)\in \mathcal{L}(\mathbb{R}^{d})\ \mathrm{and\
}\overline{b}_{l}(t)\in \mathbb{R}^{d}\text{,}  \label{se3:10}
\end{equation}%
and\ $\widetilde{\Theta }_{l}(t):\widehat{T}\rightarrow \mathbb{R}^{d}$ is a
$C^{k+1}$ mapping. Similarly, we shall also express the mapping $F_{l}(t)$
as
\begin{equation}
F_{l}(t)=\overline{F}_{l}(t)+\Phi _{l}(t),  \label{se3:11}
\end{equation}%
where\ $\Phi _{l}(t):\widehat{T}\rightarrow \mathbb{R}^{d}$ is a $C^{k,1}$
mapping. Notice that by construction, $\overline{F}_{l}(t)=\widehat{I}%
_{1}F_{l}(t)$, $\widehat{I}_{1}:C(\widehat{T})\rightarrow \widehat{P}_{1}(%
\widehat{T})$ being the Lagrange interpolant of degree 1. These
decompositions imply that under certain conditions to be stated below, one
can consider that both the curved $d$-simplex $T_{l}(t)$ and the
isoparametric $d$-simplex $\widetilde{T}_{l}(t)$ may be viewed as
perturbations of the straight $d$-simplex\ $\overline{T}_{l}(t)$. Likewise, $%
\widetilde{T}_{l}(t)$ can be considered as an approximation to $T_{l}(t)$.

In the next series of lemmas we shall establish the conditions under
which the partitions $\mathcal{T}_{h}(t)$ are conforming and
quasi-uniformly regular of order $k$, see Definition
\ref{definition2.5} of the Appendix, so that we associate with them
$H^{1}-$conforming finite elements spaces $V_{h}(t)$ which have the
approximation properties stated in Theorems \ref{theorem3} and
\ref{corollary1}.

\begin{lemma}
\label{lTht1} Let the family $\left( \mathcal{T}_{h}\right) _{h}$ be
quasi-uniformly regular of order $k$, and let $\overline{h}_{l}(t)$ and $
\overline{\kappa }_{l}(t)$ denote the lengths of the largest and smallest
edges of the straight $d-$simplex $\overline{T}_{l}(t)$, and $\overline{h}%
(t):=\max_{l}(\overline{h}_{l}(t))$. If $(t_{n}-t)\left\vert u\right\vert_{C(%
\left[ 0,T\right] ;\ W^{1,\infty }(\Omega )^{d})}$ is sufficiently small
for $t\in [t_{n-q},t_{n}]$, there are constants $\nu ^{\ast }$ and
$\sigma ^{\ast }$ such that for all $l$
\begin{equation*}
\frac{\overline{h}(t)}{\overline{h}_{l}(t)}\leq \nu ^{\ast }\text{ \ \textrm{%
and} \ }\frac{\overline{h}_{l}(t)}{\overline{\kappa }_{l}(t)}\leq \sigma
^{\ast }.
\end{equation*}
\end{lemma}

\begin{proof}
Let $\overline{h}_{l}$ and $\overline{\rho }_{l}$ be the diameters of $%
\overline{T}_{l}$ and of the largest sphere inscribed in $\overline{T}_{l}$
respectively, and $h:=\max_{l}(\overline{h}_{l})$, then the inequality (\ref%
{int4}) implies that for all $l$
\begin{equation*}
K^{-1}\overline{h}_{l}\leq \overline{h}_{l}(t)\leq K\overline{h}_{l}\text{ \
\textrm{and \ }}\overline{h}(t)\leq Kh.
\end{equation*}%
Similarly, since $\overline{\kappa }_{l}(t)\geq K^{-1}\overline{\rho }_{l}$,
then, for any $l$,%
\begin{equation*}
\frac{\overline{h}(t)}{\overline{h}_{l}(t)}\leq K^{2}\frac{h}{\overline{h}%
_{l}}\leq K^{2}\nu ,\ \ \frac{\overline{h}_{l}(t)}{\overline{\kappa }_{l}(t)}%
\leq K^{2}\frac{\overline{h}_{l}}{\overline{\rho }_{l}}\leq K^{2}\sigma .
\end{equation*}
So, letting $\nu ^{\ast }=K^{2}\nu $ and $\sigma ^{\ast }=K^{2}\sigma $, $%
\nu $ and $\sigma $ being the quasi-uniform regularity constants of $\left(
\mathcal{T}_{h}\right) _{h}$, the result follows.
\end{proof}

The next lemma provides a quantitative evaluation of the difference between
the elements $T_{l}(t)$ and $\widetilde{T}_{l}(t)$.

\begin{lemma}
\label{lTht} Let $u\in C(\left[ 0,T\right] ;W^{k+1,\infty }(\Omega )^{d})$
and $\Delta t$ and $h$ be sufficiently small, then for all $t\in \lbrack
t_{n-q},t_{n})$ and for all $1\leq l\leq NE$, $F_{l}(t)$ is a mapping of
class $C^{k,1}(\widehat{T})$, and there is a constant $C$ independent of $%
\Delta t$ and $h$, but depending on $q$ and $k$, such that for $0\leq m\leq
k $
\begin{equation}
\left\vert F_{l}(t)-\widetilde{F}_{l}(t)\right\vert _{W^{m,\infty }(\widehat{%
T})^{d}}\leq C\overline{h}_{l}^{k+1}\Delta t\left\Vert u\right\Vert
_{L^{\infty }(0,T;W^{k+1,\infty }(T_{l}(t))^{d})}.  \label{se3:12}
\end{equation}
\end{lemma}

\begin{proof}
First, we prove that $F_{l}(t)$ is a mapping of class $C^{k,1}(\widehat{T})$%
. Thus, recalling (\ref{se3:7}) we recast $F_{l}(t)(\widehat{x})$ as
\begin{equation*}
F_{l}(t)(\widehat{x})=F_{l}(\widehat{x})-\int_{t}^{t_{n}}\widehat{v}(%
\widehat{x},\tau )d\tau ,
\end{equation*}
where $\widehat{v}(\widehat{x},\tau ):=u(X(F_{l}(\widehat{x}),t_{n};\tau
),\tau )$; now, noting that for each $t_{n}>\tau \geq t$,\ we can define the
mapping $F_{l}(\tau ):=\widehat{T}\rightarrow \mathbb{R}^{d}$ by
\begin{equation*}
F_{l}(\tau )(\widehat{x})=X(\cdot ,t_{n};\tau )\circ F_{l}(\widehat{x}),
\end{equation*}%
so that we can set $T_{l}(\tau ):=\left\{ X(x,t_{n};\tau ):x\in T_{l}\subset
\overline{\Omega }\right\} $, $T_{l}(\tau )\subset \overline{\Omega }(\tau )$%
, then by virtue of Lemma \ref{solX} it follows that $\widehat{v}(\widehat{x}%
,\tau )\in W^{k+1,\infty }(\widehat{T})^{d}$; hence, $F_{l}(t)$ is a mapping
of class $C^{k,1}(\widehat{T})$. To prove the estimate (\ref{se3:12}), we
make use of (\ref{se3:6}) and the interpolation theory in Sobolev spaces;
therefore, we can write that
\begin{equation*}
\left\vert F_{l}(t)-\widetilde{F}_{l}(t)\right\vert _{W^{m,\infty }(\widehat{%
T})^{d}}\leq C\left\vert F_{l}(t)\right\vert _{W^{k+1,\infty }(\widehat{T}%
)^{d}}=C\left\vert \int_{t}^{t_{n}}\widehat{v}(\widehat{x},\tau )d\tau
\right\vert _{W^{k+1,\infty }(\widehat{T})^{d}}.
\end{equation*}%
Since for $t\in \lbrack t_{n-q},t_{n})$, $\left\vert \int_{t}^{t_{n}}
\widehat{v}(\widehat{x},\tau )d\tau \right\vert _{W^{k+1,\infty }(\widehat{T}%
)^{d}}\leq q\Delta t\left\vert \widehat{v}(\widehat{x},\tau )\right\vert
_{L^{\infty }(t,t_{n};W^{k+1,\infty }(\widehat{T})^{d})}$, then by virtue of
Lemma \ref{lemma7} in the Appendix the result follows.
\end{proof}

The next lemma shows that when $\Delta t$ and $h$ are sufficiently small,
both $T_{l}(t)$ and $\widetilde{T}_{l}(t)$ are curved elements in the sense
defined in \cite{Bern} so that inequalities such as (\ref{se2:83}) and (\ref%
{se2:84}) must hold for these elements too.

\begin{lemma}
\label{lTht2} \ (A) Let $u\in C(\left[ 0,T\right] ;W^{k+1,\infty }(\Omega
)^{d})$ and $\Delta t$ and $h$ be sufficiently small. Then, for all $n>q,\
t\in \lbrack t_{n-q},t_{n})$ and $l$, the element $T_{l}(t)$ is a curved
element of class $C^{k,1}$; that is, $F_{l}(t)$ is a $C^{k,1}$%
-diffeomorphism from $\widehat{T}$ onto $T_{l}(t)$ and there is a constant $%
c_{T_{l}(t)}$ such that
\begin{equation*}
c_{T_{l}(t)}=\sup_{\widehat{x}\in \widehat{T}}\left\Vert D\Phi _{l}(t)(%
\widehat{x})\overline{B}_{l}^{-1}(t)\right\Vert <1,
\end{equation*}%
and for $2\leq s\leq k+1$ there are constants $c_{s}$ such that
\begin{equation*}
\left\{
\begin{array}{l}
\sup_{\widehat{x}\in \widehat{T}}\left\Vert D^{s}F_{l}(t)(\widehat{x}%
)\right\Vert \cdot \left\Vert \overline{B}_{l}(t)\right\Vert ^{-s}\leq
c_{s}(T_{l}(t)),\ \mathrm{and} \\
\\
\sup_{x\in T}\left\Vert D^{s}F_{l}^{-1}(t)(x)\right\Vert \leq
c_{-s}\left\Vert \overline{B}(t)_{l}\right\Vert ^{2(s-1)}\left\Vert
\overline{B}(t)_{l}^{-1}\right\Vert ^{s}.%
\end{array}%
\right.
\end{equation*}%
The constants $c_{-s}$ depend continuously on $c_{_{T_{l}}},\
c_{2}(T_{l}),\ldots ,c_{k+1}(T_{l})$.

(B) Likewise, the isoparametric element $\widetilde{T}_{l}(t)$ is a curved
element of class $C^{k+1}$ because (i)$\ \widetilde{F}_{l}(t)$ is a $C^{k+1}$%
-mapping from $\widehat{T}$ onto $\widetilde{T}_{l}(t)$; (ii) there is
constant $c_{\widetilde{T}(l)}$
\begin{equation*}
c_{\widetilde{T}_{l}(t)}=\sup_{\widehat{x}\in \widehat{T}}\left\Vert D%
\widetilde{\Theta }_{l}(t)(\widehat{x})\overline{B}_{l}^{-1}(t)\right\Vert
<1;
\end{equation*}
and (iii)$\ \widetilde{F}_{l}(t)$ fulfills the results of Lemma \ref%
{lemma1_apen} in the Appendix.
\end{lemma}

\begin{proof}
(A) Lemma \ref{lTht} says that $F_{l}(t)$ is a diffeomorphism of class $%
C^{k,1}$, so we have to prove now the existence of the constants $%
c_{T_{l}(t)},$ $c_{s}(T_{l}(t))$ and $c_{-s}(T_{l}(t))$. First, from (\ref%
{se3:10}) and (\ref{se3:11}) one obtains that
\begin{equation*}
\sup_{\widehat{x}\in \widehat{T}}\left\Vert D\Phi _{l}(t)(\widehat{x})%
\overline{B}_{l}^{-1}(t)\right\Vert \leq \sup_{\widehat{x}\in \widehat{T}%
}\left\Vert D(F_{l}(t)(\widehat{x})-\overline{F}_{l}(t)(\widehat{x}%
))\right\Vert \cdot \left\Vert \overline{B}_{l}^{-1}(t)\right\Vert .
\end{equation*}
Since $\overline{F}_{l}(t)(\widehat{x})=\widehat{I}_{1}F_{l}(t)(\widehat{x})$
\begin{equation*}
\sup_{\widehat{x}\in \widehat{T}}\left\Vert D(F_{l}(t)(\widehat{x})-%
\overline{F}_{l}(t)(\widehat{x}))\right\Vert =\left\vert F_{l}(t)-\widehat{I}%
_{1}F_{l}(t)\right\vert _{W^{1,\infty }(\widehat{T})^{d}}\leq C\left\vert
F_{l}(t)(\widehat{x})\right\vert _{W^{2,\infty }(\widehat{T})^{d}},
\end{equation*}
and by (\ref{se3:7})
\begin{equation*}
\left\vert F_{l}(t)(\widehat{x})\right\vert _{W^{2,\infty }(\widehat{T}%
)^{d}}\leq \left\vert F_{l}(\widehat{x})\right\vert _{W^{2,\infty }(\widehat{%
T})^{d}}+\left\vert \int_{t}^{t_{n}}\widehat{v}(\widehat{x},\tau )d\tau
\right\vert _{W^{2,\infty }(\widehat{T})^{d}},
\end{equation*}%
with\ $\widehat{v}(\widehat{x},\tau ):=u(X(F_{l}(\widehat{x}),t_{n};\tau
),\tau )$. From (\ref{se2:84}) it follows that $\ \left\vert F_{l}(\widehat{x%
})\right\vert _{W^{2,\infty }(\widehat{T})^{d}}\leq c_{2}(T_{l})\overline{h}%
_{l}^{2}$, and we estimate $\left\vert \int_{t}^{t_{n}}\widehat{v}(\widehat{x%
},\tau )d\tau \right\vert _{W^{2,\infty }(\widehat{T})^{d}}$ arguing as in
the proof of Lemma \ref{lTht}. Hence, there is a constant $K(T_{l})$ such
that
\begin{equation*}
\sup_{\widehat{x}\in \widehat{T}}\left\Vert D(F_{l}(t)(\widehat{x})-%
\overline{F}_{l}(t)(\widehat{x}))\right\Vert \leq K(T_{l})\overline{h}%
_{l}^{2}\left( 1+\Delta t\left\Vert u\right\Vert _{C(\left[ t,t_{n}\right]
;W^{2,\infty }(T_{l}(t))^{d})}\right) .
\end{equation*}
Recalling that for affine equivalent elements
\begin{equation*}
\left\Vert \overline{B}_{l}^{-1}(t)\right\Vert \leq \frac{\widehat{h}}{%
\overline{\rho }_{l}(t)}\leq c(\sigma ^{\ast },\nu ^{\ast })\overline{h}%
_{l}^{-1},
\end{equation*}%
then
\begin{equation}
c_{T_{l}(t)}=\sup_{\widehat{x}\in \widehat{T}}\left\Vert D\Phi _{l}(t)(%
\widehat{x})\overline{B}_{l}^{-1}(t)\right\Vert \leq c(\sigma ^{\ast },\nu
^{\ast })K(T_{l})\overline{h}_{l}\left( 1+\Delta t\left\Vert u\right\Vert
_{C(\left[ t,t_{n}\right] ;W^{2,\infty }(T_{l}(t))^{d})}\right) <1
\label{se3:132}
\end{equation}
if $\Delta t$ and $h$ are sufficiently small. To prove the existence of the
constants $c_{s}(T_{l}(t))$, for $2\leq s\leq k+1$, we notice once again
that
\begin{equation*}
\begin{array}{l}
\sup_{\widehat{x}\in \widehat{T}}\left\Vert D^{s}F_{l}(t)(\widehat{x}%
)\right\Vert =\left\vert F_{l}(t)(\widehat{x})\right\vert _{W^{s,\infty }(%
\widehat{T})} \\
\\
\leq \left\vert F_{l}(\widehat{x})\right\vert _{W^{s,\infty }(\widehat{T}%
)^{d}}+\left\vert \int_{t}^{t_{n}}\widehat{v}(\widehat{x},\tau )d\tau
\right\vert _{W^{s,\infty }(\widehat{T})^{d}} \\
\\
\leq c_{s}(T_{l})\overline{h}_{l}^{s}+\overline{c}_{s}(T_{l})qh^{s}\Delta
t\left\Vert u\right\Vert _{C(\left[ t,t_{n}\right] ;W^{s,\infty
}(T_{l}(t))^{d})} \\
\\
\leq C\overline{h}_{l}^{s}(1+\Delta t\left\Vert u\right\Vert _{C(\left[
t,t_{n}\right] ;W^{s,\infty }(T_{l}(t))^{d})}).%
\end{array}%
\end{equation*}%
Hence, it follows the existence of a constant $c_{s}(T_{l}(t))$. The
existence of the constants $c_{-s}(T_{l}(t))$ is proven as in Lemma 2.2 of
\cite{Bern}.

(B) First, since $\widetilde{F}_{l}(t)(\widehat{x})=\widehat{I}_{k}F_{l}(t)(%
\widehat{x})$, then $\widetilde{F}_{l}(t)$ is a mapping of class $C^{k+1}(%
\widehat{T})$. Secondly, to calculate the constant $c_{\widetilde{T}_{l}(t)}$
we recall that $D\widetilde{\Theta }_{l}(t)(\widehat{x})\overline{B}%
_{l}^{-1}(t)=D\left( \widetilde{F}_{l}(t)(\widehat{x})-\overline{F}_{l}(t)(%
\widehat{x})\right) \overline{B}_{l}^{-1}(t)$ and $\overline{F}_{l}(t)(%
\widehat{x})=\widehat{I}_{1}F_{l}(t)(\widehat{x})$. So, we can write
\begin{equation*}
D\left( \widetilde{F}_{l}(t)(\widehat{x})-\overline{F}_{l}(t)(\widehat{x}%
)\right) =D\left( \widetilde{F}_{l}(t)(\widehat{x})-F_{l}(t)(\widehat{x}%
)\right) +D\left(F_{l}(t)(\widehat{x})-\overline{F}_{l}(t)(\widehat{x}%
)\right).
\end{equation*}%
We have already estimated $\sup_{\widehat{x}\in \widehat{T}}\left\Vert
D(F_{l}(t)(\widehat{x})-\overline{F}_{l}(t)(\widehat{x}))\right\Vert $, so
it remains to estimate $D\left( \widetilde{F}_{l}(t)(\widehat{x})-F_{l}(t)(%
\widehat{x})\right) $. To do so, we recall that $\widetilde{F}_{l}(t)(%
\widehat{x})=\widehat{I}_{k}F_{l}(t)(\widehat{x})$ and (\ref{se3:7}), then
setting $\widehat{v}(\widehat{x},\tau )=u(X(F_{l}(\widehat{x}),t_{n};\tau ))$
yields
\begin{equation*}
\widetilde{F}_{l}(t)(\widehat{x})-F_{l}(t)(\widehat{x})=\int_{t_{n}}^{t}%
\left( \widehat{I}_{k}\widehat{v}(\widehat{x},\tau )-\widehat{v}(\widehat{x}%
,\tau )\right) d\tau .
\end{equation*}
From this expression and approximation theory we obtain, using the arguments
of the proof of Lemma \ref{lTht}, that
\begin{equation*}
\begin{array}{l}
\sup_{\widehat{x}\in \widehat{T}}\left\Vert D\left( \widetilde{F}_{l}(t)(%
\widehat{x})-F_{l}(t)(\widehat{x})\right) \right\Vert =\left\vert \widetilde{%
F}_{l}(t)(\widehat{x})-F_{l}(t)(\widehat{x})\right\vert _{W^{1,\infty }(%
\widehat{T})^{d}} \\
\\
\leq q\Delta t\left\vert \widehat{v}-\widehat{I}_{k}\widehat{v}\right\vert
_{L^{\infty }(t,t_{n};W^{1,\infty }(\widehat{T})^{d})}\leq C\Delta
th_{l}^{k+1}\left\Vert u\right\Vert _{C(\left[ 0,T\right]
;W^{k+1}(T_{l}(t))^{d})}.%
\end{array}%
\end{equation*}
Then, collecting these estimates we have that
\begin{equation}
\begin{array}{r}
c_{\widetilde{T}_{l}(t)}=\sup_{\widehat{x}\in \widehat{T}}\left\Vert D%
\widetilde{\Theta}_{l}(t)(\widehat{x})\overline{B}_{l}^{-1}(t)\right\Vert
\leq \sup_{\widehat{x}\in \widehat{T}}\left\Vert D\left( \widetilde{F}%
_{l}(t)(\widehat{x})-F_{l}(t)(\widehat{x})\right) \right\Vert \cdot
\left\Vert \overline{B}_{l}^{-1}(t)\right\Vert \\
\\
+\sup_{\widehat{x}\in \widehat{T}}\left\Vert D\left( F_{l}(t)(\widehat{x})-%
\overline{F}_{l}(t)(\widehat{x})\right) \right\Vert \cdot \left\Vert
\overline{B}_{l}^{-1}(t)\right\Vert \\
\\
\leq c(\sigma ^{\ast },\nu ^{\ast })\left[ K(T_{l})\overline{h}_{l}+\left(C%
\overline{h}_{l}^{k}+K(T_{l})\overline{h}_{l}\right) \Delta t\left\Vert
u\right\Vert _{C(\left[ 0,T\right] ;W^{k+1}(T_{l}(t))^{d})}\right] <1%
\end{array}
\label{se3:133}
\end{equation}%
if $\Delta t$ and $h$ are sufficiently small. Finally,\ the proof of
statement (iii) is identical to the proof of Lemma 2.1 of \cite{Bern}, so is
omitted.
\end{proof}

\begin{remark}
\label{hdelta_t} Since both constants $c_{T_{l}(t)}$ and
$c_{\widetilde{T}_{l}(t)}$ are required to be smaller than $1$, then
one has to choose positive numbers $h_{1}$ and $\Delta t_{1}$ such
that for $h<h_{1}$ and $\Delta t<\Delta t_{1}$ both (\ref{se3:132})
and (\ref{se3:133}) hold. Hence, in the sequel we shall assume that
both $\Delta t$ and $h$ are such that $\Delta t\in (0,\Delta t_{1})$
and $h\in (0,h_{1})$.
\end{remark}

\begin{lemma}
\label{conformal} If $\mathcal{T}_{h}$ is a geometrically conforming
partition, then the partition $\mathcal{T}_{h}(t)=X(\mathcal{T}_{h},t_{n};t)$
composed of curved elements of class $C^{k,1}$ is also geometrically
conforming.
\end{lemma}

\begin{proof}
Let any pair of elements $T_{l},T_{j}\in \mathcal{T}_{h}$ and assume that
there is a $(d-1)-$face $\Gamma ^{lj}=T_{l}\cap T_{j}\subset
\Omega $. Since $T_{l}=F_{l}(\widehat{T})$ and $T_{j}=F_{j}(%
\widehat{T})$ and $\mathcal{T}_{h}$ is geometrically conforming, then there
is a face $\widehat{\Gamma }$ of $\widehat{T}$ such that $\Gamma
^{lj}=F_{l}(\widehat{\Gamma })=F_{j}(\widehat{\Gamma })$ and $\left.
F_{l}(\widehat{T})\right\vert _{\Gamma ^{lj}}=\left. F_{j}(\widehat{T%
})\right\vert _{\Gamma ^{lj}}$, see (\cite{EG}, Definition 1.55). Now, we consider
the elements $T_{l}(t)=F_{l}(t)(\widehat{T})=X(\cdot ,t_{n};t)\circ
F_{l}(\widehat{T})$ and $T_{j}(t)=F_{j}(t)(\widehat{T})=X(\cdot
,t_{n};t)\circ F_{j}(\widehat{T})$, so $F_{l}(t)(\widehat{\Gamma }%
)=X(\cdot ,t_{n};t)\circ F_{l}(\widehat{\Gamma })$ and $F_{j}(t)(%
\widehat{\Gamma })=X(\cdot ,t_{n};t)\circ F_{j}(\widehat{\Gamma })$.
Hence, $\left. F_{l}(t)(\widehat{T})\right\vert _{\Gamma ^{lj}}=X(\cdot
,t_{n};t)\circ \left. F_{l}(\widehat{T})\right\vert _{\Gamma ^{lj}}$, $%
\left. F_{j}(t)(\widehat{T})\right\vert _{\Gamma ^{lj}}=X(\cdot
,t_{n};t)\circ \left. F_{j}(\widehat{T})\right\vert _{\Gamma ^{lj}}$ and
given that$\left. F_{l}(\widehat{T})\right\vert _{\Gamma ^{lj}}=\left.
F_{j}(\widehat{T})\right\vert _{\Gamma ^{lj}}$ it follows that there is
a $(d-1)-$face $\Gamma ^{lj}(t)=T_{l}(t)\cap T_{j}(t)\subset
\Omega _{h}(t)$ such that $\Gamma ^{lj}(t)=F_{T_{l}}(t)(\widehat{\Gamma }%
)=F_{T_{j}}(t)(\widehat{\Gamma })$.
\end{proof}

Using the same arguments it can be shown that the partitions $\widetilde{%
\mathcal{T}}_{h}(t)$ formed by the isoparametric simplices of type $\mathrm{%
(k)}$ are also conforming.

Now, by virtue of the Definitions \ref{definition2.5}-\ref{definition5} of
the Appendix we are in a position to introduce the $H^{1}-$conforming finite
element spaces $V_{h}(t)\subset H^{1}(\Omega (t))$ associated with the
partitions $\mathcal{T}_{h}(t)$. Thus,%
\begin{equation}
\left\{
\begin{array}{l}
V_{h}(t):=\left\{ u_{h}\in C(\overline{\Omega }_{h}(t)):\forall T_{l}(t)\in
\mathcal{T}_{h}(t)\ \left. u_{h}\right\vert _{T_{l}(t)}\in P(T_{l}(t)),\
1\leq l\leq NE\right\} , \\
\\
P(T_{l}(t)):=\left\{ p(y)=\phi _{t_{n}}^{t}p(x)\ \ \mathrm{with\ \ }p(x)\in
P(T_{l})\right\} .%
\end{array}%
\right.  \label{se3:2}
\end{equation}%
So, if $\left\{ \chi _{j}(x)\right\} _{j=1}^{M}$ is the set of global basis
functions of $V_{h}$, then the set$\left\{ \phi _{t_{n}}^{t}\chi
_{j}(x)\right\} _{j=1}^{M}=\left\{ \chi _{j}(X(x,t_{n};t))\right\}
_{j=1}^{M} $ is the set of global basis functions of $V_{h}(t)$ and any
function $u_{h}(X(x,t_{n};t),t)\in V_{h}(t)$ is expressed as%
\begin{equation}
u_{h}(X(x,t_{n};t),t)=\sum_{j=1}^{M}U_{j}(t)\chi _{j}(X(x,t_{n};t)).
\label{se3:4}
\end{equation}

Returning to the construction of the partitions $\widetilde{\mathcal{T}}%
_{h}(t)$ and $\mathcal{T}_{h}(t)$ we recall, see (\ref{se3:5}) and (\ref%
{se3:1}), that for each $l$
\begin{equation*}
\widetilde{\mathcal{T}}_{h}(t)\ni \widetilde{T}_{l}(t)=\widetilde{F}_{l}(t)(%
\widehat{T})\ \ \mathrm{and\ \ }\mathcal{T}_{h}(t)\ni T_{l}(t)=F_{l}(t)(%
\widehat{T}).
\end{equation*}
So, noting that for all$\ \widehat{x}\in \widehat{T}$ there is one and only
one $x\in T_{l}$ such that $\widehat{x}=F_{l}^{-1}(x)$, we have that for any
point $x\in T_{l}$ there is one and only one $\widetilde{y}(t)\in \widetilde{%
T}_{l}(t)$ and a unique $y(t)\in T_{l}(t)$ such that
\begin{equation}
\widetilde{y}(t)=\widetilde{F}_{l}(t)\circ F_{l}^{-1}(x)\ \ \mathrm{and\ \ }
y(t)=F_{l}(t)\circ F_{l}^{-1}(x).  \label{se3:121}
\end{equation}%
Since by virtue of Lemma \ref{lTht2} (ii)$\ \widetilde{F}_{l}(t)$ is a
mapping of class $C^{k+1}(\widehat{T})$ and by virtue of the theorem of the
inverse function $F_{l}^{-1}(x)$ is also a mapping of class $C^{k+1}(T_{l})$%
, then we can define a local mapping of class $C^{k+1}(T_{l}),$ $\widetilde{X%
}_{l}(\cdot ,t_{n};t):T_{l}\rightarrow \widetilde{T}_{l}(t)$ as $x\in
T_{l}\rightarrow $\ $\widetilde{X}_{l}(x,t_{n};t)=\widetilde{F}_{l}(t)\circ
F_{l}^{-1}(x)\in \widetilde{T}_{l}(t)$. Hence, we can construct the global
mapping
\begin{equation*}
\widetilde{X}(\cdot ,t_{n};t):\Omega \rightarrow \widetilde{\Omega }_{h}(t)
\text{\ \ \ \textrm{as\ \ \ }}\mathbf{\forall }l,\ \ \left. \widetilde{X}%
(x,t_{n};t)\right\vert _{T_{l}}=\widetilde{X}_{l}(x,t_{n};t).
\end{equation*}%
The mapping $F_{l}(t)\circ F_{l}^{-1}(x)$ is of class $C^{k,1}(T_{l})$ with $%
F_{l}(t)\circ F_{l}^{-1}(x)=\left. X(x,t_{n};t)\right\vert _{T_{l}}$,\ and
by virtue of Lemma \ref{solX} $X(x,t_{n};t)$ is a mapping of class $C^{k,1}$
from $\Omega \rightarrow \Omega _{h}(t)$.\ Though $\widetilde{F}_{l}(t)$ and
$F_{l}^{-1}$ are mappings of class $C^{k+1}$, the mapping $\widetilde{X}(\cdot,t_{n};t)$ is of class $C^{0,1}(\Omega )$ instead. This result is stated in the next lemma.

\begin{lemma}
\label{X-tilde} For all $n\geq q$ and $t_{n-q}\leq t<t$, the mapping $%
\widetilde{X}(\cdot ,t_{n};t):\Omega \rightarrow \widetilde{\Omega }_{h}(t)$
is of class $C^{0,1}(\Omega ).$
\end{lemma}

\begin{proof}
Noting that if $NN$ denotes the number of nodes of the element $T_{l}$ and $%
\{\chi _{j}^{(l)}(x)\}_{j=1}^{NN}$ is the set of local basis functions of $%
P(T_{l})$, we can set (see (\ref{se3:5}))
\begin{equation*}
\widetilde{X}_{l}(x,t_{n};t)=\widetilde{F}_{l}(t)\circ
F_{l}^{-1}(x)=\sum_{j=1}^{NN}X(a_{j}^{(l)},t_{n};t)\chi _{j}^{(l)}(x),
\end{equation*}%
then $\widetilde{X}(x,t_{n};t)\in V_{h},$ see (\ref{se2:85}); so $\widetilde{%
X}(x,t_{n};t)\in W^{1,\infty }(\Omega )$, and by virtue of Theorem 5.8.4 of
\cite{Evans} $\widetilde{X}(x,t_{n};t)\in C^{0,1}(\Omega )$.
\end{proof}

Based on this lemma we can associate with the partition $\widetilde{\mathcal{%
T}}_{h}(t)$ the $H^{1}-$conforming finite element space $\widetilde{V}%
_{h}(t) $ defined as
\begin{equation*}
\left\{
\begin{array}{l}
\widetilde{V}_{h}(t):=\left\{ u_{h}\in C(\overline{\widetilde{\Omega }}%
_{h}(t)):\forall \widetilde{T}_{l}(t)\in \widetilde{\mathcal{T}}_{h}(t)\
\left. u_{h}\right\vert _{\widetilde{T}_{l}(t)}\in P(\widetilde{T}_{l}(t)),\
1\leq l\leq NE\right\} , \\
\\
P(\widetilde{T}_{l}(t)):=\left\{ p(\widetilde{y})=\widehat{p}\circ
\widetilde{F}_{l}^{-1}(t)(\widetilde{y})\ \mathrm{with\ }\widehat{p}(%
\widehat{x})\in \widehat{P}_{k}(\widehat{T})\right\} .%
\end{array}
\right.
\end{equation*}
Hence, $\widetilde{V}_{h}(t)=Span\left\{ \chi _{j}(\widetilde{X}%
(x,t_{n};t))\right\} $. Note that $\chi _{j}(\widetilde{X}(x,t_{n};t))=\chi
_{j}(x)$,

\begin{remark}
\label{diffeomorphism} Notice that by construction $X(x,t_{n};t)$ is a
diffeomorphism from $\Omega $ onto $\Omega _{h}(t)$, and given that $\Omega
_{h}(t)=\Omega $ because $\Omega $ is invariant, then $X(x,t_{n};t)$ is a
diffeomorphisn from $\Omega $ onto itself. As for $\widetilde{X}(x,t_{n};t)$%
, by construction it is a homeomorphism from $\Omega $ onto $\widetilde{%
\Omega }_{h}(t)$, but not from $\Omega $ onto itself because, although $%
\Omega \cap \widetilde{\Omega }_{h}(t)\neq \left\{ \emptyset \right\} $, $%
\Omega \neq \widetilde{\Omega }_{h}(t)$.
\end{remark}

\begin{corollary}
\label{corol1} Under the conditions of Lemma \ref{lTht} and for $m$ integer,
$0\leq m\leq 1$, it holds that
\begin{equation}
\left\vert X(x,t_{n};t)-\widetilde{X}(x,t_{n};t)\right\vert _{W^{m.\infty
}(\Omega )^{d}}\leq Ch^{k+1-m}\Delta t\left\Vert u\right\Vert _{C(\left[ 0,T%
\right] ;W^{k+1,\infty }(\Omega )^{d})}.  \label{se3:13}
\end{equation}
where the constant $C$ depends on $q$, $k$ and $\sigma ^{\ast }$, and when $%
m=0$ it is understood that
\begin{equation*}
\left\vert X(x,t_{n};t)-\widetilde{X}(x,t_{n};t)\right\vert _{W^{m.\infty
}(\Omega )^{d}}=\left\Vert X(x,t_{n};t)-\widetilde{X}(x,t_{n};t)\right\Vert
_{L^{\infty }(\Omega )^{d}}.
\end{equation*}
\end{corollary}

\begin{proof}
First, we observe that
\begin{equation*}
\left\vert X(x,t_{n};t)-\widetilde{X}(x,t_{n};t)\right\vert _{W^{m.\infty
}(\Omega )^{d}}=\max_{l}\left\vert X(x,t_{n};t)-\widetilde{X}%
(x,t_{n};t)\right\vert _{W^{m.\infty }(T_{l})^{d}}.
\end{equation*}%
Recalling the definitions of $X(x,t_{n};t)$ and $\widetilde{X}(x,t_{n};t)$
we have that for all $l$
\begin{equation*}
\left\vert X(x,t_{n};t)-\widetilde{X}(x,t_{n};t)\right\vert _{W^{m.\infty
}(T_{l})^{d}}=\left\vert \left( F_{l}(t)-\widetilde{F}_{l}(t)\right) \circ
F_{l}^{-1}(x)\right\vert _{W^{m.\infty }(T_{l})^{d}}.
\end{equation*}%
Hence, by virtue of Lemma \ref{lemma7} of the Appendix, (\ref{se3:12}) and
the relations
\begin{equation*}
\left\Vert u\right\Vert _{C(\left[ 0,T\right] ;W^{k+1,\infty
}(T_{l}(t))^{d})}\leq \left\Vert u\right\Vert _{C(\left[ 0,T\right]
;W^{k+1,\infty }(\Omega _{h}(t))^{d})}\leq C\left\Vert u\right\Vert _{C(%
\left[ 0,T\right] ;W^{k+1,\infty }(\Omega )^{d})},
\end{equation*}
the estimate (\ref{se3:13} ) follows.
\end{proof}

This result is useful to define a bounded open domain $\mathcal{O}$, such
that both $\Omega _{h}(t)\subset \subset \mathcal{O}$ and $\widetilde{\Omega
}_{h}(t)\subset \subset \mathcal{O}$, so that on the account that $\partial
\Omega _{h}(t)$ is piecewise smooth and Lipschitz continuous, we can
introduce the continuous linear extension operator $E_{t}:W^{1,p}(\Omega
_{h}(t))\rightarrow W^{1,p}(\mathcal{\mathbb{R}}^{d})$ with the properties $%
E_{t}v=v$ a.e in $\Omega _{h}(t)$, $E_{t}v$ has a support within $\mathcal{O}
$, and for$\mathrm{\ }m\geq 0$ \textrm{and}$\ p\in \lbrack 1,\infty ]$ there
is a constant $C=C(\Omega ,\mathcal{O})$ such that
\begin{equation}
\left\Vert E_{t}v\right\Vert _{W^{m,p}(\mathcal{O})}\leq C\left\Vert
v\right\Vert _{W^{m,p}(\Omega _{h}(t))}.  \label{se3:130}
\end{equation}%
To this end, we consider the sign distance function \cite{ER1}
\begin{equation*}
d(x,t)=\left\{
\begin{array}{l}
-\underset{y\in \partial \Omega_{h}(t)}{\inf }\left\vert x-y\right\vert \
\mathrm{if}\ \ x\notin \Omega _{h}(t), \\
 \ \ \ \ \ \ \ \ \ \ \ \ \ \ \ 0\ \ \ \ \ \mathrm{if}\ \ x\in
\partial \Omega _{h}(t), \\
 \ \underset{y\in \partial \Omega_{h}(t)}{\inf }\left\vert x-y\right\vert \ \
\mathrm{if}\ \ x\in \Omega _{h}(t),%
\end{array}%
\right.
\end{equation*}%
here, $\left\vert \cdot \right\vert $ denotes the Euclidean distance between
the points $x$ and $y$. Noting that $\left\vert \nabla d(x,t)\right\vert =1$
in a neighborhood of $\partial \Omega _{h}(t)$, then the normal vector at
almost every point $x\in \partial \Omega _{h}(t)$ is $\nu (x,t)=\nabla
d(x,t) $. In view of (\ref{se3:13}), let $r_{\partial \Omega
}=d^{1/2}Ch^{k+1}\Delta t\left\Vert u\right\Vert _{C(\left[ 0,T\right]
;W^{k+1,\infty }(\Omega )^{d})}$, we define a band about $\Omega _{h}(t)$,
see Figure 4 \begin{figure}[h]
    \centering
    \begin{tikzpicture}[scale=2.0, >=stealth]

        \coordinate (v0) at (-0.5,0.0);
        \coordinate (v1) at (0,0);
        \coordinate (v2) at (2,0.5);
        \coordinate (v3) at (4,0.75);
        \coordinate (v4) at (4.5,0.7);
        \draw[thick, darkgreen] plot[smooth, tension=1] coordinates{(v1) (v2) (v3)};
        \node[draw,circle,fill=black,inner sep=0pt,minimum size=0.10cm,label=right:]() at (v1) {};
        \node[draw,circle,fill=black,inner sep=0pt,minimum size=0.10cm,label=right:]() at (v2) {};
        \node[draw,circle,fill=black,inner sep=0pt,minimum size=0.10cm,label=right:]() at (v3) {};
        \draw[thick, darkgreen] plot[smooth, tension=1] coordinates{(v1)  ($0.5*(v1)+0.5*(v0)-(0,0.02)$) (v0)};
        \draw[thick, darkgreen] plot[smooth, tension=1] coordinates{(v3)  ($0.5*(v3)+0.5*(v4)+(0,0.02)$) (v4)};
        \node[left]() at (-0.5,-0.05) {\textcolor{darkgreen}{$\partial\widetilde{\Omega}_{h}(t)$}};

        \coordinate (w0) at (-0.5,0.10);
        \coordinate (w1) at (0.7,0.05);
        \coordinate (w2) at (3.1,0.77);
        \coordinate (w3) at (4.5,0.65);
        \draw[thick, red] plot[smooth, tension=0.5] coordinates{(w0) (v1) (w1) (v2) (w2) (v3) (w3)};
        \node[left]() at (-0.5,0.20) {\textcolor{red}{$\partial\Omega_{h}(t)$}};

        \draw[red, thick, dashed] plot[smooth, tension=0.5] coordinates{($(w0)+(0,0.5)$)($(v1)+(0,0.5)$) ($(w1)+(0,0.5)$) ($(v2)+(0,0.5)$) ($(w2)+(0,0.5)$) ($(v3)+(0,0.5)$) ($(w3)+(0,0.5)$)};
        \draw[red, thick, dashed] plot[smooth, tension=0.5] coordinates{($(w0)-(0,0.5)$) ($(v1)-(0,0.5)$) ($(w1)-(0,0.5)$) ($(v2)-(0,0.5)$) ($(w2)-(0,0.5)$) ($(v3)-(0,0.5)$) ($(w3)-(0,0.5)$)};

        \draw[<->] (3.1,0.77) -- (3.05,1.27) node[midway, right] {\small $r_{\partial\Omega}$};
        \draw[<->] (3.1,0.77) -- (3.15,0.27) node[midway, right] {\small $r_{\partial\Omega}$};

    \end{tikzpicture}
\caption {\ The band of width $2r_{\partial \Omega}$ about
$\Omega_{h}(t)$.}
\end{figure}
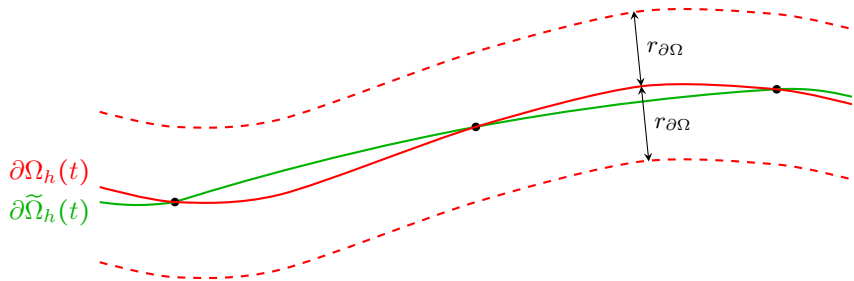
 \label{Fig4}, as $B_{r_{\partial \Omega }}(x,t)=\left\{ x\in \mathbb{R}%
^{d}:\left\vert d(x,t)\right\vert <r_{\partial \Omega }\right\} $.
On the account that $d(x,t)\in
C^{2}(B_{r_{\partial \Omega }}(x,t))$, one can define for any $x\in
B_{r_{\partial \Omega }}(x,t)$ the normal projection $p(x)$ on the boundary $
\partial \Omega _{h}(t)$ as
\begin{equation*}
x=p(x)+d(x,t)\nu (p(x),t).
\end{equation*}%
Since this decomposition is unique, then $\nu $ can be extended to a vector
field on all of $B_{r_{\partial \Omega }}(x,t)$ by letting $\nu (x,t)=\nu
(p(x),t)$. Based on these considerations and given that $\overline{\Omega }%
_{h}(t)=\overline{\Omega }$, we define the bounded domain\ $\mathcal{O}$ as\
$\mathcal{O}:=\overline{\Omega }\bigcup B_{r_{\partial \Omega }}(x,t)$ with\
$\partial \mathcal{O}:=\left\{ y\in \mathbb{R}^{d}:y=x+r_{\partial \Omega
}\nu (x,t),\ x\in \partial \Omega \right\} $.

Hereafter, unless otherwise stated, we shall use the shorthand notations $\widetilde{X}^{n,n-i}(x)$ and $X^{n,n-i}(x)$ to denote $\widetilde{X}(x,t_{n};t_{n-i})$ and $X(x,t_{n};t_{n-i})$ respectively, and $%
a^{n-i}(X^{n,n-i}(x))$ and $a^{n-i}(\widetilde{X}^{n,n-i}(x))$ to denote $%
a(X(x,t_{n};t_{n-i}),t_{n-i})$ and $a(\widetilde{X}(x,t_{n};t_{n-i}),t_{n-i})$
respectively.

\begin{lemma}
\label{Halpha} Let $0\leq \alpha \leq 1$, and let $\Delta t \in
(0,\Delta t_{1})$ and $h \in (0,h_{1})$ be sufficiently small; then,
for all $n\geq q$ and $t_{n-i}\in \{t_{n-q},\ldots ,t_{n}\}$, the
mapping $H_{\alpha }^{n}:\Omega \rightarrow
\mathcal{O}$ defined as%
\begin{equation}
H_{\alpha }^{n}(x)=X^{n,n-i}(x)+\alpha (\widetilde{X}%
^{n,n-i}(x)-X^{n,n-i}(x)),  \label{se3:131}
\end{equation}%
is one-to one and of class $C^{0,1}$.\ Moreover, let $\frac{\partial
H_{\alpha }^{n}(x)}{\partial x}$ denote the Jacobian matrix of the
transformation, there exists a constant $C$ such that%
\begin{equation*}
1-C\Delta t\leq \det \left( \frac{\partial H_{\alpha }^{n}(x)}{\partial x}%
\right) \leq 1+C\Delta t.
\end{equation*}
\end{lemma}

\begin{proof}
First, we prove that $H_{\alpha }^{n}(x)$ is of class $C^{0,1}$. This
readily follows from the fact that\ $X^{n,n-i}(x)$ is of class $%
C^{k,1}(\Omega )$, $\widetilde{X}^{n,n-i}(x)$ is of class $C^{0,1}(\Omega )$
and $C^{k,1}(\Omega )\subset C^{0,1}(\Omega )$. Next, we prove that $%
H_{\alpha }^{n}(x)$ is a one-to-one mapping by contradiction. To this end,
we introduce the function $g(x):=\widetilde{X}^{n,n-i}(x)-X^{n,n-i}(x)$, $%
g(x)\in C^{0,1}(\Omega )$. Let $x_{1},x_{2}\in \Omega $, $x_{1}\neq x_{2}$,
then by Radamacher's Theorem \cite[Chapter 5]{Evans}
\begin{equation*}
\left\vert g(x_{1})-g(x_{2})\right\vert \leq \left\vert g\right\vert
_{W^{1,\infty }(\Omega )^{d}}\left\vert x_{1}-x_{2}\right\vert ,
\end{equation*}%
and from (\ref{se3:13}) it follows that there is constant $C$ depending on $%
\left\Vert u\right\Vert _{L^{\infty }(0,T;W^{k+1,\infty }(\Omega )^{d})},$ $%
q,m$ and $\sigma ^{\ast }$ such that
\begin{equation*}
\left\vert g(x_{1})-g(x_{2})\right\vert \leq Ch^{k}\Delta t\left\vert
x_{1}-x_{2}\right\vert .
\end{equation*}%
Now, if $H_{\alpha }^{n}(x_{1})=H_{\alpha }^{n}(x_{2})$ we can set that
\begin{equation*}
X^{n,n-i}(x_{1})-X^{n,n-i}(x_{2})=-\alpha (g(x_{1})-g(x_{2})),
\end{equation*}%
and from (\ref{int4})%
\begin{equation*}
K^{-1}\left\vert x_{1}-x_{2}\right\vert \leq \left\vert
g(x_{1})-g(x_{2})\right\vert \leq Ch^{k}\Delta t\left\vert
x_{1}-x_{2}\right\vert \text{,}
\end{equation*}%
but for $h$ and $\Delta t$ sufficiently small, that is, for $0<h<h_{s}<h_{1}$  and $0<\Delta t<\Delta t_{s}<\Delta t_{1}$, $CKh^{k}\Delta t<1$; so, $x_{1}=x_{2}$. We recall that $\Delta t_{1}$ and $h_{1}$ are the parameters discussed in Remark \ref{hdelta_t}.
Moreover, we notice that%
\begin{equation*}
\begin{array}{l}
\lim \sup_{x_{1}\rightarrow x_{2}}\displaystyle\frac{\left\vert H_{\alpha
}^{n}(x_{1})-H_{\alpha }^{n}(x_{2})\right\vert }{\left\vert
x_{1}-x_{2}\right\vert }\leq K(1+\frac{Ch^{k}}{K}\Delta t)\ \ \mathrm{and}
\\
\\
\lim \sup_{H_{\alpha }(x_{1})\rightarrow H_{\alpha }(x_{2})}\displaystyle%
\frac{\left\vert x_{1}-x_{2}\right\vert }{\left\vert H_{\alpha
}^{n}(x_{1})-H_{\alpha }^{n}(x_{2})\right\vert }\geq
[K(1+\frac{Ch^{k}}{K}
\Delta t)]^{-1}.%
\end{array}%
\end{equation*}%
Since we can set $K=1+C\Delta t$, from these inequalities it follows that
there is another constant $C$ such that
\begin{equation*}
1-C\Delta t\leq \det \left( \frac{\partial H_{\alpha }^{n}(x)}{\partial x}%
\right) \leq 1+C\Delta t.
\end{equation*}
\end{proof}

The usefulness of the extension operator and Lemma \ref{Halpha} is
demonstrated in the proof of the next result which is important for the
stability and error analysis of the method.

\begin{lemma}
\label{X-X} Assume that $u\in C(\left[ 0,T\right] ;W^{k+1,\infty
}(\Omega )^{d})$, $v\in L^{\infty }(0,T;H^{1}(\Omega ))$, $\Delta
t\in (0,\Delta t_{s})$ and $h\in
(0,h_{s})$. Consider the continuous linear extension operator $%
E_{t_{n-i}}:W^{1,p}(\Omega _{h}(t_{n-i}))\rightarrow W^{1,p}(\mathcal{%
\mathbb{R}}^{d})$ with the properties $E_{t_{n-i}}v=v$ a.e in $\Omega
_{h}(t_{n-i})$, $E_{t_{n-i}}v$ has a support within $\mathcal{O}$.Then, for
all $n\geq q$ and $1\leq i\leq q$, there exists a constant $C(u)$
independent of $\Delta t$ and $h$ such that
\begin{equation}
\left\Vert v^{n-i}(X^{n,n-i}(x))-v^{n-i}(\widetilde{X}^{n,n-i}(x))\right%
\Vert _{L^{2}(\Omega )}\leq C(u)h^{k+1}\Delta t\left\Vert \nabla
v^{n-i}\right\Vert _{L^{2}(\Omega )^{d}},  \label{se3:16}
\end{equation}%
where $C(u)=C(\Omega ,\mathcal{O},q,k,\sigma ^{\ast },\left\Vert
u\right\Vert _{L^{\infty }(W^{k+1,\infty }(\Omega )^{d})})$.
\end{lemma}

\begin{proof}
By virtue of (\ref{se3:131}) we have that%
\begin{equation*}
E_{t_{n-i}}v^{n-i}(X^{n,n-i}(x))-E_{t_{n-i}}v^{n-i}(\widetilde{X}%
^{n,n-i}(x))=-\int_{0}^{1}\frac{\partial }{\partial \alpha }%
E_{t_{n-i}}v^{n-i}(H_{\alpha }^{n}(x))d\alpha .
\end{equation*}%
Since $\frac{\partial }{\partial \alpha }E_{t_{n-i}}v^{n-i}(H_{\alpha
}^{n}(x))=-DE_{t_{n-i}}v^{n-i}(H_{\alpha }^{n}(x))\cdot (X^{n,n-i}(x)-%
\widetilde{X}^{n,n-i}(x))$, then it follows that%
\begin{equation*}
\left\vert E_{t_{n-i}}v^{n-i}(X^{n,n-i}(x))-E_{t_{n-i}}v^{n-i}(\widetilde{X}%
^{n,n-i}(x))\right\vert \leq \left\Vert X^{n,n-i}(x)-\widetilde{X}%
^{n,n-i}(x)\right\Vert _{L^{\infty }(\Omega )^{d}}\int_{0}^{1}\left\vert
\nabla E_{t_{n-i}}v^{n-i}(H_{\alpha }^{n}(x))\right\vert d\alpha .
\end{equation*}%
Hence, by virtue of (\ref{se3:130}) and the fact that $\Omega =\Omega_{h}
(t_{n-i})\subset \subset \mathcal{O}$,%
\begin{equation*}
\begin{array}{c}
\displaystyle\int_{\Omega }\left\vert v^{n-i}(X^{n,n-i}(x))-v^{n-i}(\widetilde{X}^{n,n-i}(x))\right\vert ^{2}d\Omega \leq \displaystyle\int_{\mathcal{O} }\left\vert
E_{t_{n-i}}v^{n-i}(X^{n,n-i}(x))-E_{t_{n-i}}v^{n-i}(\widetilde{X}%
^{n,n-i}(x))\right\vert ^{2}dx \\
\\
\leq \left\Vert X^{n,n-i}(x)-\widetilde{X}^{n,n-i}(x)\right\Vert _{L^{\infty
}(\Omega )^{d}}^{2}\displaystyle\int_{0}^{1}\left( \int_{\mathcal{O} }\left\vert \nabla E_{t_{n-i}}v^{n-i}(H_{\alpha }^{n}(x))\right\vert
^{2}dx\right) d\alpha \\
\\
\leq C\left\Vert X^{n,n-i}(x)-\widetilde{X}^{n,n-i}(x)\right\Vert
_{L^{\infty }(\Omega )^{d}}^{2}\left\Vert \nabla v^{n-i}\right\Vert
_{L^{2}(\Omega )^{d}}^{2}%
\end{array}%
\end{equation*}%
So, from (\ref{se3:13}) it follows the result (\ref{se3:16}).
\end{proof}

It is worth remarking that by virtue of (\ref{s2:2}) $(J^{n,n-i})^{-1}\leq
C_{J}:=\exp (q\Delta tC_{\mathrm{div}})$, then it readily follows that%
\begin{equation}
\left\Vert v^{n-i}(X^{n,n-i}(x))\right\Vert _{L^{2}(\Omega )}\leq
C_{J}\left\Vert v^{n-1}\right\Vert _{L^{2}(\Omega )}.  \label{se3:17}
\end{equation}


\subsection{Calculation of the integrals $\left(
c_{h}^{n-i},v_{h}^{n-i}\right) _{\Omega _{h}(t_{n-i})}$ of (\protect\ref%
{s2:11}) for $1\leq i\leq q$}

We explain the procedure to approximate%
\begin{equation*}
\left( c_{h}^{n-i},v_{h}^{n-i}\right) _{\Omega
_{h}(t_{n-i})}:=\sum_{l=1}^{NE}%
\int_{T_{l}(t_{n-i})}c_{h}^{n-i}(y^{n-i})v_{h}^{n-i}(y^{n-i})d\Omega
_{h}(t_{n-i}),\ 1\leq i\leq q,
\end{equation*}%
As is usual in finite element calculations, we shall employ the element of
reference $\widehat{T}$ to calculate the integrals. Thus, performing the
change of variable, see (\ref{se3:121}), $\left. y^{n-i}\right\vert
_{T_{l}(t_{n-i})}=F_{l}(t_{n-i})(\widehat{x})$,\ $\left. d\Omega
_{h}(t_{n-i})\right\vert =J_{F_{l}(t_{n-i})}(\widehat{x})d\widehat{x}$,
where $J_{F_{l}(t_{n-i})}(\widehat{x})$ is the Jacobian determinant of the
transformation $F_{l}(t_{n-i})(\widehat{x})$, we can set%
\begin{equation*}
\begin{array}{r}
\left( c_{h}^{n-i},v_{h}^{n-i}\right) _{T_{l}(t_{n-i})}:=\displaystyle%
\int_{T_{l}(t_{n-i})}c_{h}^{n-i}(y^{n-i})v_{h}^{n-i}(y^{n-i})d\Omega
_{h}(t_{n-i}) \\
\\
=\displaystyle\int_{\widehat{T}}c_{h}^{n-i}(F_{l}(t_{n-i})(\widehat{x}%
))v_{h}^{n-i}(F_{l}(t_{n-i})(\widehat{x}))J_{F_{l}(t_{n-i})}(\widehat{x})d%
\widehat{x},%
\end{array}%
\end{equation*}%
where in view of (\ref{se3:1}) $J_{F_{l}(t_{n-i})}(\widehat{x})=\displaystyle%
\left\vert \frac{\partial X(F_{l}(\widehat{x}),t_{n};t_{n-i})}{\partial x}%
\right\vert \displaystyle\left\vert \frac{\partial F_{l}(\widehat{x})}{%
\partial \widehat{x}}\right\vert $, with $\displaystyle\left\vert \frac{%
\partial F_{l}(\widehat{x})}{\partial \widehat{x}}\right\vert $ being the
Jacobian determinant of the transformation $F_{l}:\widehat{T}\rightarrow
T_{l}$, and $\displaystyle\left\vert \frac{\partial X(F_{l}(\widehat{x}%
),t_{n};t_{n-i})}{\partial x}\right\vert =J(F_{l}(\widehat{x}%
),t_{n};t_{n-i}) $, see (\ref{s2:0}). So, we can set that
\begin{equation*}
J_{F_{l}(t_{n-i})}(\widehat{x})=\left\vert \frac{\partial F_{l}(\widehat{x})%
}{\partial \widehat{x}}\right\vert \exp \left( -\int_{t_{n-i}}^{t_{n}}\nabla
_{X}\cdot u(F_{l}(s)(\widehat{x}),s)ds\right) .
\end{equation*}%
In general, the integrals defined on $\widehat{T}$ are calculated via
quadrature rules of high order due to the difficulty of calculating them
analytically; so, we set
\begin{equation}
\begin{array}{l}
\displaystyle\int_{\widehat{T}}c_{h}^{n-i}(F_{l}(t_{n-i})(\widehat{x}%
))v_{h}^{n-i}(F_{l}(t_{n-i})(\widehat{x}))J_{F_{l}(t_{n-i})}(\widehat{x})d%
\widehat{x} \\
\\
=\displaystyle\sum_{g=1}^{nqp}\varpi _{g}c_{h}^{n-i}(F_{l}(t_{n-i})(\widehat{%
x}_{g}))v_{h}^{n-i}(F_{l}(t_{n-i})(\widehat{x}_{g}))J_{F_{l}(t_{n-i})}(%
\widehat{x}_{g}),%
\end{array}
\label{I1}
\end{equation}%
where $nqp$ denotes the number of quadrature points $\widehat{x}_{g}\in
\widehat{T}$ and $\varpi _{g}$ are the weights of the rule. There are two
main problems in applying the quadrature formula. The first one has to do
with the calculation of $F_{l}(t_{n-i})(\widehat{x}_{g})$, for $%
F_{l}(t_{n-i})(\widehat{x}_{g})=X^{n,n-1}(F_{l}(\widehat{x}%
_{g}))=X^{n,n-1}(x_{g})$, with $X^{n,n-1}(x_{g})$ being the solution of the
system (\ref{int2}) with the initial condition $X(x,t_{n};t_{n})=x_{g}=F_{l}(%
\widehat{x}_{g})\in \overline{\Omega }$; so, if for accuracy reasons $nqp$
has to be large, then at each time step and for each element $l$ one has to
solve $nqp$ times the system (\ref{int2}) with the initial condition $x_{g}$%
. The second problem is about to finding the element $T_{k}$ of the fixed
mesh $\mathcal{T}_{h}$ that contains each point $F_{l}(t_{n-i})(\widehat{x}%
_{g})$ because $c_{h}^{n-i}(F_{l}(t_{n-i})(\widehat{x}))$ and $%
J_{F_{l}(t_{n-i})}(\widehat{x})$ must be evaluated at $F_{l}(t_{n-i})(%
\widehat{x}_{g})$. In order to alleviate the computational burden of this
procedure we follow the ideas of \cite{bc}-\cite{BS2} and approximate $%
F_{l}(t_{n-i})(\widehat{x})$ by $\widetilde{F}_{l}(t_{n-i})(\widehat{x}%
)=I_{k}F_{l}(t_{n-i})(\widehat{x})$; so, the integrals are now
\begin{equation}
\begin{array}{l}
\displaystyle\int_{\widehat{T}}c_{h}^{n-i}(\widetilde{F}_{l}(t_{n-i})(%
\widehat{x}))v_{h}^{n-i}(\widetilde{F}_{l}(t_{n-i})(\widehat{x}))\widetilde{J%
}_{F_{l}(t_{n-i})}(\widehat{x})d\widehat{T} \\
\\
=\displaystyle\sum_{g=1}^{nqp}\varpi _{g}c_{h}^{n-i}(\widetilde{F}%
_{l}(t_{n-i})(\widehat{x}_{g}))v_{h}^{n-i}(\widetilde{F}_{l}(t_{n-i})(%
\widehat{x}_{g}))\widetilde{J}_{F_{l}(t_{n-i})}(\widehat{x}_{g}),%
\end{array}
\label{I2}
\end{equation}%
where%
\begin{equation*}
\widetilde{J}_{F_{l}(t_{n-i})}(\widehat{x})=\left\vert \frac{\partial F_{l}(%
\widehat{x})}{\partial \widehat{x}}\right\vert \exp \left(
-\int_{t_{n-i}}^{t_{n}}\nabla _{X}\cdot u(\widetilde{F}_{l}(s)(\widehat{x}%
),s)ds\right) .
\end{equation*}%
and $\widetilde{F}_{l}(s)(\widehat{x}_{g})\in \widetilde{T}_{l}(s)$;
recalling (\ref{se3:5}), $\widetilde{F}_{l}(s)(\widehat{x}%
_{g})=\sum_{j=1}^{NN}X(a_{j}^{(l)},t_{n};s)\widehat{\chi }_{j}(\widehat{x}%
_{g})$. From a computational point of view (in terms of CPU time and memory
storage requirements) is much more efficient to calculate the points $%
\widetilde{F}_{l}(s)(\widehat{x}_{g})$ than the points $F_{l}(s)(\widehat{x}%
_{g})$ because the integration of the system (\ref{int2}) with the initial
conditions $X(x,t_{n};t_{n})=F_{l}(\widehat{x}_{g})$ is replaced by the
finite element interpolation in the element $\widetilde{T}_{l}(s)$.
Likewise, the search-locate procedure to find the element $T_{k}$ of the
fixed mesh $\mathcal{T}_{h}$ that contains each point $\widetilde{y}_{\lg
}^{n-1}:=\widetilde{F}_{l}(t_{n-i})(\widehat{x}_{g})$ is also more
efficient. Looking at the integral (\ref{I1}) and recalling that $\left.
X^{n,n-i}(x)\right\vert _{T_{l}}=F_{l}(t_{n-i})(\widehat{x})$ and $\left.
d\Omega \right\vert _{T_{l}}=\displaystyle\left\vert \frac{\partial F_{l}(%
\widehat{x})}{\partial \widehat{x}}\right\vert d\widehat{T}$, we have that
\begin{equation*}
\begin{array}{c}
\displaystyle\int_{\widehat{T}}c_{h}^{n-i}(F_{l}(t_{n-i})(\widehat{x}%
))v_{h}^{n-i}(F_{l}(t_{n-i})(\widehat{x}))J_{F_{l}(t_{n-i})}(\widehat{x})d%
\widehat{T} \\
\\
=\displaystyle%
\int_{T_{l}}c_{h}^{n-i}(X^{n,n-i}(x))v_{h}^{n-i}(X^{n,n-i}(x))J^{n,n-i}(x)d%
\Omega ,%
\end{array}%
\end{equation*}%
where, see (\ref{s2:1}),%
\begin{equation*}
J^{n,n-i}(x)=J(x,t_{n};t_{n-i})=\exp \left( -\int_{t_{n-i}}^{t_{n}}\nabla
_{X}\cdot u(X^{n,n-i}(x),s)ds\right) ;
\end{equation*}%
hence, it follows that%
\begin{equation*}
\int_{\Omega _{h}(t_{n-i})}c_{h}^{n-i}(y^{n-i})v_{h}^{n-i}(y^{n-i})d\Omega
_{h}(t_{n-i})=\int_{\Omega
}c_{h}^{n-i}(X^{n,n-i}(x))v_{h}^{n-i}(X^{n,n-i}(x))J^{n,n-i}(x)d\Omega .
\end{equation*}
As for the integral (\ref{I2}), we recall that $\left. \widetilde{X}%
^{n,n-i}(x)\right\vert _{T_{l}}=\widetilde{F}_{l}(t_{n-i})(\widehat{x})$ and
consequently
\begin{equation*}
\begin{array}{c}
\displaystyle\int_{\widehat{T}}c_{h}^{n-i}(\widetilde{F}_{l}(t_{n-i})(%
\widehat{x}))v_{h}^{n-i}(\widetilde{F}_{l}(t_{n-i})(\widehat{x}))\widetilde{J%
}_{F_{l}(t_{n-i})}(\widehat{x})d\widehat{T} \\
\\
=\displaystyle\int_{T_{l}}c_{h}^{n-i}(\widetilde{X}^{n,n-i}(x))v_{h}^{n-i}(
\widetilde{X}^{n,n-i}(x))\widetilde{J}^{n,n-i}(x)d\Omega ,%
\end{array}%
\end{equation*}
where
\begin{equation*}
\widetilde{J}^{n,n-i}(x)=\widetilde{J}(x,t_{n};t_{n-i})=\exp \left(
-\int_{t_{n-i}}^{t_{n}}\nabla _{X}\cdot u(\widetilde{X}^{n,n-i}(x),s)ds%
\right) .
\end{equation*}
Therefore,
\begin{equation*}
\int_{\Omega }c_{h}^{n-i}(\widetilde{X}^{n,n-i}(x))v_{h}^{n-i}(\widetilde{X}%
^{n,n-i}(x))\widetilde{J}^{n,n-i}(x)d\Omega
=\sum_{l=1}^{NE}\int_{T_{l}}c_{h}^{n-i}(\widetilde{X}^{n,n-i}(x))v_{h}^{n-i}(%
\widetilde{X}^{n,n-i}(x))\widetilde{J}^{n,n-i}(x)d\Omega
\end{equation*}
is an approximation to\ $\displaystyle\int_{\Omega
}c_{h}^{n-i}(X^{n,n-i}(x))v_{h}^{n-i}(X^{n,n-i}(x))J^{n,n-i}(x)d\Omega $.
Now, taking into account that $v_{h}(X^{n,n-i}(x))\in V_{h}(t_{n-i})$ and $%
v_{h}(\widetilde{X}^{n.n-i}(x))\in \widetilde{V}_{h}(t_{n-i})$ are of the
form
\begin{equation*}
v_{h}(X^{n,n-i}(x))=\sum_{j=1}^{M}V_{j}\chi _{j}(X^{n,n-i}(x)),\ \ v_{h}(%
\widetilde{X}^{n.n-i}(x))=\sum_{j=1}^{M}V_{j}\chi _{j}(\widetilde{X}%
^{n,n-i}(x)),
\end{equation*}
and given that both $\chi _{j}(X^{n,n-i}(x))$ and $\chi _{j}(\widetilde{X}%
^{n,n-i}(x))=\chi _{j}(x)$, then $v_{h}(X^{n,n-i}(x))$ and $v_{h}(\widetilde{%
X}^{n.n-i}(x))$ are equal to $v_{h}(x)\in V_{h}$. Then, returning to (\ref%
{s2:11}), we calculate $c_{h}^{n}\in V_{h}$ as solution of
\begin{equation}
\left\{
\begin{array}{l}
\alpha _{0}\left( c_{h}^{n},v_{h}\right) _{\Omega }+\Delta
ta(c_{h}^{n},v_{h})=\Delta t(f^{n},v_{h})_{\Omega } \\
\\
-\displaystyle\sum_{i=1}^{q}\alpha _{i}\left( \widetilde{J}%
^{n,n-i}(x)c_{h}^{n-i}(\widetilde{X}^{n,n-i}(x)),v_{h}(x))\right) _{\Omega
}\ \forall v_{h}\in V_{h},%
\end{array}%
\right.  \label{se3:14}
\end{equation}
where it is remarkable to point out that we have replaced $\displaystyle %
\int_{\Omega }c_{h}^{n-i}(X^{n,n-i}(x))v_{h}^{n-i}(x)J^{n,n-i}(x)d\Omega $
by $\displaystyle\int_{\Omega }c_{h}^{n-i}(\widetilde{X}^{n,n-i}(x))v_{h}(x))%
\widetilde{J}^{n,n-i}(x)d\Omega$. Next, it is convenient to ascertain the
error committed in doing this replacement. To this end we denote this error
by $IE^{n,n-i}$, so we have that
\begin{equation*}
IE^{n,n-i}:=\int_{\Omega
}J^{n,n-i}(x)c_{h}^{n-i}(X^{n,n-i}(x))v_{h}(x)d\Omega -\int_{\Omega}%
\widetilde{J}^{n,n-i}(x)c_{h}^{n-i}(\widetilde{X}^{n,n-i}(x))v_{h}(x)d\Omega,
\end{equation*}%
which, after simple arithmetical operations and hereafter using the notation
$\widehat{c}_{h}^{n-i}(x)$ and $\widetilde{c}_{h}^{n-i}(x)$ to denote $%
c_{h}^{n-i}(X^{n,n-i}(x))$ and $c_{h}^{n-i}(\widetilde{X}^{n,n-i}(x))$
respectively, can also be expressed by%
\begin{equation}
IE^{n,n-i}=\int_{\Omega }J^{n,n-i}(x)\left( \widehat{c}_{h}^{n-i}(x)-%
\widetilde{c}_{h}^{n-i}(x)\right) v_{h}(x)d\Omega +\int_{\Omega
}(J^{n,n-i}(x)-\widetilde{J}^{n,n-i}(x))\widetilde{c}_{h}^{n-i}(x)v_{h}(x)d%
\Omega .  \label{I3}
\end{equation}%
Notice that when $i=0$, $IE^{n,n}=0$ because $J^{n,n}(x)=\widetilde{J}%
^{n,n}(x)=1$ and $\widehat{c}_{h}^{n}(x)=\widetilde{c}%
_{h}^{n}(x)=c_{h}^{n}(x)\in V_{h}$. We have the following result.

\begin{lemma}
\label{erintg} For all $n\geq q$, $1\leq i\leq q$, $\Delta t\in
(0,\Delta t_{s})$ and $h\in
(0,h_{s})$, there is a positive constant $C(u)$ such that the error $%
IE^{n,n-i}$ is bounded by%
\begin{equation}
\left\vert IE^{n,n-i}\right\vert \leq C(u)\Delta th^{k+1}\left( \left\Vert
\nabla c_{h}^{n-i}\right\Vert _{L^{2}(\Omega )}+\left\Vert
c_{h}^{n-i}\right\Vert _{L^{2}(\Omega )}\right) \left\Vert v_{h}\right\Vert
_{L^{2}(\Omega )},  \label{I4}
\end{equation}%
where the constant $C(u)=C(\Omega ,\mathcal{O},C_{J},q,k,\sigma ^{\ast
},\left\Vert u\right\Vert _{L^{\infty }(W^{k+1,\infty }(\Omega )^{d})})$.
\end{lemma}

\begin{proof}
By the triangle inequality it follows that%
\begin{equation*}
\left\vert IE^{n,n-i}\right\vert \leq \int_{\Omega }\left\vert
J^{n,n-i}(x)\right\vert \left\vert \left( \widehat{c}_{h}^{n-i}(x)-%
\widetilde{c}_{h}^{n-i}(x)\right) \right\vert \left\vert v_{h}(x)\right\vert
d\Omega +\int_{\Omega }\left\vert J^{n,n-i}(x)-\widetilde{J}%
^{n,n-i}(x)\right\vert \left\vert \widetilde{c}_{h}^{n-i}(x)\right\vert
\left\vert v_{h}(x)\right\vert d\Omega .
\end{equation*}%
Next, we apply H\"{o}lder inequality to estimate the terms on the right hand
side and obtain that%
\begin{equation*}
\begin{array}{l}
\left\vert IE^{n,n-i}\right\vert \leq \left( \left\Vert J^{n,n-i}\right\Vert
_{L^{\infty }(\Omega )}\left\Vert \widehat{c}_{h}^{n-i}-\widetilde{c}%
_{h}^{n-i}\right\Vert _{L^{2}(D)}\right. \\
\\
\ \ \ \ \ \ \ \ \ \ \ \ \ \ \ +\left. \left\Vert J^{n,n-i}-\widetilde{J}%
^{n,n-i}\right\Vert _{L^{\infty }(\Omega )}\left\Vert \widetilde{c}%
_{h}^{n-i}\right\Vert _{L^{2}(D)}\right) \left\Vert v_{h}\right\Vert
_{L^{2}(D)}.%
\end{array}%
\end{equation*}%
To bound the terms on the right side of this inequality we notice that by
virtue of (\ref{se3:16})%
\begin{equation*}
\left\Vert \widehat{c}_{h}^{n-i}-\widetilde{c}_{h}^{n-i}\right\Vert
_{L^{2}(\Omega )}\leq C(u)\Delta th^{k+1}\left\Vert \nabla
c_{h}^{n-i}\right\Vert _{L^{2}(\Omega )}
\end{equation*}%
and from (\ref{se3:17})%
\begin{equation*}
\left\Vert \widehat{c}_{h}^{n-i}\right\Vert _{L^{2}(\Omega )}\leq
C_{J}\left\Vert c_{h}^{n-1}\right\Vert _{L^{2}(\Omega )}
\end{equation*}%
and by the triangle and inverse inequalities
\begin{equation*}
\begin{array}{r}
\left\Vert \widetilde{c}_{h}^{n-i}\right\Vert _{L^{2}(\Omega )}\leq
\left\Vert \widehat{c}_{h}^{n-i}\right\Vert _{L^{2}(\Omega )}+\left\Vert
\widehat{c}_{h}^{n-i}-\widetilde{c}_{h}^{n-i}\right\Vert _{L^{2}(\Omega )}
\\
\\
\leq C_{J}\left( 1+C(u)\Delta th^{k}\right) \left\Vert
c_{h}^{n-1}\right\Vert _{L^{2}(\Omega )}.%
\end{array}%
\end{equation*}
To estimate $\left\Vert J^{n,n-i}\right\Vert _{L^{\infty }(\Omega )}$ we
make use of (\ref{s2:1}) and obtain that
\begin{equation*}
\left\Vert J^{n,n-i}\right\Vert _{L^{\infty }(\Omega )}\leq C_{J},\ \
C_{J}:=\exp (q\Delta tC_{\mathrm{div}}).
\end{equation*}
To estimate $\left\Vert J^{n,n-i}-\widetilde{J}^{n,n-i}\right\Vert
_{L^{\infty }(\Omega )}$ we introduce the function $B:\Omega
_{h}(t_{n-i})\rightarrow \mathbb{R}$ such as $B(X^{n,n-i}(x)):=J^{n,n-i}(x)$%
, so $B\in W^{k,\infty }(\Omega _{h}(t_{n-i}))$, and consider the continuous
linear extension operator of Lemma \ref{X-X}, then arguing as in the proof
of this lemma we readily arrive at%
\begin{equation*}
\left\Vert J^{n,n-i}-\widetilde{J}^{n,n-i}\right\Vert _{L^{\infty }(\Omega
)}\leq C\left\Vert X^{n,n-i}(x)-\widetilde{X}^{n,n-i}(x)\right\Vert
_{L^{\infty }(\Omega )^{d}}\left\Vert B\right\Vert _{W^{1,\infty }(\Omega
)^{d}},
\end{equation*}%
so, by virtue of (\ref{se3:13}), there is a constant $C(u)=$ $C(\Omega ,%
\mathcal{O(}t_{n-i}),C_{J},q,k,\sigma ^{\ast },\left\Vert u\right\Vert
_{L^{\infty }(W^{k+1,\infty }(\Omega )^{d})})$ such that%
\begin{equation*}
\left\Vert J^{n,n-i}-\widetilde{J}^{n,n-i}\right\Vert _{L^{\infty }(\Omega
)}\leq C(u)\Delta th^{k+1}.
\end{equation*}%
Hence, collecting these estimates we readily obtain (\ref{I4}).
\end{proof}

\begin{remark}
If we were able to calculate the integrals $%
\int_{T_{l}(t_{n-i})}c_{h}^{n-i}(X^{n,n-i}(x))v_{h}^{n-i}(X^{n,n-i}(x))d%
\Omega_{h}(t_{n-i})$ exactly, the procedure would be conservative because
letting $v_{h}^{n-i}=1$ and denoting $J(x,t_{n};t_{n-i})$ by $J^{n,n-i}(x)$
we have that%
\begin{equation*}
\sum_{l=1}^{NE}\int_{T_{l}(t_{n-i})}c_{h}^{n-i}(X^{n,n-i}(x))d\Omega
(t_{n-i})=\int_{\Omega
_{h}(t_{n-i})}c_{h}^{n-i}(X^{n,n-i}(x))J^{n,n-i}(x)d\Omega
_{h}(t_{n-i})=\int_{\Omega }c_{h}^{n-i}(z)d\Omega .
\end{equation*}
with $z=X^{n,n-i}(x)$. Hence, it follows that we can recast (\ref{s2:11}) as
\begin{equation*}
\sum_{n=q}^{N}\sum_{i=0}^{q}\alpha _{i}\int_{\Omega }c_{h}^{n-i}(x)d\Omega
+a_{0}\Delta t\sum_{n=q}^{N}\int_{\Omega }c_{h}^{n}(x)d\Omega =\Delta
t\sum_{n=q}^{N}\int_{\Omega }f^{n}(x)d\Omega .
\end{equation*}
Since $\sum_{i=0}^{q}\alpha _{i}=0$, it is easy to see that the latter
equality yields
\begin{equation*}
\sum_{i=0}^{q-1}\delta _{i}\int_{\Omega }c_{h}^{N-i}(x)d\Omega
+\sum_{i=1}^{q}\beta _{q-i}\int_{\Omega }c_{h}^{q-i}(x)d\Omega +a_{0}\Delta
t\sum_{n=q}^{N}\int_{\Omega }c_{h}^{n}(x)d\Omega =\Delta
t\sum_{n=q}^{N}\int_{\Omega }f^{n}(x)d\Omega ,
\end{equation*}
where $\delta _{i}=\sum_{j=0}^{i}\alpha _{i}$ and $\beta
_{q-i}=\sum_{j=i}^{q}\alpha _{j}$. Notice that if $a_{0}$ and $f^{n}$ were
identically zero then for all $N>q$, $\sum_{i=0}^{q-1}\delta
_{i}\int_{\Omega }c_{h}^{N-i}(x)d\Omega $ is constant. Our methods are not
fully conservative because as we have seen above the integrals $%
\int_{T_{l}(t_{n-i})}c_{h}^{n-i}d\Omega _{h}(t_{n-i})$ are not calculated
exactly, so that $\sum_{l=1}^{NE}\int_{\widetilde{T}
_{l}(t_{n-i})}c_{h}^{n-i}(\widetilde{y}^{n-i})d\Omega _{h}(t_{n-i})\neq
\int_{\Omega }c_{h}^{n-i}(x)d\Omega $. This is the reason we say that our
methods are nearly-conservative. However, numerical experiments, see \cite%
{colera}, show that the global mass conservation error of our method is much
smaller than that of the conventional LG method.
\end{remark}

\section{Stability and error analysis}

We present the stability and error analysis in the $L^{2}$-norm of
the NCLG-BDF-q methods for $1\leq q\leq 5$, although it is
possible to extend the analysis to $q=6$ using the approach
presented in \cite{Ak1}, however this would enlarge the article
excessively. So, for the sake of economy, we restrict our analysis
up to order $q=5$, and also assume that the feet of the
characteristics $X(x,t_{n};t_{n-i})$ are calculated exactly. In
\cite{BS1} is presented the procedure to include the error committed
when the feet $X(x,t_{n};t_{n-i})$ are approximated by one-step
numerical methods for initial value problems.

\subsection{Stability analysis}

Following \cite{Ak}, we study the stability of the solution $c_{h}^{n}$ in
the $L^{2}-$norm making use of a result due to Nevanlinna and Odeh \cite{NOD}%
, which, based on Dahlquist G-stability theory, makes relatively easy the
stability analysis in the $L^{2}-$norm for the BDF-q schemes up to order $%
q=5 $. To this end, we recast (\ref{se3:14}) conveniently for the
application of such a result. In doing so, we have that
\begin{equation}
\left\{
\begin{array}{l}
\displaystyle\sum_{i=0}^{q}\alpha _{i}\left( J^{n,n-i}\widehat{c}%
_{h}^{n-i},v_{h}\right) _{\Omega }+\Delta ta(c_{h}^{n},v_{h}) \\
\\
=\displaystyle\sum_{i=1}^{q}\alpha _{i}\left( J^{n,n-i}(\widehat{c}%
_{h}^{n-i}-\widetilde{c}_{h}^{n-i}),v_{h}\right) _{\Omega }+\displaystyle%
\sum_{i=1}^{q}\alpha _{i}\left( (J^{n,n-i}-\widetilde{J}^{n,n-i})\widetilde{c%
}_{h}^{n-i},v_{h}\right) _{\Omega }+\Delta t(f^{n},v_{h})_{\Omega }.%
\end{array}%
\right.  \label{se4:1}
\end{equation}%
According to the result of Nevanlinna and Odeh for the BDF-q schemes
with $1\leq q\leq 5$, see Lemma 1.1 of \cite{Ak} and references
therein, if the sequence $\{v^{i}\}_{i=0}^{q}$ belongs to an inner
product space $(W,(\cdot ,\cdot ))$ with associated norm $\left\Vert
\cdot \right\Vert $, then there exist constants $0\leq \eta _{q}<1$,
real numbers $a_{i}$, and a positive definite matrix $G=(g_{ij})\in \mathbb{R}%
^{q\times q}$ such that
\begin{equation}
\text{Re}\left( \sum_{i=0}^{q}\alpha _{i}v^{i},v^{q}-\eta _{q}v^{q-1}\right)
=\sum_{i,j=1}^{q}g_{ij}(v^{i},v^{j})-%
\sum_{i,j=1}^{q}g_{ij}(v^{i-1},v^{j-1})+\left\Vert
\sum_{i=0}^{q}a_{i}v^{i}\right\Vert ^{2}.  \label{se4:3}
\end{equation}%
The smallest possible values of $\eta _{q}$ are: $\eta _{1}=\eta _{2}=0,\
\eta _{3}=0.836,\ \eta _{4}=0.2878,\ \eta _{5}=0.8160$. To recast (\ref%
{se4:1}) in the light of this formula, we take as
inner product space $ L^{2}(\Omega )$ and interpret that for each $n$, $%
v^{i}=J^{n,n-(q-i)}\widehat{c}_{h}^{n-(q-i)}(x)=J^{n,n-(q-i)}c_{h}^{n-(q-i)}(X^{n,n-(q-i)}(x))$ so
that $v^{q}=\widehat{c}_{h}^{n}(x)=c_{h}^{n}\in V_{h}$ because $J^{n,n}=1$, and $v^{q-1}=
\widehat{c}_{h}^{(n-1)}(x):=c_{h}^{n-1}$. So, using (\ref{se4:3}) in (\ref%
{se4:1}) and setting $m=n-1$ yields
\begin{equation}
\begin{array}{r}
\displaystyle\sum_{i,j=0}^{q-1}g_{i+1j+1}\left( \int_{\Omega
}J^{n,n-i}J^{n,n-j}\widehat{c}_{h}^{n-i}\widehat{c}_{h}^{n-j}d\Omega
-\int_{\Omega }J^{m,m-i}J^{m,m-j}\widehat{c}_{h}^{m-i}\widehat{c}%
_{h}^{m-j}d\Omega \right) \\
\\
+\Delta t\mu \displaystyle\int_{\Omega }\nabla c_{h}^{n}\cdot \nabla \left(
c_{h}^{n}-\eta _{q}c_{h}^{n-1}\right) d\Omega +\Delta ta_{0}\displaystyle%
\int_{\Omega }c_{h}^{n}(c_{h}^{n}-\eta _{q}c_{h}^{n-1})d\Omega \\
\\
\leq \displaystyle\underbrace{\sum_{i=1}^{q}\alpha _{i}\int_{\Omega
}J^{n,n-i}(\widehat{c}_{h}^{n-i}-\widetilde{c}_{h}^{n-i})(c_{h}^{n}-\eta
_{q}c_{h}^{n-1})d\Omega }_{T_{1}} \\
\\
+\displaystyle\underbrace{\sum_{i=1}^{q}\alpha _{i}\int_{\Omega }\left(
\widetilde{J}^{n,n-i}-J^{n,n-i}\right) \widetilde{c}_{h}^{n-i}(c_{h}^{n}-%
\eta _{q}c_{h}^{n-1})d\Omega }_{T_{1}} \\
\\
+\Delta t\displaystyle\underbrace{\int_{\Omega }f^{n}(c_{h}^{n}-\eta
_{q}c_{h}^{n-1})d\Omega \ }_{T_{2}}\equiv T_{1}+T_{2}.%
\end{array}
\label{se4:3.5}
\end{equation}%
Let $\widehat{C}_{h}^{n}:=[J^{n,n}c_{h}^{n}(x),\ldots ,J^{n,n-q+1}\widehat{c}%
_{h}^{n-q+1}]^{T}$, since $G$ is a symmetric positive definite matrix we can
define a norm $\left\Vert \widehat{C}_{h}^{n}\right\Vert _{G}$ by%
\begin{equation*}
\left\Vert \widehat{C}_{h}^{n}\right\Vert
_{G}:=\sum_{i,j=0}^{q-1}g_{i+1j+1}\int_{\Omega }J^{n,n-i}J^{n,n-j}\widehat{c}%
_{h}^{n-i}\widehat{c}_{h}^{n-j}d\Omega,
\end{equation*}%
then the above inequality can be written as%
\begin{equation}
\left\Vert \widehat{C}_{h}^{n}\right\Vert _{G}-\left\Vert \widehat{C}%
_{h}^{m}\right\Vert _{G}+\Delta ta(c_{h}^{n},c_{h}^{n}-\eta
_{q}c_{h}^{n-1})\leq T1+T2.  \label{se4:4}
\end{equation}

We establish the stability of the BDF-q scheme (\ref{se4:1}) for $1\leq
q\leq 5$.

\begin{theorem}
\label{stability} Assume that $u\in $\ $C(\left[ 0,T\right] ;W^{k+1,\infty
}(\Omega )^{d})$. Let the constants%
\begin{equation*}
\begin{array}{l}
C(u):=C(\Omega ,\mathcal{O},C_{J},q,k,\sigma ^{\ast },\left\Vert u\right\Vert
_{L^{\infty }(W^{k+1,\infty }(\Omega )^{d})}), \\
\\
\lambda _{0}:=1+\sum_{i=1}^{q}\left\vert \alpha _{i}\right\vert \left(
2+C^{2}(u)h^{2k}\right) ,\ \beta _{0}:=1-\Delta t\lambda _{0}\vartheta ^{-1},
\\
\\
\beta _{1}:=\min (\mu ,a_{0})\vartheta ^{-1}\ \mathrm{and}\ \beta _{2}:=\frac{%
\max (\mu ,a_{0})\eta _{q}}{2},\
\end{array}%
\end{equation*}
with $\vartheta $ being the smallest eigenvalue of the matrix $G$.
Then, for $1\leq q\leq 5$, $\ h\in (0,h_{s})$, $\Delta t\in
(0,\Delta t_{s})$ and such that $\Delta t\vartheta
^{-1}\lambda _{0}<1$, the scheme (\ref{se4:1}) is stable in the $L^{2}-$%
norm, that is, there exists a constant $K$ such that%
\begin{equation}
\left\Vert c_{h}^{N}\right\Vert _{L^{2}(\Omega )}^{2}+\Delta t\beta
_{1}\sum_{n=q}^{N}\left\Vert c_{h}^{n}\right\Vert _{H^{1}(\Omega )}^{2}\leq
K\left( A+\Delta tB+\Delta t\sum_{n=q}^{N}\left\Vert f^{n}\right\Vert
_{L^{2}(\Omega )}^{2}\right) ,  \label{se4:41}
\end{equation}%
where $K=\vartheta ^{-1}\exp (\frac{\lambda _{0}\vartheta ^{-1}}{\beta _{0}}%
t_{N})$,\ $A=\xi \sum_{l=0}^{q-1}\left\Vert c_{h}^{l}\right\Vert
_{L^{2}(\Omega )}^{2}$, $\xi >0$ being a bounded constant depending upon the
coefficients $g_{ij}$ of the matrix $G$, and $B=\beta _{2}\left\Vert
c_{h}^{q-1}\right\Vert _{H^{1}(\Omega )}^{2}$.
\end{theorem}

\begin{proof}
We estimate the terms of (\ref{se4:4}). We start with $\Delta
ta(c_{h}^{n},c_{h}^{n}-\eta _{q}c_{h}^{n-1})$. Applying the Cauchy-Schwarz
inequality we readily obtain that
\begin{equation*}
\mu \int_{\Omega }\nabla c_{h}^{n}\cdot \nabla \left( c_{h}^{n}-\eta
_{q}c_{h}^{n-1}\right) d\Omega \geq \alpha \mu \left\Vert \nabla
c_{n}^{n}\right\Vert _{L^{2}(\Omega )}^{2}+\frac{\mu \eta _{q}}{2}\left(
\left\Vert \nabla c_{h}^{n}\right\Vert _{L^{2}(\Omega )}^{2}-\left\Vert
\nabla c_{h}^{n-1}\right\Vert _{L^{2}(\Omega )}^{2}\right) ,
\end{equation*}%
where\ $\alpha =1-\eta _{q}>0$. Similarly,%
\begin{equation*}
a_{0}\int_{\Omega }c_{h}^{n}(c_{h}^{n}-\eta _{q}c_{h}^{n-1})d\Omega \geq
\alpha a_{0}\left\Vert c_{n}^{n}\right\Vert _{L^{2}(\Omega )}^{2}+\frac{%
a_{0}\eta _{q}}{2}\left( \left\Vert c_{h}^{n}\right\Vert _{L^{2}(\Omega
)}^{2}-\left\Vert c_{h}^{n-1}\right\Vert _{L^{2}(\Omega )}^{2}\right) .
\end{equation*}%
Hence%
\begin{equation}
\begin{array}{r}
\Delta ta(c_{h}^{n},c_{h}^{n}-\eta _{q}c_{h}^{n-1})\geq \Delta t\alpha
\left( \mu \left\Vert \nabla c_{h}^{n}\right\Vert _{L^{2}(\Omega
)^{d}}^{2}+a_{0}\left\Vert c_{h}^{n}\right\Vert _{L^{2}(\Omega )}^{2}\right)
\\
\\
+\displaystyle\frac{\Delta t\mu \eta _{q}}{2}\left( \left\Vert \nabla
c_{h}^{n}\right\Vert _{L^{2}(\Omega )^{d}}^{2}-\left\Vert \nabla
c_{h}^{n-1}\right\Vert _{L^{2}(\Omega )^{d}}^{2}\right) \\
\\
+\displaystyle\frac{\Delta ta_{0}\eta _{q}}{2}\left( \left\Vert
c_{h}^{n}\right\Vert _{L^{2}(\Omega )}^{2}-\left\Vert c_{h}^{n-1}\right\Vert
_{L^{2}(\Omega )}^{2}\right) .%
\end{array}
\label{se4:5}
\end{equation}%
Next, we estimate the term $T_{1}$. To do so we use the triangle and H\"{o}lder inequalities and the fact that%
\begin{equation*}
\left\Vert c_{h}^{n}-\eta _{q}c_{h}^{n-1}\right\Vert _{L^{2}(\Omega )}\leq
\left\Vert c_{h}^{n}\right\Vert _{L^{2}(\Omega )}+\left\Vert
c_{h}^{n-1}\right\Vert _{L^{2}(\Omega )},
\end{equation*}%
then we can set
\begin{equation}
\begin{array}{l}
T_{1}\leq \sum_{i=1}^{q}\left\vert \alpha _{i}\right\vert \left( \left\Vert
J^{n,n-i}\right\Vert _{L^{\infty }(D)}\left\Vert (\widehat{c}_{h}^{n-i}-%
\widetilde{c}_{h}^{n-i})\right\Vert _{L^{2}(\Omega )}\right. \\
\\
\ \ \ \ \ \ \ \ +\left. \left\Vert J^{n,n-i}-\widetilde{J}%
^{n,n-i}\right\Vert _{L^{\infty }(\Omega )}\left\Vert \widetilde{c}%
_{h}^{n-i}\right\Vert _{L^{2}(D)}\right) (\left\Vert c_{h}^{n}\right\Vert
_{L^{2}(\Omega )}+\left\Vert c_{h}^{n-1}\right\Vert _{L^{2}(\Omega )}).%
\end{array}
\label{se4:5.1}
\end{equation}%
Employing the same arguments as in the proof of the estimate (\ref{I4}) as well as
the inverse inequality $\left\Vert \nabla c_{h}^{n-i}\right\Vert
_{L^{2}(\Omega )^{d}}\leq ch^{-1}\left\Vert c_{h}^{n-i}\right\Vert
_{L^{2}(\Omega )}$ we obtain that%
\begin{equation*}
T_{1}\leq C(u)\Delta th^{k}\sum_{i=1}^{q}\left\vert \alpha _{i}\right\vert
\left\Vert c_{h}^{n-i}\right\Vert _{L^{2}(\Omega )}(\left\Vert
c_{h}^{n}\right\Vert _{L^{2}(\Omega )}+\left\Vert c_{h}^{n-1}\right\Vert
_{L^{2}(\Omega )}).
\end{equation*}%
Using the elementary inequality it follows that
\begin{equation}
T_{1}\leq \Delta t\sum_{i=0}^{q}\zeta _{i}\left\Vert c_{h}^{n-i}\right\Vert
_{L^{2}(\Omega )}^{2},  \label{se4:6}
\end{equation}%
where for $2\leq i\leq q,\ \zeta _{i}=C^{2}(u)h^{2k}\left\vert \alpha
_{i}\right\vert $, $\zeta _{0}=\sum_{i=1}^{q}\left\vert \alpha
_{i}\right\vert $ and $\zeta _{1}=\sum_{i=1}^{q}\left\vert \alpha
_{i}\right\vert +C^{2}(u)h^{2k}\left\vert \alpha _{1}\right\vert $.

As for the term $T_{2}$, we apply the Cauchy-Schwarz and triangle inequality
and obtain that
\begin{equation}
\begin{array}{r}
T_{2}\leq \Delta t\left\Vert f^{n}\right\Vert _{L^{2}(\Omega )}\left(
\left\Vert c_{h}^{n}\right\Vert _{L^{2}(\Omega )}+\eta _{q}\left\Vert
c_{h}^{n-1}\right\Vert _{L^{2}(\Omega )}\right) \\
\\
\leq \Delta t\left\Vert f^{n}\right\Vert _{L^{2}(\Omega )}^{2}+\displaystyle%
\frac{\Delta t}{2}(\left\Vert c_{h}^{n}\right\Vert _{L^{2}(\Omega
)}^{2}+\left\Vert c_{h}^{n-1}\right\Vert _{L^{2}(\Omega )}^{2}).%
\end{array}
\label{se4:7}
\end{equation}%
Substituting (\ref{se4:5}), (\ref{se4:6}) and (\ref{se4:7}) into (\ref{se4:4}%
) yields%
\begin{equation*}
\begin{array}{r}
\left\Vert \widehat{C}_{h}^{n}\right\Vert _{G}-\left\Vert \widehat{C}%
_{h}^{m}\right\Vert _{G}+\Delta t\alpha \left( \mu \left\Vert \nabla
c_{h}^{n}\right\Vert _{L^{2}(\Omega )}^{2}+a_{0}\left\Vert
c_{h}^{n}\right\Vert _{L^{2}(\Omega )}^{2}\right) \\
\\
+\Delta t\displaystyle\frac{\mu \eta _{q}}{2}\left( \left\Vert \nabla
c_{h}^{n}\right\Vert _{L^{2}(\Omega )^{d}}^{2}-\left\Vert \nabla
c_{h}^{n-1}\right\Vert _{L^{2}(\Omega )^{d}}^{2}\right) \\
\\
+\Delta t\displaystyle\frac{a_{0}\eta _{q}}{2}\left( \left\Vert
c_{h}^{n}\right\Vert _{L^{2}(\Omega )}^{2}-\left\Vert c_{h}^{n-1}\right\Vert
_{L^{2}(\Omega )}^{2}\right) \\
\\
\leq \Delta t\sum_{i=0}^{q}\zeta _{i}^{\ast }\left\Vert
c_{h}^{n-i}\right\Vert _{L^{2}(\Omega )}^{2}+\Delta t\left\Vert
f^{n}\right\Vert _{L^{2}(\Omega )}^{2},%
\end{array}%
\end{equation*}%
where $\zeta _{0}^{\ast }=\zeta _{0}+1/2,\ \zeta _{1}^{\ast }=\zeta _{1}+1/2$
and \ for $i\geq 2$, $\zeta _{i}^{\ast }=\zeta _{i}$. Summing this
expression from $n=q$ up to $N$ we obtain that%
\begin{equation*}
\begin{array}{l}
\left\Vert \widehat{C}_{h}^{N}\right\Vert _{G}+\Delta t\gamma
_{1}\sum_{n=q}^{N}\left\Vert c_{h}^{n}\right\Vert _{H^{1}(\Omega )}^{2} \\
\\
\leq \left\Vert \widehat{C}_{h}^{q-1}\right\Vert _{G}+\Delta t\displaystyle%
\beta _{2}\left\Vert c_{h}^{q-1}\right\Vert _{H^{1}(\Omega )}^{2} \\
\\
+\Delta t\sum_{n=q}^{N}\left\Vert f^{n}\right\Vert _{L^{2}(\Omega
)}^{2}+\Delta t\lambda _{0}\sum_{n=0}^{N}\left\Vert c_{h}^{n}\right\Vert
_{L^{2}(\Omega )}^{2},%
\end{array}%
\end{equation*}%
where $\lambda _{0}=\sum_{i=0}^{q}\zeta _{i}^{\ast }$ and $\gamma _{1}=\min
(\mu ,a_{0})$. Since the matrix $G$ is symmetric and positive definite and $%
J^{N,N}=1$, then we can set that%
\begin{equation*}
\left\Vert \widehat{C}_{h}^{N}\right\Vert _{G}\geq \vartheta \left(
\left\Vert c_{h}^{N}\right\Vert _{L^{2}(\Omega )}^{2}+\ldots +\left\Vert
J^{N,N-q+1}\widehat{c}_{h}^{N-q+1}\right\Vert _{L^{2}(\Omega )}^{2}\right) ,
\end{equation*}%
$\vartheta $ being the smallest eigenvalue of $G$, and there is a positive
constant $\xi $ depending on the coefficients $g_{ij}$ such that
\begin{equation*}
\left\Vert \widehat{C}_{h}^{q-1}\right\Vert _{G}\leq \xi \left( \left\Vert
c_{h}^{0}\right\Vert _{L^{2}(\Omega )}^{2}+\ldots +\left\Vert
c_{h}^{q-1}\right\Vert _{L^{2}(\Omega )}^{2}\right) .
\end{equation*}%
Hence, we obtain that
\begin{equation*}
\begin{array}{l}
\left\Vert c_{h}^{N}\right\Vert _{L^{2}(\Omega )}^{2}+\Delta t\beta
_{1}\sum_{n=q}^{N}\left\Vert c_{h}^{n}\right\Vert _{H^{1}(\Omega )}^{2} \\
\\
\leq \xi \vartheta ^{-1}\sum_{l=0}^{q-1}\left\Vert c_{h}^{l}\right\Vert
_{L^{2}(\Omega )}^{2}+\Delta t\displaystyle\beta _{2}\vartheta
^{-1}\left\Vert c_{h}^{q-1}\right\Vert _{H^{1}(\Omega )}^{2} \\
\\
+\Delta t\vartheta ^{-1}\sum_{n=q}^{N}\left\Vert f^{n}\right\Vert
_{L^{2}(\Omega )}^{2}+\Delta t\vartheta ^{-1}\lambda
_{0}\sum_{n=0}^{N}\left\Vert c_{h}^{n}\right\Vert _{L^{2}(\Omega )}^{2}.%
\end{array}%
\end{equation*}%
Under the condition that $\Delta t\vartheta ^{-1}\lambda _{0}<1$, the
application of Gronwall inequality yields the result (\ref{se4:41}).
\end{proof}

\subsection{Error Analysis}

We estimate the error in the $L^{2}$-norm. Let $c(x,t)$ be the solution of (%
\ref{int1}) and $c_{h}^{n}$ be the solution of (\ref{s2:11}) for $1\leq
q\leq 5$, then for all $n\geq q$ we set%
\begin{equation}
e^{n}(x):=c^{n}(x)-c_{h}^{n}(x)=\rho ^{n}(x)+\theta _{h}^{n}(x),
\label{se5:1}
\end{equation}%
with%
\begin{equation}
\rho ^{n}(x):=c^{n}(x)-R_{h}c^{n}(x)\ \ \mathrm{and\ }\ \theta
_{h}^{n}(x)=:R_{h}c^{n}(x)-c_{h}^{n}(x).  \label{se5:2}
\end{equation}%
$R_{h}c^{n}(x)$ is the elliptic projection of $c^{n}(x)$ onto $V_{h}$
defined by%
\begin{equation}
a\left( R_{h}c^{n}-c^{n},v_{h}\right) =0\ \ \forall v_{h}\in V_{h}.
\label{se5:3}
\end{equation}%
It is clear that $R_{h}c^{n}$ exists and is unique. We have the following
estimate for $\rho ^{n}(x)$ \cite{DFJ}.

\begin{lemma}
\label{rho} Assuming that $c$ and $\displaystyle\frac{\partial c}{\partial t}%
\in L^{2}(0,T;H^{k+1}(\Omega ))$, there is a constant $C$ independent of $h$
such that%
\begin{equation}
\left\Vert \rho \right\Vert _{L^{2}(0,T;L^{2}(\Omega ))}+\left\Vert \frac{%
\partial \rho }{\partial t}\right\Vert _{L^{2}(0,T;L^{2}(\Omega ))}\leq
Ch^{k+1}\left[ \left\Vert c\right\Vert _{L^{2}(0,T;H^{k+1}(\Omega
))}+\left\Vert \frac{\partial c}{\partial t}\right\Vert
_{L^{2}(0,T;H^{k+1}(\Omega ))}\right] .  \label{se5:4}
\end{equation}
\end{lemma}

The next lemma is an adaptation of Lemma 6 of \cite{BS2} to the
present method. The first inequality (\ref{se5:5}) is crucial to
obtain the optimal error estimate in the $L^{2}$-norm when $\min
(\mu ,a_{0})$ is in the range of moderately small to large values,
and it is basically due to \cite{DR}, Lemma 1.

\begin{lemma}
\label{rho2} For all $n\geq q,\ $let $v^{n-i}\in H^{1}(\Omega )$, the
following estimates hold%
\begin{equation}
\left\Vert J^{n,n-i}\left( v^{n-i}-\widehat{v}^{n-i}\right) \right\Vert
_{-1}\leq i\Delta tC_{J}K(u)\left\Vert v^{n-i}\right\Vert _{L^{2}(\Omega )},
\label{se5:5}
\end{equation}%
\begin{equation}
\begin{array}{l}
\left\Vert J^{n,n-i}\left( v^{n-i}-\widehat{v}^{n-i}\right) \right\Vert
_{L^{2}(\Omega )} \\
\\
\leq \Delta tC_{1J}\min \left( C_{2J}K(u)i\left\Vert \nabla
v^{n-i}\right\Vert _{L^{2}(\Omega )^{d}},\displaystyle\left\Vert \frac{%
v^{n-i}}{\Delta t}\right\Vert _{L^{2}(\Omega )}\right) .%
\end{array}
\label{se5:6}
\end{equation}%
where $\widehat{v}^{n-i}:=v^{n-i}(X^{n,n-i}(x))$, $\left\Vert \cdot
\right\Vert _{-1}$ denotes the norm in the dual space of $H^{1}(\Omega )$,
and the constants%
\begin{equation*}
K(u):=\left\Vert u\right\Vert _{L^{\infty }(L^{\infty }(\Omega ^{d}))},\
C_{J}:=\exp (q\Delta tC_{\mathrm{div}}),\ C_{1J}:=C_{J}(1+C_{J}),\ C_{2J}:=%
\frac{C_{J}^{1/2}}{1+C_{J}}.
\end{equation*}
\end{lemma}

In this section we assume that the following assumptions hold.

\begin{itemize}
\item \textbf{A}1 $c^{0}(x)\in H^{k+1}(\Omega );$

\item \textbf{A2}\ $c(x,t)\in L^{\infty }(0,T;H^{k+1}(\Omega ));$

\item \textbf{A3} $\frac{\partial c}{\partial t}\in L^{2}(0,T;H^{k+1}(\Omega ))$ \
and $\frac{D^{q+1}\left( Jc\right) }{Dt^{q+1}}\in L^{2}(\Omega )$

\item \textbf{A4}\ $u\in C(\left[ 0,T\right] ;W^{k+1,\infty }(\Omega )^{d});$

\item \textbf{A5}\
Let $\Delta t_{s}$ and $h_{s}$ be the parameters discussed in the proof of Lemma
\ref{Halpha}.  Let $h\in \left( 0,h_{s}\right) $, $\lambda _{0}$
and $\vartheta ^{-1}$ be the constants defined in the statement of
Theorem \ref{stability} and $\Delta t\in (0,\Delta t_{s})$
satisfying $\Delta t\vartheta ^{-1}\lambda _{0}<1$;.

\item \textbf{A6} Theorem \ref{corollary1} in the Appendix.
\end{itemize}

\begin{theorem}
\label{error} Under the assumptions \textbf{A1}-\textbf{A6} there are positive bounded
constants $F_{1}$, $F_{2}$, $F_{3}$, $F_{4}$, and $F_{5}$ such that
\begin{equation}
\begin{array}{l}
\displaystyle\max_{q\leq m\leq N}\left\Vert e^{m}\right\Vert _{L^{2}(\Omega
)}^{2}\leq F_{1}\left( F_{2}h^{2(k+1)}+F_{3}\Delta t^{2q}+F_{4}\right. \\
\\
+\left. F_{5}K^{2}(u)\min \displaystyle\left( \frac{h^{2}}{\gamma _{1}},%
\frac{1}{G_{q}(1+C_{J})},\frac{C_{J}}{\frac{K^{2}(u)\Delta t^{2}}{h^{2}}}%
\right) h^{2k}\right)%
\end{array}
\label{se5:7.1}
\end{equation}%
where $F_{1}=\vartheta ^{-1}\exp (\frac{\lambda _{0}\vartheta ^{-1}}{\beta
_{0}}t_{N})$, $F_{2}=F_{2}(c,c_{t}),\ F_{3}=C\left\Vert \frac{D^{q+1}Jc}{%
Dt^{q+1}}\right\Vert _{L^{2}(\Omega )}^{2}$, $G_{q}=\displaystyle\frac{%
\sum_{i=1}^{q}i^{2}\alpha _{i}^{2}}{\sum_{i=1}^{q}\alpha _{i}^{2}}$,%
\begin{equation*}
F_{4}=\xi \sum_{i=0}^{q-1}\left\Vert e^{i}\right\Vert _{L^{2}(\Omega )}^{2}+%
\frac{\Delta t\eta _{q}\max (\mu ,a_{0})}{2}(\left\Vert e^{q-1}\right\Vert
_{H^{1}(\Omega )}^{2}+Ch^{2k}),\ F_{5}=F_{5}(q,c,\alpha _{i}).
\end{equation*}
\end{theorem}

\begin{remark}
\label{regimes} In the statement of the theorem it is worth noting the
following items: 1) The exponential dependence on the coefficient $\gamma
_{1}^{-1}$ (recall that $\gamma _{1}=min(\mu ,a_{0})$) has been eliminated;
to this respect see \cite{CW} where the error constants are carefully
calculated and the exponential dependence upon the coefficient $\mu ^{-1}$
is also eliminated. 2) The fourth term on the right side of the inequality (%
\ref{se5:7.1}) describes the behavior of the error at different regimes of
the solution, i.e., diffusive regime and advection-dominated regime. Thus,
(i) at low local P\'{e}clet number, i.e., $\frac{h^{2}}{\gamma _{1}}$ very
low, this means large diffusion coefficient, the contribution of this term
to the error is $F_{5}\frac{K^{2}(u)h^{2(k+1)}}{\gamma _{1}}$, so the error
is then $O(h^{2(k+1)}+\Delta t^{2q})$, which can be considered as optimal;
(ii) on the contrary, when $\frac{h^{2}}{\gamma _{1}}$ is very large or in
other words in the advection dominated regime, the contribution of this term
to the error is now either $F_{5}C_{J}\left( \frac{h^{k+1}}{\Delta t}\right)
^{2}$ or $\frac{F_{5}K^{2}(u)}{G_{q}(1+C_{J})}h^{2k}$, depending upon the
balance between the terms $\frac{1}{G_{q}(1+C_{J})}$ and $\frac{h^{2}C_{J}}{%
K^{2}(u)\Delta t^{2}}$, \ so that the error is now either $O\left( \left(
\frac{h^{k+1}}{\Delta t}\right) ^{2}+\Delta t^{2q}\right)$ if $\frac{1}{G_{q}(1+C_{J})}\geq \frac{h^{2}C_{J}}{
K^{2}(u)\Delta t^{2}}$ or $O(h^{2k}+\Delta t^{2q})$ if $\Delta t$
is sufficiently small such that $\frac{1}{G_{q}(1+C_{J})}< \frac{h^{2}C_{J}}{
K^{2}(u)\Delta t^{2}}$.
\end{remark}

\begin{proof}
Since, $c_{h}^{n}(x)=c^{n}(x)-(\rho ^{n}(x)+\theta _{h}^{n}(x))$ using (\ref%
{se5:3}) we can set that
\begin{equation*}
a\left( c_{h}^{n},v_{h}\right) =a\left( (c^{n}-(\rho ^{n}+\theta
_{h}^{n}),v_{h}\right) =a\left( (c^{n}-\theta _{h}^{n}),v_{h}\right) ,
\end{equation*}%
but by virtue of (\ref{s2:7}) and noting that $\widehat{c}%
_{h}^{n}(x)=c_{h}^{n}(x)$,\ $\widehat{c}%
_{h}^{n-i}(x)=c_{h}^{n-i}(X(x,t_{n};t_{n-i}))$ and the Jacobian matrix $F$
is the identity matrix when $t=t_{n}$, it readily follows that for $t=t_{n}$%
\begin{equation*}
a\left( c^{n},v_{h}\right) =-\left( \left. \frac{D(J\widehat{c})}{Dt}%
\right\vert _{t=t_{n}},v_{h}\right) _{\Omega }+\left( f^{n},v_{h}\right)
_{\Omega },
\end{equation*}%
so, we have that%
\begin{equation*}
a\left( c_{h}^{n},v_{h}\right) =-a\left( \theta _{h}^{n},v_{h}\right)
-\left( \left. \frac{D(J\widehat{c})}{Dt}\right\vert _{t=t_{n}},v_{h}\right)
_{\Omega }+\left( f^{n},v_{h}\right) _{\Omega }.
\end{equation*}%
Next, we note that in (\ref{se4:1}) $\widehat{c}_{h}^{n-i}$ can be \
decomposed as $\widehat{c}_{h}^{n-i}=\widehat{c}^{n-i}-\widehat{\rho }^{n-i}-%
\widehat{\theta }_{h}^{n-i}$, where $\widehat{\rho }^{n-i}:=\rho
^{n-i}(X^{n,n-i}(x))$ and $\widehat{\theta }_{h}^{n-i}:=\theta
_{h}^{n-i}(X^{n,n-i}(x))$. Then, using this decomposition and the
expression of $a\left( c_{h}^{n},v_{h}\right) $, we
obtain from (\ref{se4:1}) the following equation for $\widehat{\theta }%
_{h}^{n-i}$:%
\begin{equation}
\begin{array}{r}
\displaystyle\sum_{i=0}^{q}\alpha _{i}\left( J^{n,n-i}\widehat{\theta }%
_{h}^{n-i},v_{h}\right) _{\Omega }+\Delta ta\left( \theta
_{h}^{n},v_{h}\right) =-\displaystyle\underbrace{\sum_{i=0}^{q}\alpha
_{i}\left( J^{n,n-i}\widehat{\rho }^{n-i},v_{h}\right) _{\Omega }}_{\left(
B1\right) } \\
\\
-\Delta t\displaystyle\underbrace{\left( \left. \frac{D(J\widehat{c})}{Dt}%
\right\vert _{t=t_{n}}-\frac{1}{\Delta t}\sum_{i=0}^{q}\alpha _{i}J^{n,n-i}%
\widehat{c}^{n-i},v_{h}\right) _{\Omega }}_{(B2)} \\
\\
+\displaystyle\underbrace{\sum_{i=1}^{q}\alpha _{i}\left( J^{n,n-i}(\widehat{%
c}_{h}^{n-i}-\widetilde{c}_{h}^{n-i}),v_{h}\right) _{\Omega }}_{(B3)} \\
\\
+\displaystyle\underbrace{\sum_{i=1}^{q}\alpha _{i}\left( \left( \widetilde{J%
}^{n,n-i}-J^{n,n-i}\right) \widetilde{c}_{h}^{n-i},v_{h}\right) _{\Omega }}
_{(B3)}.%
\end{array}
\label{se5:8}
\end{equation}%
We proceed to calculate an estimate of $\theta _{h}^{n}$ in the $L^{2}$-norm
using the same procedure as in the stability analysis and letting $%
v_{h}=\theta _{h}^{n}-\eta _{q}\theta _{h}^{n-1}$. We bound the terms on the
left hand side first. Let $\widehat{\Theta }^{n}:=[J^{n,n}\theta
_{h}^{n}(x),\ldots ,J^{n,n-q+1}\widehat{\theta }_{h}^{n-q+1}]^{T}$, then as
in the stability proof and with $m=n-1$
\begin{equation*}
\sum_{i=0}^{q}\alpha _{i}\left( J^{n,n-i}\widehat{\theta }%
_{h}^{n-i},v_{h}\right) _{\Omega }=\left\Vert \widehat{\Theta }%
^{n}\right\Vert _{G}-\left\Vert \widehat{\Theta }^{m}\right\Vert
_{G}+\left\Vert \sum_{i=0}^{q}\alpha _{i}J^{n,n-(q-i)}\widehat{\theta }
_{h}^{n-(q-i)}\right\Vert _{L^{2}(\Omega )}^{2}.
\end{equation*}%
As for the term $\Delta ta\left( \theta _{h}^{n},v_{h}\right) $, we already
know that (see (\ref{se4:5}))%
\begin{equation*}
\begin{array}{r}
\Delta ta\left( \theta _{h}^{n},\theta _{h}^{n}-\eta _{q}\theta
_{h}^{n-1}\right) \geq \Delta t\alpha \left( \mu \left\Vert \nabla \theta
_{h}^{n}\right\Vert _{L^{2}(\Omega )}^{2}+a_{0}\left\Vert \theta
_{h}^{n}\right\Vert _{L^{2}(\Omega )}^{2}\right) \\
\\
+\Delta t\left( F_{1}^{n}-F_{1}^{n-1}\right) +\Delta t\left(
F_{2}^{n}-F_{2}^{n-1}\right) ,%
\end{array}%
\end{equation*}%
where%
\begin{equation*}
F_{1}^{n}=\frac{\mu \eta _{q}}{2}\left\Vert \nabla \theta
_{h}^{n}\right\Vert _{L^{2}(\Omega )}^{2}\ \ \mathrm{and\ \ }F_{2}^{n}=\frac{%
a_{0}\eta _{q}}{2}\left\Vert \theta _{h}^{n}\right\Vert _{L^{2}(\Omega
)}^{2}.
\end{equation*}%
So, recalling that $\gamma _{1}=\alpha \min (\mu ,a_{0})$, we can write (\ref%
{se5:8}) as%
\begin{equation}
\begin{array}{r}
\left\Vert \widehat{\Theta }^{n}\right\Vert _{G}-\left\Vert \widehat{\Theta }%
^{m}\right\Vert _{G}+\Delta t\gamma _{1}\left\Vert \theta
_{h}^{n}\right\Vert _{H^{1}(\Omega )}^{2}+\Delta t\left(
F_{1}^{n}-F_{1}^{n-1}\right) \\
\\
+\Delta t\left( F_{2}^{n}-F_{2}^{n-1}\right) \leq B1+B2+B3.%
\end{array}
\label{se5:9}
\end{equation}%
Now, we turn our attention to estimate the terms $B1$, $B2$ and $B3.$

\bigskip

\textbf{Term} $B1=-\sum_{i=0}^{q}\alpha _{i}\left( J^{n,n-i}\widehat{\rho }%
^{n-i},v_{h}\right) _{\Omega }$

\bigskip

In order to get an optimal error estimate in the diffusive regime is
convenient to further decompose $\widehat{\rho }^{n-i}(x)$ as $\widehat{\rho
}^{n-i}(x)=\rho ^{n-i}(x)+(\widehat{\rho }^{n-i}(x)-\rho ^{n-i}(x))$ so that%
\begin{equation*}
\begin{array}{r}
-\displaystyle\sum_{i=0}^{q}\alpha _{i}\left( J^{n,n-i}\widehat{\rho }%
^{n-i},v_{h}\right) _{\Omega }=-\displaystyle\sum_{i=0}^{q}\alpha _{i}\left(
J^{n,n-i}\rho ^{n-i},v_{h}\right) _{\Omega } \\
\\
+\displaystyle\sum_{i=0}^{q}\alpha _{i}\left( J^{n,n-i}(\rho ^{n-i}-\widehat{%
\rho }^{n-i}),v_{h}\right) _{\Omega }.%
\end{array}%
\end{equation*}%
Then, we have that%
\begin{equation*}
\begin{array}{r}
\displaystyle\left\vert -\sum_{i=0}^{q}\alpha _{i}\left( J^{n,n-i}\widehat{%
\rho }^{n-i},v_{h}\right) _{\Omega }\right\vert \leq \displaystyle\left\vert
\sum_{i=0}^{q}\alpha _{i}\left( J^{n,n-i}\rho ^{n-i},v_{h}\right) _{\Omega
}\right\vert \\
\\
+\displaystyle\left\vert \sum_{i=1}^{q}\alpha _{i}\left( J^{n,n-i}(\rho
^{n-i}-\widehat{\rho }^{n-i}),v_{h}\right) _{\Omega }\right\vert \equiv
(A1+A2).
\end{array}%
\end{equation*}%
Applying the Cauchy-Schwarz inequality yields%
\begin{equation*}
A1\leq \left\Vert \sum_{i=0}^{q}\alpha _{i}J^{n,n-i}\rho ^{n-i}\right\Vert
_{L^{2}(\Omega )}\left\Vert v_{h}\right\Vert _{L^{2}(\Omega )}.
\end{equation*}%
Since $\sum_{i=0}^{q}\alpha _{i}=0$, then there is a constant $K$ such that%
\begin{equation*}
\begin{array}{r}
\displaystyle\left\vert \sum_{i=0}^{q}\alpha _{i}J^{n,n-i}\rho
^{n-i}\right\vert =\displaystyle\left\vert \sum_{i=0}^{q}\alpha _{i}\left(
J^{n,n}\rho ^{n}-J^{n,n-i}\rho ^{n-i}\right) \right\vert \leq K\displaystyle%
\int_{t_{n-q}}^{t_{n}}\left\vert \frac{\partial J\rho }{\partial t}%
\right\vert dt \\
\\
\leq C\displaystyle\int_{t_{n-q}}^{t_{n}}\left\vert \rho +\frac{\partial
\rho }{\partial t}\right\vert dt\leq C\Delta t^{1/2}\left( \displaystyle%
\int_{t_{n-q}}^{t_{n}}\left\vert \rho +\frac{\partial \rho }{\partial t}%
\right\vert ^{2}\right) ^{1/2},%
\end{array}%
\end{equation*}%
where the constant $C$ depends on $K$, $C_{\mathrm{div}}$, and $C_{J}$,
that is,$\ C=C(K,C_{\mathrm{div}},C_{J})$. So,
\begin{equation*}
\left\Vert \sum_{i=0}^{q}\alpha _{i}J^{n,n-i}\rho ^{n-i}\right\Vert
_{L^{2}(\Omega )}\leq C\Delta t^{1/2}\left( \left\Vert \rho \right\Vert
_{L^{2}(t_{n-q},t_{n};L^{2}(\Omega ))}+\left\Vert \frac{\partial \rho }{%
\partial t}\right\Vert _{L^{2}(t_{n-q},t_{n};L^{2}(\Omega ))}\right).
\end{equation*}%
Now, applying the elementary inequality it follows that there is a constant $%
\varepsilon _{1}>0$ such that%
\begin{equation}
A1\leq \frac{C}{\varepsilon _{1}}\left( \left\Vert \rho \right\Vert
_{L^{2}(t_{n-q},t_{n};L^{2}(\Omega ))}^{2}+\left\Vert \frac{\partial \rho }{%
\partial t}\right\Vert _{L^{2}(t_{n-q},t_{n};L^{2}(\Omega ))}^{2}\right)
+\varepsilon _{1}\Delta t\left( \left\Vert \theta _{h}^{n}\right\Vert
^{2}+\left\Vert \theta _{h}^{n-1}\right\Vert ^{2}\right) .  \label{se5:11}
\end{equation}%
We can estimate the term $A2$ in two ways, each one being relevant at each one of the
two regimes the problem may experiment depending on the local P\'{e}clet
number $Pe=\frac{hK(u)}{\mu }$.%
\begin{equation*}
\left( A2\right) \leq \left\{
\begin{array}{l}
\sum_{i=1}^{q}\left\vert \alpha _{i}\right\vert \left\Vert J^{n,n-i}(\rho
^{n-i}-\widehat{\rho }^{n-i})\right\Vert _{-1}\left\Vert v_{h}\right\Vert
_{H^{1}(\Omega )}\ \equiv \left( A2\right) _{1}\ \mathrm{or} \\
\\
\sum_{i=1}^{q}\left\vert \alpha _{i}\right\vert \left\Vert J^{n,n-i}(\rho
^{n-i}-\widehat{\rho }^{n-i})\right\Vert _{L^{2}(\Omega )}\left\Vert
v_{h}\right\Vert _{L^{2}(\Omega )}\equiv \left( A2\right) _{2}.%
\end{array}%
\right.
\end{equation*}%
By virtue of (\ref{se5:5}) and the elementary inequality it follows that
there is a constant$\ \varepsilon _{2}=\gamma _{1}/2$ such that%
\begin{equation*}
\left( A2\right) _{1}\leq \Delta t\displaystyle\frac{2qC_{J}^{2}K^{2}(u)%
\sum_{i=1}^{q}i^{2}\left\vert \alpha _{i}\right\vert ^{2}\left\Vert \rho
^{n-i}\right\Vert _{L^{2}(\Omega )}^{2}}{\gamma _{1}}+\Delta t\displaystyle%
\frac{\gamma _{1}}{2}\left( \left\Vert \theta _{h}^{n}\right\Vert
_{H^{1}(\Omega )}^{2}+\left\Vert \theta _{h}^{n-1}\right\Vert _{H^{1}(\Omega
)}^{2}\right) .
\end{equation*}%
Since $\left\Vert \rho ^{n-i}\right\Vert _{L^{2}(\Omega )}\leq
C_{0}h^{k+1}\left\Vert c^{n-i}\right\Vert _{H^{k+1}(\Omega )}$, where the
constant $C_{0}=c_{1}\frac{\max (\mu .a_{0})}{\min (\mu .a_{0})}$, $c_{1}$
being the approximation constant in $V_{h}$, then we can set that%
\begin{equation*}
\left( A2\right) _{1}\leq \Delta t\displaystyle\frac{%
2qC_{0}C_{J}^{2}K^{2}(u)h^{2k+2}\sum_{i=1}^{q}i^{2}\left\vert \alpha
_{i}\right\vert ^{2}\left\Vert c^{n-i}\right\Vert _{H^{k+1}(\Omega )}^{2}}{%
\gamma _{1}}+\Delta t\displaystyle\frac{\gamma _{1}}{2}\left( \left\Vert
\theta _{h}^{n}\right\Vert _{H^{1}(\Omega )}^{2}+\left\Vert \theta
_{h}^{n-1}\right\Vert _{H^{1}(\Omega )}^{2}\right) .
\end{equation*}%
Similarly, we can estimate $\left( A2\right) _{2}$ by using (\ref{se5:6})
and the elementary inequality, so there is a constant $\varepsilon _{3}$
such that
\begin{equation*}
\begin{array}{r}
(A2)_{2}\leq \displaystyle\frac{q\Delta t}{2\varepsilon _{3}}\min \left(
C_{1}^{2}\sum_{i=1}^{q}\left\vert \alpha _{i}\right\vert ^{2}\left\Vert
\nabla \rho ^{n-i}\right\Vert _{L^{2}(\Omega
)^{d}}^{2},C_{2}\sum_{i=1}^{q}\left\vert \alpha _{i}\right\vert
^{2}\left\Vert \frac{\rho ^{n-i}}{\Delta t}\right\Vert _{L^{2}(\Omega
)}^{2}\right) \\
\\
+C_{1J}\varepsilon _{3}\Delta t\left( \left\Vert \theta _{h}^{n}\right\Vert
_{L^{2}(\Omega )}^{2}+\left\Vert \theta _{h}^{n-1}\right\Vert _{L^{2}(\Omega
)}^{2}\right) ,%
\end{array}%
\end{equation*}%
where $C_{1}^{2}=K^{2}(u)C_{J}^{2}/(1+C_{J})$ and $C_{2}=C_{1J}$. Since $%
\left\Vert \nabla \rho ^{n-i}\right\Vert _{L^{2}(\Omega )^{d}}\leq
C_{0}h^{k}\left\Vert c^{n-i}\right\Vert _{H^{k+1}(\Omega )}$, then we can
set that
\begin{equation*}
\begin{array}{r}
\left( A2\right) _{2}\leq \displaystyle\frac{C_{0}q\Delta t}{2\varepsilon
_{3}}\min \left( \frac{C_{1}^{2}\Delta t^{2}}{h^{2}},C_{2}\right) \left(
\frac{h^{k+1}}{\Delta t}\right) ^{2}\sum_{i=1}^{q}\left\vert \alpha
_{i}\right\vert ^{2}\left\Vert c^{n-i}\right\Vert _{H^{k+1}(\Omega )}^{2} \\
\\
+C_{1J}\varepsilon _{3}\Delta t\left( \left\Vert \theta _{h}^{n}\right\Vert
_{L^{2}(\Omega )}^{2}+\left\Vert \theta _{h}^{n-1}\right\Vert _{L^{2}(\Omega
)}^{2}\right) .%
\end{array}%
\end{equation*}%
Thus, letting $\varepsilon _{3}=1/8C_{1J}$, and using (\ref{se5:11}) we can
write the following two estimates for $B1$.
\begin{equation}
\begin{array}{l}
B1\leq \displaystyle\frac{C}{\varepsilon _{1}}\displaystyle\left( \left\Vert \rho
\right\Vert _{L^{2}(t_{n-q},t_{n};L^{2}(\Omega ))}^{2}+\left\Vert \frac{%
\partial \rho }{\partial t}\right\Vert _{L^{2}(t_{n-q},t_{n};L^{2}(\Omega
))}^{2}\right) \\
\\
+\Delta t\displaystyle\frac{4qC_{0}C_{J}^{2}K^{2}(u)h^{2(k+1)}%
\sum_{i=1}^{q}i^{2}\left\vert \alpha _{i}\right\vert ^{2}\left\Vert
c^{n-i}\right\Vert _{H^{k+1}(\Omega )}^{2}}{\gamma _{1}} \\
\\
+\Delta t\displaystyle\frac{\gamma _{1}}{2}\left( \left\Vert \theta
_{h}^{n}\right\Vert _{H^{1}(\Omega )}^{2}+\left\Vert \theta
_{h}^{n-1}\right\Vert _{H^{1}(\Omega )}^{2}\right) +\Delta t\varepsilon
_{1}\left( \left\Vert \theta _{h}^{n}\right\Vert _{L^{2}(\Omega
)}^{2}+\left\Vert \theta _{h}^{n-1}\right\Vert _{L^{2}(\Omega )}^{2}\right)%
\end{array}
\label{se5:11.1}
\end{equation}%
or
\begin{equation}
\begin{array}{l}
B1\leq \displaystyle\frac{C}{\varepsilon _{1}}\displaystyle\left( \left\Vert \rho
\right\Vert _{L^{2}(t_{n-q},t_{n};L^{2}(\Omega ))}^{2}+\left\Vert \frac{%
\partial \rho }{\partial t}\right\Vert _{L^{2}(t_{n-q},t_{n};L^{2}(\Omega
))}^{2}\right) \\
\\
+\displaystyle4C_{1J}C_{0}q\Delta t\min \left( \frac{C_{1}^{2}\Delta t^{2}}{%
h^{2}},C_{2}\right) \left( \frac{h^{k+1}}{\Delta t}\right)
^{2}\sum_{i=1}^{q}\left\vert \alpha _{i}\right\vert ^{2}\left\Vert
c^{n-i}\right\Vert _{H^{k+1}(\Omega )}^{2} \\
\\
+\Delta t\displaystyle\left( \frac{1}{8}+\varepsilon _{1}\right) \left(
\left\Vert \theta _{h}^{n}\right\Vert _{L^{2}(\Omega )}^{2}+\left\Vert
\theta _{h}^{n-1}\right\Vert _{L^{2}(\Omega )}^{2}\right).%
\end{array}
\label{se5:11.2}
\end{equation}

\bigskip

\textbf{Term} $B2=-\Delta t\displaystyle\left( \left. \frac{D(J\widehat{c})}{%
Dt}\right\vert _{t=t_{n}}-\frac{1}{\Delta t}\sum_{i=0}^{q}\alpha
_{i}J^{n,n-1}\widehat{c}^{n-i},v_{h}\right) _{\Omega }$

\bigskip

By the Cauchy-Schwarz and elementary inequalities it results that%
\begin{equation*}
B2\leq \frac{\Delta t}{2\varepsilon _{4}}\left\Vert \left. \frac{D(J\widehat{%
c})}{Dt}\right\vert _{t=t_{n}}-\frac{1}{\Delta t}\sum_{i=0}^{q}\alpha
_{i}J^{n,n-1}\widehat{c}^{n-i}\right\Vert _{L^{2}(\Omega )}^{2}+\Delta
t\varepsilon _{4}\left( \left\Vert \theta _{h}^{n}\right\Vert _{L^{2}(\Omega
)}^{2}+\left\Vert \theta _{h}^{n-1}\right\Vert _{L^{2}(\Omega )}^{2}\right) .
\end{equation*}%
To estimate the first term on the right hand side of this inequality
we apply the same arguments as in the proof of Theorem 10.1 of
Chapter 10 of \cite{Tho} to obtain that
\begin{equation*}
\Delta t\left\Vert \left. \frac{D(J\widehat{c})}{Dt}\right\vert _{t=t_{n}}-%
\frac{1}{\Delta t}\sum_{i=0}^{q}\alpha _{i}J^{n,n-i}\widehat{c}%
^{n-i}\right\Vert _{L^{2}(\Omega )}^{2}\leq C\Delta t^{2q}\left\Vert \frac{%
D^{q+1}\left( J\widehat{c}\right) }{Dt^{q+1}}\right\Vert
_{L^{2}(t_{n-q},t_{n};L^{2}(\Omega ))}^{2},
\end{equation*}%
so,%
\begin{equation}
B2\leq \frac{C\Delta t^{2q}}{\varepsilon _{4}}\left\Vert \frac{D^{q+1}\left(
J\widehat{c}\right) }{Dt^{q+1}}\right\Vert
_{L^{2}(t_{n-q},t_{n};L^{2}(\Omega ))}^{2}+\Delta t\varepsilon _{4}\left(
\left\Vert \theta _{h}^{n}\right\Vert _{L^{2}(\Omega )}^{2}+\left\Vert
\theta _{h}^{n-1}\right\Vert _{L^{2}(\Omega )}^{2}\right) .  \label{se5:12}
\end{equation}

\bigskip

\textbf{Term} $B3=\displaystyle\sum_{i=1}^{q}\alpha _{i}\left( J^{n,n-i}(%
\widehat{c}_{h}^{n-i}-\widetilde{c}_{h}^{n-i})+\left( \widetilde{J}%
^{n,n-i}-J^{n,n-i}\right) \widetilde{c}_{h}^{n-i},v_{h}\right) _{\Omega }$

\bigskip

To estimate this term we go back to (\ref{I3}) and (\ref{I4}) and therefore
we can set%
\begin{equation*}
B3\leq \Delta tC(u)h^{k+1}\sum_{i=1}^{q}\left\vert \alpha _{i}\right\vert
\left( \left\Vert \nabla c_{h}^{n-i}\right\Vert _{L^{2}(\Omega
)^{d}}+\left\Vert c_{h}^{n-i}\right\Vert _{L^{2}(\Omega )}\right) \left\Vert
v_{h}\right\Vert _{L^{2}(\Omega )}.
\end{equation*}%
Replacing $c_{h}^{n-i}$ by $c^{n-i}-\rho ^{n-i}-\theta _{h}^{n-i}$, noting
that $\left\Vert v_{h}\right\Vert _{L^{2}(\Omega )}\leq \left\Vert \theta
_{h}^{n}\right\Vert _{L^{2}(\Omega )}+\left\Vert \theta
_{h}^{n-1}\right\Vert _{L^{2}(\Omega )}$ and applying the elementary
inequality it follows that%
\begin{equation*}
\begin{array}{r}
B3\leq \Delta tC^{2}(u)h^{2(k+1)}\displaystyle\sum_{i=1}^{q}\left\vert
\alpha _{i}\right\vert \left( \left\Vert c^{n-i}\right\Vert _{H^{1}(\Omega
)}^{2}+\left\Vert \rho ^{n-i}\right\Vert _{H^{1}(\Omega )}^{2}+\left\Vert
\theta _{h}^{n-i}\right\Vert _{H^{1}(\Omega )}^{2}\right) \\
\\
+\Delta t\displaystyle\sum_{i=1}^{q}\left\vert \alpha _{i}\right\vert \left(
\left\Vert \theta _{h}^{n}\right\Vert _{L^{2}(\Omega )}^{2}+\left\Vert
\theta _{h}^{n-1}\right\Vert _{L^{2}(\Omega )}^{2}\right) .%
\end{array}%
\end{equation*}%
Making use of the inverse inequality $\left\Vert \theta
_{h}^{n-i}\right\Vert _{H^{1}(\Omega )}^{2}\leq ch^{-2}\left\Vert \theta
_{h}^{n-i}\right\Vert _{L^{2}(\Omega )}^{2}$ we can write that%
\begin{equation}
B3\leq \Delta tC^{2}(u)h^{2(k+1)}\sum_{i=1}^{q}\left\vert \alpha
_{i}\right\vert \left( \left\Vert c^{n-i}\right\Vert _{H^{1}(\Omega
)}^{2}+\left\Vert \rho ^{n-i}\right\Vert _{H^{1}(\Omega )}^{2}\right)
+\Delta t\sum_{i=0}^{q}\zeta _{i}\left\Vert \theta _{h}^{n-i}\right\Vert
_{L^{2}(\Omega )}^{2},  \label{se5:14}
\end{equation}%
Now, substituting the estimates (\ref{se5:11.1}),\ (\ref{se5:11.2}), (\ref%
{se5:12}) and (\ref{se5:14}) in (\ref{se5:9}) and setting $\varepsilon
_{1}=1/8$ and $\varepsilon _{4}=1/4$, it follows that%
\begin{equation}
\left\{
\begin{array}{l}
\left\Vert \widehat{\Theta }^{n}\right\Vert _{G}-\left\Vert \widehat{\Theta }%
^{m}\right\Vert _{G}+\Delta t\frac{\gamma _{1}}{2}\left( \left\Vert \theta
_{h}^{n}\right\Vert _{H^{1}(\Omega )}^{2}-\left\Vert \theta
_{h}^{n-1}\right\Vert _{H^{1}(\Omega )}^{2}\right) +\Delta t\left(
F_{1}^{n}-F_{1}^{n-1}\right) \\
\\
+\Delta t\left( F_{2}^{n}-F_{2}^{n-1}\right) \leq h^{2(k+1)}(K_{1}+\Delta
tK_{2})+\Delta t8qC_{0}\min \left( L_{1},L_{2},L_{3}\right) \left( \frac{
h^{k+1}}{\Delta t}\right) ^{2} \\
\\
+C\Delta t^{2q}\left\Vert \frac{D^{q+1}Jc}{Dt^{q+1}}\right\Vert
_{L^{2}(t_{n-q},t_{n};L^{2}(\Omega ))}^{2}+\Delta t\sum_{i=0}^{q}\zeta
_{i}^{\ast }\left\Vert \theta _{h}^{n-i}\right\Vert _{L^{2}(\Omega )}^{2},%
\end{array}%
\right.  \label{se5:15}
\end{equation}%
where $\zeta _{0}^{\ast }=\zeta _{0}+1/2$, $\zeta _{1}^{\ast }=\zeta
_{1}+1/2 $ and \ for $i\geq 2$, $\zeta _{i}^{\ast }=\zeta _{i}$,%
\begin{equation*}
\begin{array}{l}
L_{1}=\frac{C_{J}^{2}K^{2}(u)\Delta t^{2}}{\gamma _{1}}\sum_{i=1}^{q}i^{2}%
\left\vert \alpha _{i}\right\vert ^{2}\left\Vert c^{n-i}\right\Vert
_{H^{k+1}(\Omega )}^{2},\  \\
\\
\left( L_{2},L_{3}\right) =\left( \frac{C_{J}^{2}}{1+C_{J}}\frac{%
K^{2}(u)\Delta t^{2}}{h^{2}},C_{1J}\right) \sum_{i=1}^{q}\left\vert \alpha
_{i}\right\vert ^{2}\left\Vert c^{n-i}\right\Vert _{H^{k+1}(\Omega )}^{2},%
\end{array}%
\end{equation*}
and$\ $the constants%
\begin{equation*}
\begin{array}{l}
K_{1}=C\left( \left\Vert c\right\Vert _{L^{2}(t_{n-q},t_{n};H^{k+1}(\Omega
))}^{2}+\left\Vert \frac{\partial c}{\partial t}\right\Vert
_{L^{2}(t_{n-q},t_{n};H^{k+1}(\Omega ))}^{2}\right) , \\
\\
K_{2}=C\left( \left\Vert \nabla c^{n-i}\right\Vert _{L^{2}(\Omega
)^{d}}^{2}+\left\Vert \nabla \rho ^{n-i}\right\Vert _{L^{2}(\Omega
)^{d}}^{2}\right) .%
\end{array}%
\end{equation*}%
Summing on both sides of (\ref{se5:15}) from $n=q$ up to $n=N$ yields

\begin{equation*}
\left\{
\begin{array}{l}
\left\Vert \widehat{\Theta }_{h}^{N}\right\Vert _{G}-\left\Vert \widehat{%
\Theta }_{h}^{q-1}\right\Vert _{G}+\Delta t\frac{\gamma _{1}}{2}\left(
\left\Vert \theta _{h}^{N}\right\Vert _{H^{1}(\Omega )}^{2}-\left\Vert
\theta _{h}^{q-1}\right\Vert _{H^{1}(\Omega )}^{2}\right) +\Delta t\left(
F_{1}^{N}-F_{1}^{q-1}\right) \\
\\
+\Delta t\left( F_{2}^{N}-F_{2}^{q-1}\right) \leq h^{2(k+1)}\left( \overline{%
K}_{1}+\overline{K}_{2}\right) +C\Delta t^{2q}\left\Vert \frac{D^{q+1}Jc}{%
Dt^{q+1}}\right\Vert _{L^{2}(0,T;L^{2}(\Omega ))}^{2} \\
\\
+\overline{K}_{3}T\min \left( \frac{K^{2}(u)\Delta t^{2}}{\gamma _{1}},\frac{%
K^{2}(u)\Delta t^{2}}{G_{q}(1+C_{J})h^{2}},\frac{C_{1J}}{G_{q}(1+C_{J})}%
\right) \left( \frac{h^{k+1}}{\Delta t}\right) ^{2}+\Delta t\gamma
_{0}\sum_{n=0}^{N}\left\Vert \theta _{h}^{n}\right\Vert _{L^{2}(\Omega
)}^{2},%
\end{array}%
\right.
\end{equation*}%
where%
\begin{equation*}
\left\{
\begin{array}{l}
\overline{K}_{1}=C\left( \left\Vert c\right\Vert _{L^{2}(0,T;H^{k+1}(\Omega
))}^{2}+\left\Vert \frac{\partial c}{\partial t}\right\Vert
_{L^{2}(0,T;H^{k+1}(\Omega ))}^{2}\right) , \\
\\
\overline{K}_{2}=C\left( \left\Vert \nabla c\right\Vert
_{l^{2}(0,T;L^{2}(\Omega )^{d})}^{2}+h^{2k}\left\Vert c\right\Vert
_{l^{2}(0,T;H^{k+1}(\Omega ))}^{2}\right) , \\
\\
\overline{K}_{3}=8qC_{0}C_{J}^{2}\left\Vert c\right\Vert
_{l^{2}(0,T;H^{k+1}(\Omega ))}^{2}\sum_{i=1}^{q}i^{2}\alpha _{i}^{2},\ G_{q}=%
\displaystyle\frac{\sum_{i=1}^{q}i^{2}\alpha _{i}^{2}}{\sum_{i=1}^{q}\alpha
_{i}^{2}}.%
\end{array}%
\right.
\end{equation*}%
Next, as in the proof of the stability theorem, let%
\begin{equation*}
\left\Vert \widehat{\Theta }_{h}^{N}\right\Vert _{G}\geq \vartheta \left(
\left\Vert \theta _{h}^{N}\right\Vert _{L^{2}(\Omega )}^{2}+\ldots
+\left\Vert J^{N,N-q+1}\widehat{\theta }_{h}^{N-q+1}\right\Vert
_{L^{2}(\Omega )}^{2}\right)
\end{equation*}%
and%
\begin{equation*}
\left\Vert \widehat{\Theta }_{h}^{q-1}\right\Vert _{G}\leq \xi \left(
\left\Vert \theta _{h}^{0}\right\Vert _{L^{2}(\Omega )}^{2}+\ldots
+\left\Vert \theta _{h}^{q-1}\right\Vert _{L^{2}(\Omega )}^{2}\right) ,
\end{equation*}%
then it follows that%
\begin{equation}
\left\Vert \theta _{h}^{N}\right\Vert _{L^{2}(\Omega )}^{2}+\Delta t\beta
_{1}\left\Vert \theta _{h}^{N}\right\Vert _{H^{1}(\Omega )}^{2}\leq
R1+R2+R3+\Delta t\vartheta ^{-1}\lambda _{0}\sum_{n=0}^{N}\left\Vert \theta
_{h}^{n}\right\Vert _{L^{2}(\Omega )}^{2},  \label{se5:16}
\end{equation}%
where%
\begin{equation*}
\begin{array}{l}
R1=\vartheta ^{-1}\left( \xi \sum_{l=0}^{q-1}\left\Vert \theta
_{h}^{l}\right\Vert _{L^{2}(\Omega )}^{2}+\Delta t\frac{\eta _{q}\max (\mu
.a_{0})}{2}\left\Vert \theta _{h}^{q-1}\right\Vert _{H^{1}(\Omega
)}^{2}\right) , \\
\\
R2=\vartheta ^{-1}\left( \left( \overline{K}_{1}+\overline{K}_{2}\right)
h^{2(k+1)}+C\Delta t^{2q}\left\Vert \frac{D^{q+1}Jc}{Dt^{q+1}}\right\Vert
_{L^{2}(0,T;L^{2}(\Omega ))}^{2}\right), \\
\\
R3=\vartheta ^{-1}\overline{K}_{3}\min \left( \frac{K^{2}(u)\Delta t^{2}}{%
\gamma _{1}},\frac{K^{2}(u)\Delta t^{2}}{G_{q}(1+C_{J})h^{2}},\frac{C_{1J}}{%
(1+C_{J})}\right) \left( \frac{h^{k+1}}{\Delta t}\right) ^{2}.
\end{array}%
\end{equation*}%
Applying Gronwall inequality in (\ref{se5:16}) and noting that by virtue of
the triangle inequality, for all $n$ $\left\Vert e^{n}\right\Vert
_{L^{2}(\Omega )}\leq \left\Vert \rho ^{n}\right\Vert _{L^{2}(\Omega
)}+\left\Vert \theta _{h}^{n}\right\Vert _{L^{2}(\Omega )}$ and $\left\Vert
\theta _{h}^{n}\right\Vert _{H^{1}(\Omega )}\leq \left\Vert e^{n}\right\Vert
_{H^{1}(\Omega )}+\left\Vert \rho ^{n}\right\Vert _{H^{1}(\Omega )}$, the
result (\ref{se5:7.1}) follows.
\end{proof}

\section{Numerical experiments}

\label{sec:numerical-experiments}

Next, we perform some numerical experiments to study the error
behavior of the NCLG methods and compare it with that of the conventional
non conservative LG methods; also, we study the mass conservation error of
both methods. For this purpose, we define the normalized $L^{2}$ and mass errors at the final time $%
T$ as
\begin{equation*}
e_{L^{2}}:=\left[ \dfrac{\int_{\Omega }\left[ c_{h}(x,T)-c(x,T)\right]
^{2}d\Omega }{\int_{\Omega }c^{2}(x,\,T)d\Omega }\right]
^{1/2},\quad e_{m}:=\dfrac{\left\vert \int_{\Omega}\left[
c_{h}(x,T)-c(x,T)\right] d\Omega \right\vert }{\int_{\Omega}c(x,T)d\Omega}.
\end{equation*}

The test was originally proposed by Rui and Tabata \cite{RT2} to assess the
performance of their first-order conservative Lagrange-Galerkin method. It
consists of an advection-diffusion problem posed on $Q_{T}:=\Omega \times
(0,T)=(-1,1)^{2}\times (0,0.5)$ with the diffusion coefficient $\mu =0.01$, $a_{0}=0$
and the vector velocity
\begin{equation*}
u(x,t)=\left[ 1+\sin \left( t-x_{1}\right) ,1+\sin \left( t-x_{2}\right) %
\right] ^{\text{T}}.
\end{equation*}%
The boundary conditions are homogeneous Neumann conditions on the whole
boundary. The analytical solution of this problem is%
\begin{equation*}
c(x,t)=\exp \left( -\dfrac{2-\cos \left( t-x_{1}\right) -\cos \left(
t-x_{2}\right) }{\mu }\right) ,\ 0\leq t\leq T.
\end{equation*}%
We show in Fig. \ref{fig:TestII-solution} some snapshots of the numerical
solution calculated with $h=0.052$, $\Delta t=0.01$, $q=4$ and $k=5$.
Despite there is a small flux through the boundary, the mass can be
considered approximately constant along time.


\begin{figure}[th!]
\begin{center}
\includegraphics[width=0.21\linewidth, keepaspectratio]{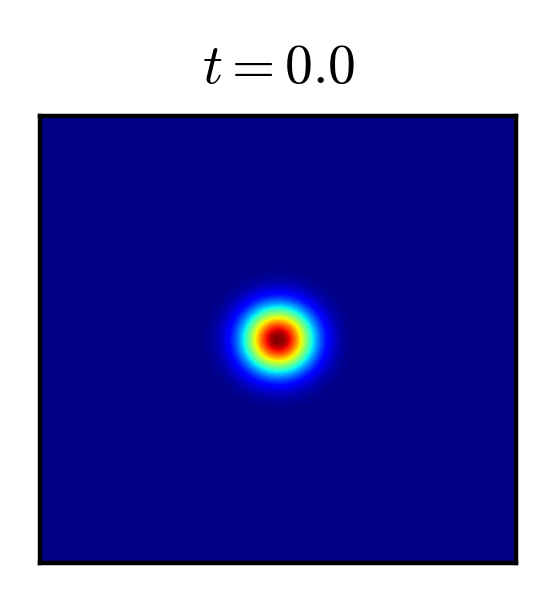}
\includegraphics[width=0.21\linewidth, keepaspectratio]{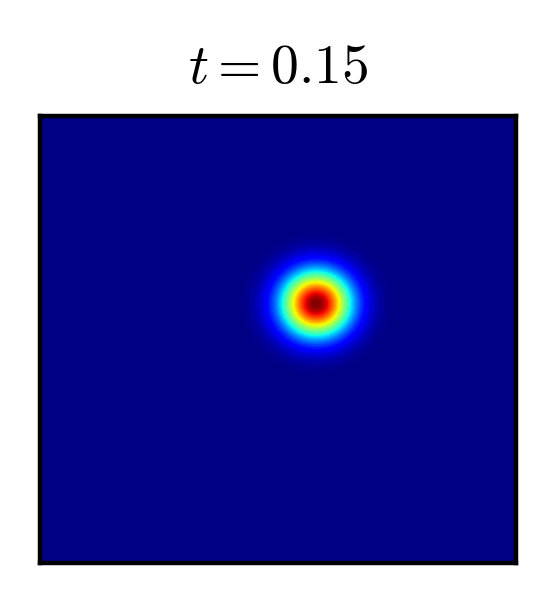} %
\includegraphics[width=0.21\linewidth, keepaspectratio]{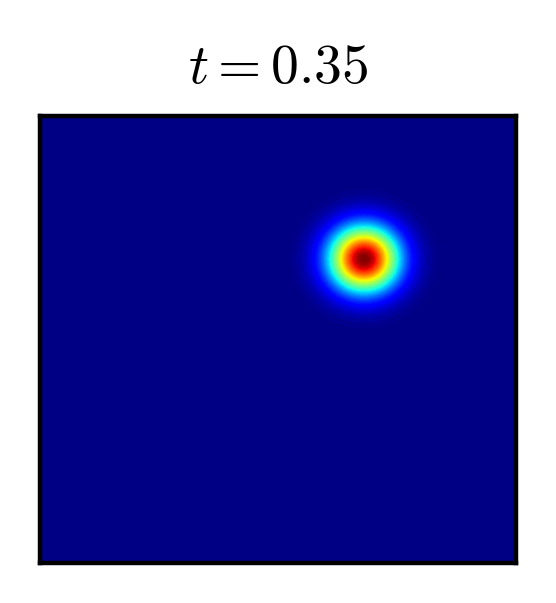} %
\includegraphics[width=0.21\linewidth, keepaspectratio]{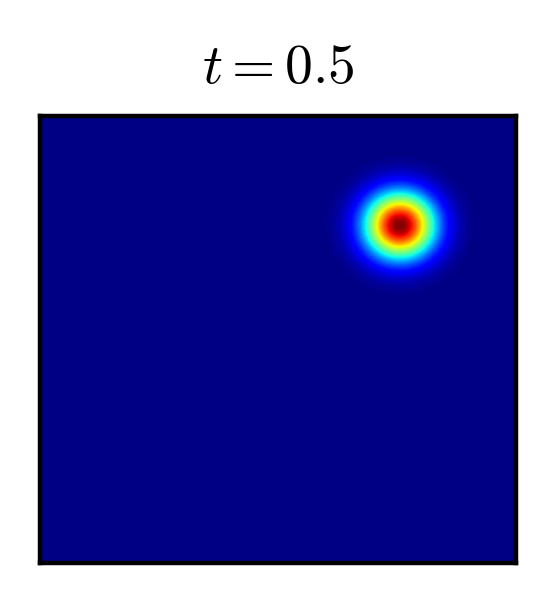} %
\includegraphics[width=0.10\linewidth, keepaspectratio]{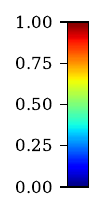}
\end{center}
\caption{Numerical solution.} \label{fig:TestII-solution}
\end{figure}

The dependence of the errors on the time step $\Delta t$ for a fixed
$h$ is illustrated in Fig. \ref{fig:TestII-Deltat}. We observe that when $\Delta t$ is sufficiently small so that the assumption \textbf{A5} of Theorem \ref{error} is satisfied, the $L^{2}$
error behaves according to Remark \ref{regimes} as $O(h^{k}+\Delta t^{q})$; so, the dominant part of the
error for $q=1,2,3,4$ is the term $\Delta t^{q}$ because $k=5$; however, there are threshold values $\Delta t_{q}$ such that for $\Delta t \leq \Delta t_{q}$ the error is $0(h^{k})$, for instance, $\Delta t_{5}\simeq 2\cdot 10^{-2}$ and $\Delta t_{4}\simeq 10^{-2}$. The mass conservation errors behave approximately as $O(\Delta t^{q})$ for the non-conservative methods, which are significantly larger than those of  the quasi-conservative methods when $1 \leq q \leq 4$. In the latter methods and for $k \geq 3$ the mass conservation errors are approximately $O(h^{k})$; however, for $k=1$ and $2$ this rate is reached when $\Delta t$ is sufficiently small, say $\Delta t\simeq 10^{-2}$.

\begin{figure}[th!]
\begin{center}
\includegraphics[width=0.49\linewidth, keepaspectratio]{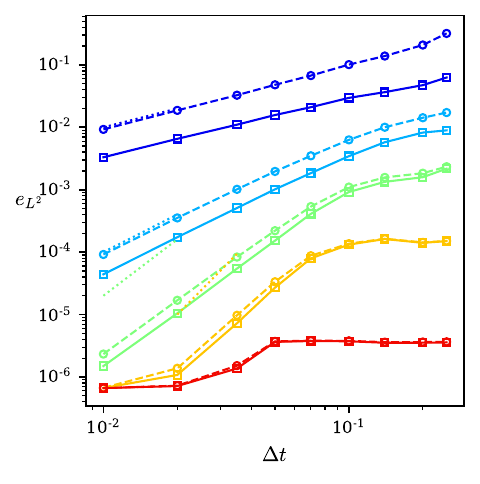}
\includegraphics[width=0.49\linewidth, keepaspectratio]{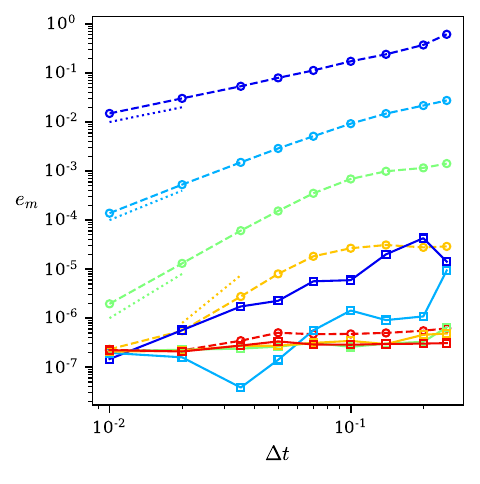} 
\caption[]{$L^2$ and mass errors with respect to $\Delta
t$ for $k=5$ and $h=0.052$. Legend: (\colorline{blue}) $q=1$,
(\colorline{medblue}) $q=2$, (\colorline{green}) $q=3$,
(\colorline{orange}) $q=4$, (\colorline{red}) $q=5$,
(\sqsolid{black}) present method, (\circdashed{black})
non-conservative scheme.} \label{fig:TestII-Deltat}
\end{center}
\end{figure}

Fig. \ref{fig:TestII-h} illustrates the behavior of the errors with respect
to the mesh size $h$ for a fixed time step $\Delta t$. In these examples the $L^{2}$
errors are almost coincident both for the conservative and the
non-conservative methods. The $L^{2}$ error is now $O(h^{k+1}+\Delta t^{q})
$; hence, for $h$ sufficiently small and $k \geq 4$ the dominant part of the error is $
O(\Delta t^{q})$, this means that the error due to the time discretization
becomes dominant. The mass conservation
errors decrease when $h$ decreases, but with a somewhat oscillating behavior for $k=1$ and $2$ due perhaps to the fact that either $h$ or $\Delta t$ are not sufficiently small; however, for $k\geq 3$ and $h$ sufficiently small the error is $0(\Delta t^{q})$l as expected.

\begin{figure}[th!]
\centering
\includegraphics[width=0.49\linewidth, keepaspectratio]{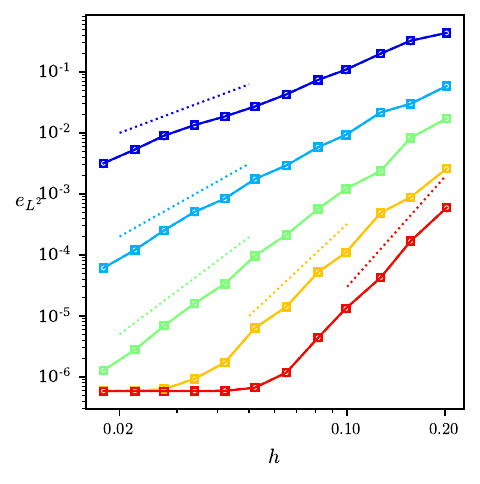}
\includegraphics[width=0.49\linewidth, keepaspectratio]{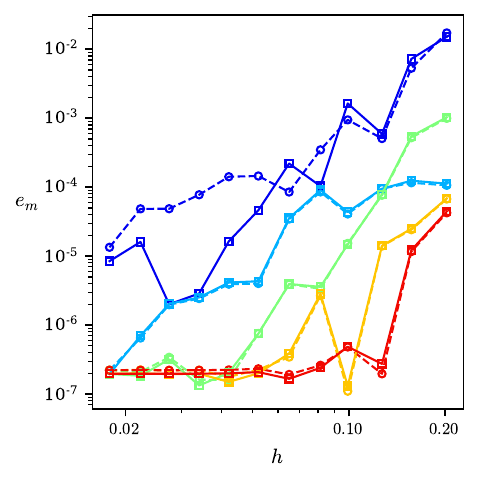}
\caption[]{$L^2$ and mass errors with respect to $h$ for
$q=4$ and $\Delta t=0.01$. Legend: (\colorline{blue}) $k=1$,
(\colorline{medblue}) $k=2$, (\colorline{green}) $k=3$,
(\colorline{orange}) $k=4$, (\colorline{red}) $k=5$,
(\sqsolid{black}) present method, (\circdashed{black})
non-conservative scheme.} \label{fig:TestII-h}
\end{figure}

Finally, we have repeated the experiments of Rui and Tabata \cite{RT2} in
order to check the performance of our method versus the method proposed by
these authors, which is based on the formulation (\ref{s2:13}) with $q=1$.
In their experiments, the mesh is constructed by dividing each edge of the
square domain into $N$ uniform segments, and $\Delta t=2/N$, for several
values of $N$. The initial condition is taken as $c_{h}^{0}=L_{h}c^{0}$,
with $L_{h}$ the Lagrange interpolation operator onto the finite element
space $V_{h}$. Then, we measure the relative errors by the formulas%
\begin{equation*}
\widehat{e}_{L^{2}}=\displaystyle\dfrac{\max_{n}\sqrt{\int_{\Omega }\left(
c_{h}^{n}-L_{h}c^{n}\right) ^{2}d\Omega }}{\max_{n}\sqrt{\int_{\Omega
}\left( L_{h}c^{n}\right) ^{2}d\Omega }},\quad \widehat{e}_{m}=\displaystyle%
\dfrac{\left\vert \int_{\Omega }\left( c_{h}^{N_{T}}-L_{h}c^{N_{T}}\right)
d\Omega \right\vert }{\left\vert \int_{\Omega }L_{h}c^{N_{T}}d\Omega
\right\vert },
\end{equation*}%
where $N_{T}=\lfloor T/\Delta t\rfloor $ denotes the last simulated time
before $T$. Note that the authors in \cite{RT2} employ a first-order-in-time
method with linear elements. Since $\Delta t\sim h$ in these experiments,
the $L^{2}$ error must converge as $O\left( \Delta t^{\min \{k,q\}}\right) $.

As seen in Fig. \ref{fig:TestII-RT}, the $L^{2}$ errors of our method are
slightly larger than those of \cite{RT2} for the case $k=q=1$. However, the
former allows to consider higher-order discretizations that provide clearly
better results. The $L^{2}$ error converges with $\Delta t$ as expected for $%
k=q\leq 4$; for $k=q=5$, the convergence is more irregular, although the
errors are still smaller. The mass conservation errors are also smaller in
our method.

\begin{figure}[th!]
\centering
\includegraphics[width=0.49\linewidth, keepaspectratio]{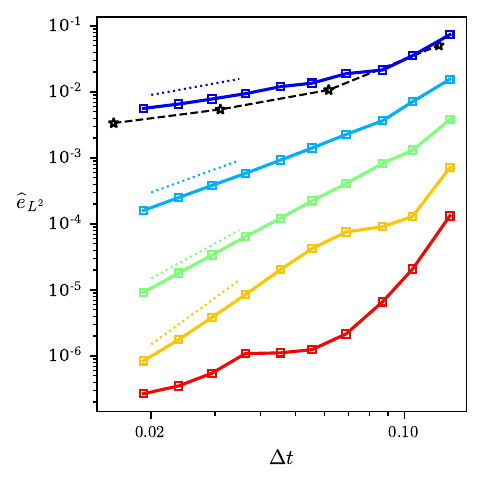}
\includegraphics[width=0.49\linewidth, keepaspectratio]{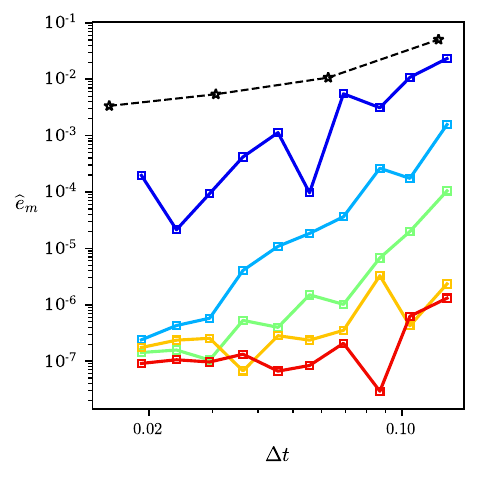}
\caption[]{Relative $L^2$ and mass errors for the numerical experiments in \cite{RT2}. Legend:
(\colorline{blue}) $k=q=1$, (\colorline{medblue}) $k=q=2$, (\colorline{green}) $k=q=3$, (\colorline{orange}) $k=q=4$, (\colorline{red}) $k=q=5$.,
($\vcenter{\hbox{\tikz\draw[thick,dashed,black] (0,0)--(0.5,0) node[midway]{$\medstar$};}}$) Rui-Tabata method \cite{RT2},
(\sqsolid{black}) present method.}
\label{fig:TestII-RT}
\end{figure}

\section*{Acknowledgements}

This research has been partially funded by grant PGC-2018-097565-B100 of
Ministerio de Ciencia, Innovaci\'{o}n y Universidades of Spain and of the
European Regional Development Fund.

\begin{center}
\pagebreak {\LARGE Appendix}
\end{center}

\textbf{Finite element discretizations and }$d$\textbf{-curved simplices}

\bigskip

The NCLG methods virtually move the elements of the fixed mesh backward in
time thus generating elements with curved $d-$faces. Since we define finite
element spaces associated with the meshes formed of such curve elements,
then for the sake of completeness of the paper, we collect in this appendix,
following the theory on curved elements of Bernardi \cite{Bern}, some
results which are relevant for us.

The triple $(T,P(T),\Sigma _{T})$ is a finite element \cite[Chapter 2.3]{Ci}%
, if: (i) $T$ is the closure of an open domain with Lipschitz piecewise
smooth boundary,\ (ii) $P(T)$ is a finite dimensional space of real-valued
functions defined over $T$, and (iii) $\Sigma _{T}$ is a set of continuous
linear forms defined on $P(T)$ with support on $T$, and such that $\Sigma
_{T}$ is $P(T)-$unisolvent, i.e., for any $\phi \in \Sigma _{T}$ there is
one and only one $v\in P_{T}$ such that $\phi (v)=1$ and for any other $\phi
^{\prime }\in \Sigma _{T}$,\ $\phi ^{\prime }\neq \phi $, $\phi ^{\prime
}(v)=0$. Moreover we shall also consider the reference finite element $(%
\widehat{T},\widehat{P}(\widehat{T}),\widehat{\Sigma }_{\widehat{T}})$ where
$\widehat{T}$ is the unit $d-$simplex, $\widehat{P}(\widehat{T})$ is the set
of polynomials of degree $\leq k,\ k$ $\geq 0$ being an integer, and $%
\widehat{\Sigma }_{\widehat{T}}$ is a set of continuous linear forms with
support on $\widehat{T}.$

\begin{definition}
\label{definition1} (\cite{Bern}, Definition 2.1) The $d-$simplex $T$ is
said to be a curved $d-$simplex\ if there exists a $C^{1}$ mapping $F_{T}:%
\widehat{T}\rightarrow T$ such that

\begin{equation}
F_{T}=\overline{F}_{T}+\Theta _{T},  \label{curvemapping}
\end{equation}%
where
\begin{equation*}
\begin{array}{c}
\overline{F}_{T}:\widehat{T}\rightarrow \mathbb{R}^{d} \\
\widehat{x}\rightarrow \overline{B}_{T}\widehat{x}+\overline{b}_{T},\  \\
\overline{B}_{T}\in \mathcal{L}(\mathbb{R}^{d})\ \ \ \text{ and }\ \ \
\overline{b}_{T}\in \mathbb{R}^{d},%
\end{array}%
\end{equation*}%
is an invertible affine mapping and $\Theta _{T}:\widehat{T}\rightarrow
\mathbb{R}^{d}$\ is a $C^{1}$\ mapping which satisfies
\begin{equation}
c_{T}=\sup_{\widehat{x}\in \widehat{T}}\parallel D\Theta _{T}(\widehat{x})%
\overline{B}_{T}^{-1}\parallel <1,  \label{curvemapping2}
\end{equation}%
where $\left\Vert \cdot \right\Vert $ denotes the two-norm of the matrix.
\end{definition}

\bigskip

Note that by the definition of $F_{T}$, a curved element $T$ can be
considered as a perturbation to the element $\overline{T}=\overline{F}_{T}(%
\widehat{T})$.\ As usual, $h_{T}$, $h_{\overline{T}}$ and $\widehat{h}$ will
denote $diam(T),\;diam(\overline{T})$ and $diam(\widehat{T})$ respectively.
Moreover the definition allows the following result.

\begin{lemma}
\label{lemma1_apen} (\cite{Bern}, Lemma 2.1) The mapping $F_{T}$ is a $C^{1}$%
-diffeomorphism from $\widehat{T}$ onto $T$ that satisfies%
\begin{equation}
\sup_{\widehat{x}\in \widehat{T}}\left\Vert DF_{T}(\widehat{x})\right\Vert
\leq (1+c_{T})\left\Vert \overline{B}_{T}\right\Vert ,  \label{ft1}
\end{equation}

\begin{equation}
\sup_{x\in T}\left\Vert DF_{T}^{-1}(x)\right\Vert \leq
(1-c_{T})^{-1}\left\Vert \overline{B}_{T}^{-1}\right\Vert ,  \label{ft2}
\end{equation}

\begin{equation}
\forall \widehat{x}\in \widehat{T},\ \ (1-c_{T})^{d}\left\vert \det
\overline{B}_{T}\right\vert \leq \left\Vert \det DF_{T}(\widehat{x}%
)\right\Vert \leq (1+c_{T})^{d}\left\vert \det \overline{B}_{T}\right\vert .
\label{ft3}
\end{equation}%
Furthermore. let $\widehat{\varrho }=$sup$\{$diam$(\widehat{S})$,$\ \widehat{%
S}$ a ball contained in$\ \widehat{T}\}$ and $\rho _{\overline{T}}$ the
diameter of the maximum inscribed ball in $\overline{T}$, since [\cite{Ci}, Thm.
3.1.3]%
\begin{equation}
\left\Vert \overline{B}_{T}\right\Vert \leq \frac{h_{\overline{T}}}{\widehat{%
\varrho }},\ \ \left\Vert \overline{B}_{T}^{-1}\right\Vert \leq \frac{%
\widehat{h}}{\rho _{\overline{T}}}\ \ \mathrm{and}\ \ \left\vert \det
\overline{B}_{T}\right\vert =\frac{meas(\overline{T})}{meas(\widehat{T})},
\label{ft4}
\end{equation}%
it readily follows from (\ref{ft3}) that%
\begin{equation}
\sup_{\widehat{x}\in \widehat{T}}\left\vert \det DF_{T}(\widehat{x}%
)\right\vert \leq \frac{1}{meas(\widehat{T})}\left( \frac{1+c_{T}}{1-c_{T}}%
\right) ^{d}meas\left( T\right) .  \label{ft5}
\end{equation}%
Also, there exist positive constants $c_{1}$ and $c_{2}$ depending on the $%
dx-$measure of the unit ball in $\mathbb{R}^{d}$ and the constant $C_{T}$ such that%
\begin{equation*}
c_{1}\rho _{\overline{T}}^{d}\leq meas(T)\leq c_{2}h_{\overline{T}}^{d}\text{%
.}
\end{equation*}
\end{lemma}

\begin{definition}
\label{definition1.5}(\cite{Bern}, Definition 2.2) A curved $d-$simplex is
of class $C^{m}$, $m\geq 1$, if the mapping $F_{T}$ is of class $C^{m}$ on $%
\widehat{T}$.
\end{definition}

For curved $d-$simplices of class $C^{m}$ there are constants
\begin{equation}
c_{l}(T)=\sup_{\widehat{x}\in \widehat{T}}\left\Vert D^{l}F_{T}(\widehat{x}%
)\right\Vert \cdot \left\Vert \overline{B}_{T}\right\Vert ^{-l},\ \ 2\leq
l\leq m,  \label{ft6}
\end{equation}%
and constants $c_{-l}(T)$, which depend continuously on $c_{T},\
c_{2}(T),\ldots ,c_{m}(T)$ such that%
\begin{equation}
\sup_{x\in T}\left\Vert D^{l}F_{T}^{-1}(x)\right\Vert \leq c_{-l}\left\Vert
\overline{B}_{T}\right\Vert ^{2(l-1)}\left\Vert \overline{B}%
_{T}^{-1}\right\Vert ^{l}.  \label{ft7}
\end{equation}%
We have the following lemma.

\begin{lemma}
\label{lemma7} (\cite{Bern}, Lemma 2.3) Assume that $T$ is a curved element
of class $m$ and let $k$ be an integer, $0\leq k\leq m$, and $p\in \lbrack
1,\infty ]$. A function $v$ is in $W^{k,p}(T)$ if and only if $\widehat{v}%
=v\circ F_{T}\in W^{k,p}(\widehat{T})$. Moreover, the following inequalities
hold.
\begin{equation}
\left\{
\begin{array}{l}
\mid v\mid _{W^{k,p}(T)}\leq C\mid \det \overline{B}_{T}\mid ^{1/p}\left(
\sum_{r=0}^{k}\parallel \overline{B}_{T}^{-1}\parallel ^{r}\mid \widehat{v}%
\mid _{W^{r,p}(\widehat{T})}\right) , \\
\\
\mid \widehat{v}\mid _{W^{k,p}(\widehat{T})}\leq C\mid \det \overline{B}%
_{T}\mid ^{-1/p}\parallel \overline{B}_{T}\parallel ^{k}\left(
\sum_{r=0}^{k}\mid v\mid _{W^{r,p}(T)}\right) .%
\end{array}%
\right.  \label{approx1}
\end{equation}%
where the constants $C$ depend continuously on $c_{T}$ (see \ref%
{curvemapping2}) and the constants $c_{l}.$
\end{lemma}

\begin{definition}
\label{definition2} (\cite{Bern}, Definition 2.4) The triplet $%
(T,P(T),\Sigma _{T})$ is said to be a curved finite element of order $m$ if:

\begin{description}
\item[(1)] $T$ is a curved $d-$simplex of class $C^{m+1}$, $P(T)$ is a finite
dimensional space of functions from $T$ onto $\mathbb{R}^{d}$, and $\Sigma _{T}$ is the set of continuous
linear forms defined on $\mathcal{D}(\mathbb{R}^{d})^{L}\oplus P(T)$ with support contained in $T$, $L$ being an integer $\geq 1$.

\item[(2)] $P(T)=\{p(x),\ x\in T:p(x)=\widehat{p}\circ F_{T}^{-1}(x),\ \widehat{%
p}\in \widehat{P}_{m}(\widehat{T})^{L}\}$, where $\widehat{P}_{m}(\widehat{T}%
)$ is the set of polynomials of degree $\leq m$ defined on $\widehat{T}$.

\item[(3)] The set $\Sigma _{T}$ is $P(T)$-unisolvent; i.e., $\forall \mu
\in \Sigma _{T}$ there is one and only one $p\in P_{T}$ such that $\mu (p)=1$
and $\forall \mu ^{\prime }\in \Sigma _{T},\ \mu ^{\prime }\neq \mu ,\ \mu
^{\prime }(p)=0$.

\item[(4)] $P(T)\subset C^{m+1}(T)^{L}.$
\end{description}
\end{definition}

Note that if all the boundaries of $T$ are straight $d-$faces, then all the
statements of Definition \ref{definition1} are valid with the only
modification that $F_{T}:\ \widehat{T}\rightarrow T$ is now an invertible
affine mapping.

Next, we summarize in the following definition some features of the family
of partitions $(\mathcal{T}_{h})_{h}$ generated on $\overline{\Omega }$.

\begin{definition}
\label{definition2.5}

\begin{description}
\item[(1)] A family of partitions $(\mathcal{T}_{h})_{h}$ composed of curved
$d-$simplices is said to be quasi-uniformly regular if there exist constants
$\sigma ,\ \nu $ and $c$ such that for each $h$ and $T$ in $\mathcal{T}_{h}$
\begin{equation*}
\frac{h_{\overline{T}}}{\rho _{\overline{T}}}\leq \sigma ,\ \ \frac{h}{h_{%
\overline{T}}}\leq \nu \text{ \ and \ }\sup_{h}\sup_{T\in D_{h}}c_{T}\leq
c<1.
\end{equation*}

\item[(2)] The family $(\mathcal{T}_{h})_{h}$ is said to be quasi-uniformly
regular of order $m$, if it is quasi-uniformly regular and for each $h,$ any
element $T$ in $\mathcal{T}_{h}$ is of class $C^{m+1}$ with (see (\ref{ft6}))%
\begin{equation*}
\sup_{h}\sup_{T\in \mathcal{T}_{h}}c_{l}(T)<+\infty \ ,\ \ \ 2\leq l\leq m+1.
\end{equation*}

\item[(3)] The partition $\mathcal{T}_{h}$ is said to be conformal if for
any two elements $T_{l}$, $T_{k}$ of $\mathcal{T}_{h}$ such that $T_{l}\cap
T_{k}=\Gamma ^{lk}$, where $\Gamma ^{lk}$ is a $(d-1)-$face,
there is a face $\widehat{\Gamma }$ of $\widehat{T}$ such that $\Gamma
^{lk}=F_{T_{l}}(\widehat{\Gamma })=F_{T_{k}}(\widehat{\Gamma })$ and $\left.
F_{T_{l}}\right\vert _{\widehat{\Gamma }}=\left. F_{T_{k}}\right\vert _{%
\widehat{\Gamma }}$.
\end{description}
\end{definition}

Given a partition $\mathcal{T}_{h}$ composed of curved $d-$simplices, in
order to define an associated finite element space $V_{h}$ we shall take
into account the set of linear forms
\begin{equation*}
\Sigma _{h}=\bigcup_{T\in \mathcal{T}_{h}}\Sigma _{T},
\end{equation*}%
and extract from $\Sigma _{h}$ a maximal system of linearly independent
continuous forms $\{\mu _{i}\}_{i=1}^{M}$. Then, for $1\leq $ $i\leq M$, let
\begin{equation*}
\Delta _{i}=\bigcup_{T\in \mathcal{T}_{h}}\left\{ T:\sup p(\mu _{i})\subset
T\right\} .
\end{equation*}%
If the family $(\mathcal{T}_{h})_{h}$ is regular, we have that for any two
elements $T$ and $T^{\prime }$ contained in $\Delta _{i}$ there is a
constant $C$ such that $h_{\overline{T}}\leq Ch_{\overline{T}^{\prime }}$.
By virtue of Definition \ref{definition2}, for each $i,\ 1\leq i\leq M$,
there exists a unique function $\chi _{i}$ such that $\forall T\in \mathcal{T%
}_{h}$,$\ \chi _{i}\mid _{T}\in P(T)$ and $\mu _{j}(\chi _{i})=\delta _{ji}$%
. The set $\{\chi _{i}\}_{i=1}^{M}$ is the set of global basis functions for
the finite element space $V_{h}$, i.e.,%
\begin{equation*}
V_{h}=Span\{\chi _{i},\ 1\leq i\leq M\}.
\end{equation*}

\begin{definition}
\label{definition3} (\cite{Bern}, Definition 3.3) The curved finite elements
of order\textit{\ }$m$
\textit{\ }$(T,P(T),\Sigma _{T}),\ T\in \mathcal{T}_{h}$\textit{,}
associated with the family of partitions $(\mathcal{T}_{h})_{h}$\ are said
to be compatible if:

\begin{description}
\item[(1)] The dimensions of $P(T)$\ are bounded independently of $h$.

\item[(2)] Let $\Delta _{i}$\ denote the support of the basis function $\chi
_{i}$ in the partition $\mathcal{T}_{h}$ , i.e., $\Delta _{i}=\bigcup_{T\in
\mathcal{T}_{h}}\{T:supp(\chi _{i})\cap T\neq \emptyset \}$.\ Then, for all $%
i,1\leq i\leq M$ ,\textit{\ }either $\Delta _{i}$\ is contained in an
element $T$\ or $\mu _{i}$\ can be extended to a continuous linear form on $%
C(\Delta _{i}).$

\item[(3)] For all $i,\ 1\leq i\leq M$, there exists a constant $K$\
independent of $h$\ such that
\begin{equation*}
\forall \rho \in P(T),\ \mid \mu _{i}(\rho )\mid \leq
Kh_{T}^{r_{i}-d/p_{i}+s_{i}}\parallel \rho \parallel _{W^{r_{i},p_{i}}(T)},
\end{equation*}%
where $r_{i}$\ is an integer\textit{\ }$\geq 0$\textit{, }$p_{i}\in \lbrack
1,+\infty ],$\ and $s_{i}$\ is a real number.\ Moreover,\ $\forall k,\ 0\leq
k\leq m+1\ $and $\forall q\in \lbrack 1,+\infty ]$%
\begin{equation*}
\parallel \chi _{i}\parallel _{k,q,T}\leq Kh_{T}^{-k+d/q-s_{i}},
\end{equation*}%
for any element $T$\ contained in $\Delta _{i}.$
\end{description}
\end{definition}

Let $W(\Omega )$ be a Banach space of functions defined on $\Omega
\rightarrow \mathbb{R}^{L}\ $(actually, $W(\Omega )\ $\ is a Banach space
associated with a sheaf-$W$ of functions defined on the domains in $\mathbb{R%
}^{d}$, see \cite{Bern} for further information).

\begin{definition}
\label{definition4} (\cite{Bern}, Definition 5.1) A finite element $%
(T,P(T),\Sigma _{T})$\ is said to be $W$-conformable if:

\begin{description}
\item[(1)] $P(T)\subset W(T)$\textit{. }

\item[(2)] For each face\ $\Gamma ^{j}\ (1\leq j\leq d+1)$\ of $T$, there
exists a subset\ $\Sigma _{T,\Gamma ^{j}}$\ of $\Sigma _{T}$\ such that:%
\textit{\ }

\begin{description}
\item[(i)] For any $u\in \{\mathcal{D}(\mathbf{%
\mathbb{R}
}^{d})^{L}\oplus P(T)\}\cap W(T)$\ and for any $\mu $\ in $\Sigma _{T,\Gamma
^{j}}$, $\mu (u)$\ depends only on $L_{T}(u)\mid _{\Gamma ^{j}}$, where $%
L_{T}$\ is a trace operator defined on\textit{\ }$W(T)$;

\item[(ii)] For any $p\in P(T),\ \{\forall \mu \in \Sigma _{T,\Gamma ^{j}},\
\mu (p)=0\}\Rightarrow \{L_{T}(p)\mid _{\Gamma ^{j}}=0\}$.
\end{description}
\end{description}
\end{definition}

For straight $d-$simplices this definition means that the triplet $(\Gamma
^{l},\ L_{T}(P_{T})\mid _{\Gamma ^{l}},\ \Sigma _{T,\Gamma ^{l}})$ is a
finite element.

We also need the definition of $W$-conforming finite elements.

\begin{definition}
\label{definition5} (\cite{Bern}, Definition 5.2) The $W$-conformable finite
elements $(T,P(T),\Sigma _{T})$\textit{, }$T\in \mathcal{T}_{h}$, associated
with the family of partitions $(\mathcal{T}_{h})_{h}$\textit{\ }are said to
be $W$-conforming, if for each $h$\ and for any triangles $T$\ and $%
T^{\prime },$\textit{\ }such that $T\cap T^{\prime }=\Gamma $,\ the spaces $%
L_{T}(P(T))\mid _{\Gamma }$\ and $T_{\varepsilon }L_{T^{\prime
}}(P(T^{\prime }))\mid _{\Gamma }$\textit{\ }are the same, and the sets $%
\Sigma _{T,\Gamma }$\ and $\Sigma _{T^{\prime },\Gamma }$\ span the same
subspace of $D^{\prime }(\mathbb{R}^{2})^{L}$. Here, $T_{\varepsilon }$ is
the involution operator in $\mathbf{%
\mathbb{R}
}^{s}$ : $(\xi )=(\xi _{i})_{1\leq i\leq s}\rightarrow \mathbf{(}\varepsilon
_{i}\xi _{i})_{1\leq i\leq s}$,\ with the k-tuple\ $\varepsilon
=(\varepsilon _{i})_{1\leq i\leq k}\in \lbrack -1,1]^{s}$.
\end{definition}

We are now in a position to introduce \ the finite element spaces associated
with the family of partitions $(\mathcal{T}_{h})_{h}$.
\begin{equation}
V_{h}:=\{v_{h}\in C(\overline{\Omega }):\ v_{h}\mid _{T}\in P(T)\ \forall
T\in \mathcal{T}_{h}\},  \label{finitespac}
\end{equation}%
with
\begin{equation}
P(T):=\{p(x):p(x)=\widehat{p}\circ F_{T}^{-1}(x),\ \widehat{p}\in \widehat{P}%
_{m}(\widehat{T})\}\text{, }  \label{finitespac1}
\end{equation}%
and $F_{T}:\widehat{T}\rightarrow T$ is such that if $T$ is a curved
element,\ then $F_{T}$ is a mapping of class $C^{m+1}$ defined by (\ref%
{curvemapping}). We must note that since the finite elements $(T,P(T),\Sigma
_{T})$ are $H^{1}$-conforming, then $V_{h}\subset H^{1}(\Omega )$ as shown
in Proposition 5.1 of \cite{Bern}. To state the approximation properties of
finite element spaces with $W$-conforming curved elements of order $m$ we
shall use Lemma \ref{lemma7} and follow the methodology of Chapters 3 and 4
of \cite{Ci}. Thus, let $T$ be a curved element of class $m$ and let $v$ be
a function defined on $T$ such that $v=\widehat{v}\circ F_{T}^{-1}$, where
the function $\widehat{v}$ is defined on the reference element $\widehat{T}$
and $F_{T}$ is the mapping of class $C^{m+1}$ defined by (\ref{curvemapping}%
) and (\ref{curvemapping2}), then one can estimate $v-\Pi _{T}v$, $\Pi _{T}$
being an interpolation operator. We have the following theorem.

\begin{theorem}
\label{theorem3}Let there be given the following items:

\begin{description}
\item[(1)] a curved finite element of order $m$\textit{\ }$(T,P(T),\Sigma
_{T})$\ and the reference finite element $(\widehat{T},\ \widehat{P}_{k}(%
\widehat{T}),\ \widehat{\Sigma }_{\widehat{T}})$;

\item[(2)] an integer $l\geq 0$\ and two numbers $p,\ q\in \lbrack 1,\infty
] $\ such that the following inclusions hold:
\begin{equation*}
W^{m+1,p}(\widehat{T})\hookrightarrow C(\widehat{T})\text{ and }W^{m+1,p}(%
\widehat{T})\hookrightarrow W^{l,q}(\widehat{T});\
\end{equation*}

\item[(3)] the interpolation operator $\widehat{\Pi }_{\widehat{T}}\in
\mathcal{L}(W^{m+1,p},W^{l,q})$\ satisfying $\forall \widehat{p}\in \widehat{%
P}_{k}(\widehat{T})$\textit{,\ }$\widehat{\Pi }_{\widehat{T}}\widehat{p}=%
\widehat{p}$\textit{;}

\item[(4)] the interpolation operator $\Pi _{T}:C^{0}(T)\rightarrow \mathbf{%
\mathbb{R}
}^{L},$\ such that for $v\in dom\Pi _{T}$\ and $\widehat{v}\in dom\widehat{%
\Pi }_{\widehat{T}}$, with $v=\widehat{v}\circ (F_{T})^{-1}$\textit{, }$(\Pi
_{T}v)^{\widehat{}}=\widehat{\Pi }_{\widehat{T}}\widehat{v}.$
\end{description}

Then, there exists a constant $C(\widehat{\Pi }_{\widehat{T}},\widehat{T})$\
such that $\forall v\in W^{m+1,p}(T)$
\begin{equation}
\mid v-\Pi _{T}v\mid _{W^{l,q}(T)}\leq Ch_{\overline{T}}^{d(\frac{1}{q}-%
\frac{1}{p})+m+1-l}\left( \sum_{r=0}^{m+1}\mid v\mid _{W^{r,p}(T)}\right) .
\label{approx2}
\end{equation}
\end{theorem}

From Corollary 4.1 of \cite{Bern}, see also Theorem 4.28 of \cite{ER}, we
can obtain an estimate of the interpolation error for the finite element
spaces $V_{h}$ associated with an exact conforming quasi-uniformly regular
family of partitions $(\mathcal{T}_{h})_{h}$.

\bigskip

\begin{theorem}
\label{corollary1} Let $V_{h}$\textit{\ }be a family of finite element
spaces associated with a complete regular family of partitions\textit{\ }$(%
\mathcal{T}_{h})_{h}$\textit{. }$V_{h}$\ is made up of compatible $H^{1}$%
-conforming curved finite elements of order $m$. Let $\Pi _{h}:C(\mathbb{R}%
^{d})\rightarrow V_{h}$\textit{\ }be an interpolation operator such that $%
\Pi _{h}v\mid _{T}=\Pi _{T}v$\ and for all $\mu _{i}\in \Sigma _{T}\ $%
\textit{\ }$\mu _{i}(\Pi _{h}v)=v(a_{i})$, where $a_{i}$\ are the mesh
points on $T$\textit{.\ }Then, for $\ l$ and $k$ integers with $l\geq 0$, $%
0\leq k\leq m$, $p,\ q\in \lbrack 1,\infty ]$,$\ v\in W^{m+1,p}(\Omega )$
and $v-\Pi _{h}v\in W^{l,q}(\Omega )$ with $W^{k+1,p}(\Omega
)\hookrightarrow W^{l,q}(\Omega )$,\ it holds
\begin{equation}
\left( \sum_{T\in \mathcal{T}_{h}}\left\vert v-\Pi _{h}v\right\vert
_{W^{l,q}(T)}^{p}\right) ^{1/p}\leq Ch^{d(\frac{1}{q}-\frac{1}{p}%
)+k+1-l}\left\Vert v\right\Vert _{W^{k+1,p}(\Omega ),}  \label{approx3}
\end{equation}%
where $0<h=\max_{T}\{h_{\overline{T}}\}\leq 1$\textit{.}\
\end{theorem}

\end{document}